\documentclass[MSNbibl,number,citesort,seceqn,dvips]{arxbj}

\aid{0}
\volume{20}
\issue{3}
\pubyear{2014}
\firstpage{1165}
\lastpage{1209}
\doi{10.3150/13-BEJ518} 

\makeatletter
\newcommand{\rrvert}{\vert}
\newcommand{\llvert}{\vert}
\newtheorem{thrm}{Theorem}[section]

\newremark{remk}{Remark}[section]

\newtheorem{lmma}{Lemma}[section]
\newtheorem{prop}{Proposition}[section]

\def\msf{\mathcal{F}}

\def\me{\mathbb{E}}
\def\mr{\mathbb{R}}
\def\mn{\mathbb{N}}
\def\mp{\mathbb{P}}

\def\bbq{\mathbb{Q}}
\def\calD{\mathcal{D}}
\def\calH{\mathcal{H}}
\def\calI{\mathcal{I}}
\def\calJ{\mathcal{J}}
\def\calK{\mathcal{K}}
\def\calR{\mathcal{R}}
\def\calS{\mathcal{S}}
\def\bbr{\mathbb{R}}
\def\bbe{\mathbb{E}} 
\def\bbp{\mathbb{P}}

\newcommand{\R}{\mathbb{R}}

\def\ed{\,{\stackrel{\frak{D}}{=}}\,}

\makeatother

\begin{document}
\begin{frontmatter}

\title{Small-time expansions for local {jump-diffusion} models with
infinite jump~activity}
\runtitle{Small-time expansions for local jump-diffusions}

\begin{aug}
\author[a]{\inits{J.E.}\fnms{Jos\'e E.} \snm{Figueroa-L\'opez}\corref{}\thanksref{a}\ead[label=e1]{figueroa@purdue.edu}},
\author[c]{\inits{Y.}\fnms{Yankeng} \snm{Luo}\thanksref{c}\ead[label=e2]{luo7@purdue.edu}}
\and
\author[b]{\inits{C.}\fnms{Cheng}~\snm{Ouyang}\thanksref{b}\ead[label=e3]{couyang@math.uic.edu}}
\runauthor{J.E. Figueroa-L\'opez, Y. Luo and C. Ouyang} 
\address[a]{Department of Statistics, Purdue University, 250 N.
University Street, West Lafayette, IN 47907, USA.
\printead{e1}}

\address[c]{Department of Mathematics, Purdue University, 250 N.
University Street, West Lafayette, IN 47907, USA.
\printead{e2}}

\address[b]{Department of Mathematics, Statistics, and Computer
Science, University of Illinois at Chicago, Chicago, IL 60607, USA.
\printead{e3}}

\end{aug}

\received{\smonth{8} \syear{2012}}
\revised{\smonth{1} \syear{2013}}

%
\begin{abstract}
We consider a Markov process {$X$}, which is the solution of a
stochastic differential equation driven by a L\'evy process $Z$ and an
independent Wiener process $W$. Under some regularity conditions,
including non-degeneracy of the diffusive and jump components of the
process as well as smoothness of the L\'evy density of $Z$ {outside any
neighborhood of the origin}, we obtain a small-time second-order
polynomial expansion for the tail distribution and the transition
density of the process $X$. Our method of proof combines a recent
{regularizing technique for deriving the analog small-time expansions
for a L\'evy process with some new tail and density estimates for
jump-diffusion processes with small jumps based on the theory of
Malliavin calculus, flow of diffeomorphisms for SDEs, and
time-reversibility}. As an application, the leading term for
out-of-the-money option prices in short maturity under a local
jump-diffusion model is also derived.
\end{abstract}

%
\begin{keyword}
\kwd{local jump-diffusion models}
\kwd{option pricing}
\kwd{small-time asymptotic expansion}
\kwd{transition densities}
\kwd{transition distributions}
\end{keyword}

\end{frontmatter}

\section{Introduction}\label{sec1}

The {small-time asymptotic} behavior {of} the transition densities of
Markov processes $\{X_{t}(x)\}_{t\geq{}0}$ with deterministic initial
condition $X_{0}(x)=x$ has been {studied} for a long time, {with a
certain focus to consider} either purely-continuous or
purely-discontinuous processes. Starting from the problem of existence,
there are several sets of sufficient conditions for the existence of
the transition density {of $X_{t}(x)$}, {denoted hereafter $p_{t}(\cdot
;x)$}. {A~stream in this direction} is based on the machinery of
Malliavin calculus, originally developed for continuous diffusions (see
the monograph Nualart \cite{Nualart}) and, then, extended to Markov process
with jumps (see {the monograph} Bichteler, Gravereaux and Jacod \cite
{BBG}). This approach can also
yield estimates of the transition density $p_{t}(\cdot;x)$ in small
time $t$. For purely-jump Markov processes, the key assumption is that
the L\'evy measure of the process admits a smooth L\'evy density. The
pioneer of this approach was L\'eandre \cite{Leandre}, who obtained the
first-order small-time asymptotic behavior of the transition density
for fully supported L\'evy densities. This result was extended in
Ishikawa \cite
{Ishikawa} to the case where the point $y$ cannot be reached with only
one jump from $x$ but rather with finitely many jumps, {while Picard
\cite
{Picard} developed a method that} can {also} {be applied to L\'evy
measures with a non-zero singular component} {(see also Picard \cite{Picardb}
and Ishikawa \cite{Ishikawa2001} for other related results)}.

The main result in L\'eandre \cite{Leandre} {states} that, {for $y\neq{}0$,}
\[
{\lim_{t\to{}0}\frac{1}{t} p_{t}(x+y;x)=g(x;y)},
\]
where ${g(x;y)}$ is the {so-called} L\'evy density of the process {$X$}
to be defined below (see~(\ref{LvyDsty0})). L\'eandre's approach
{consisted of first} separating the small jumps (say, those with sizes
smaller than an $\varepsilon>0$) and the large jumps of the underlying
L\'evy process, and {then conditioning} on the number of large jumps by
time $t$. Malliavin's calculus {was} then applied to control the
{resulting} density given that there is no large jump. For $\varepsilon
>0$ small enough, the term when there is only one large jump {was}
proved to be equivalent, up to a remainder of order $\mathrm{o}(t)$, to the
{term} resulting {from a model in which there is} no small-jump
component at all. Finally, the terms when there is more than one large
jump {were} shown to be {of order} $\mathrm{O}(t^{2})$.

Higher-order expansions of the transition density of Markov processes
with jumps have been considered quite recently and only for processes
with finite jump activity (see, e.g., Yu \cite{Yu07}) or for L\'evy
processes with possibly infinite {jump-activity}. We focus on the
literature of the latter case due to its close connection to the
present work. {R\"uschendorf and Woerner \cite{Ruschendorf} was the
first work to consider}
higher-order expansions for {the transition densities of} L\'evy
processes using L\'eandre's approach. Concretely, the following
expansion for the transition densities $\{p_{t}(y)\}_{t\geq{}0}$ of a
L\'evy process $\{Z_{t}\}_{t\geq{}0}$ was proposed {therein}:
%
%
\begin{equation}
\label{WoernerClaim} p_{t}(y) :=\frac{\mathrm{d}}{\mathrm{d}y}\bbp(Z_{t}
\leq{}y)=\sum_{n=1}^{N-1} a_{n}(y)
\frac{t^{n}}{n!} +\mathrm{O}\bigl(t^{N}\bigr)\qquad (y\neq0, N\in\mathbb{N}).
\end{equation}
As in L\'eandre \cite{Leandre}, the idea was to justify that each
higher-order
term (say, the term corresponding to $k$ large jumps) can be replaced,
up to a remainder of order $\mathrm{O}(t^{N})$, by the resulting density as if
there were no small-jump component.
However, this {approach is able to produce the correct expressions for
the higher-order coefficients $a_{2}(y),\ldots$ only in the compound
Poisson case (cf. Figueroa-L\'opez and
Houdr\'e \cite{FigHou2008})}. The {problem} was subsequently
{resolved} in Figueroa-L\'opez, Gong and
Houdr\'e {\cite{FigGongHou2011} (see Section~6 therein as well as
Figueroa-L\'opez and
Houdr\'e \cite{FigHou2008} for a preliminary related result), using a new
approach, under} the assumption that the L\'evy density of the L\'evy
process $\{Z_{t}\}_{t\geq{}0}$ is sufficiently smooth {and bounded
outside any neighborhood of the origin.
There are two key ideas in Figueroa-L\'opez, Gong and
Houdr\'e \cite{FigGongHou2011}, Figueroa-L\'opez and
Houdr\'e \cite{FigHou2008}}.
{Firstly}, instead of working directly with the transition densities,
the following analog expansions for the tail probabilities were first obtained:
%
%
\begin{equation}
\label{WoernerClaim2} \bbp(Z_{t}\geq{}y)=\sum
_{n=1}^{N-1} {A}_{n}(y) \frac{t^{n}}{n!}
+t^{N}\calR_{t}(y)\qquad (y>0, N\in\mathbb{N}),
\end{equation}
where $\sup_{0<t\leq{}t_{0}}\llvert \calR_{t}(y)\rrvert <\infty$,
for some
$t_{0}>0$. {Secondly}, by considering a smooth thresholding of the
large jumps (so that the density of large jumps {is} smooth) and
conditioning on the size of the first jump, it was possible to
{regularize} the discontinuous functional $\mathbf{1}_{\{Z_{t}\geq{}x\}}$
and, subsequently, proceed to use an iterated Dynkin's formula {(see
Section~\ref{DynkSect} below for more information)} to expand the
resulting smooth moment functions $\mathbb{E}(f(Z_{t}))$ as a power
series in $t$. Equation (\ref{WoernerClaim}) was then obtained by
differentiation of~(\ref{WoernerClaim2}), after justifying that the
functions ${A}_{n}(y)$ and the remainder $\calR_{t}(y)$ were
differentiable in $y$.

The results {and techniques} described in the previous paragraph open
the door to the study of higher-order expansions for the transition
densities of more general Markov models with infinite {jump-activity}.
We take the analysis one step further and consider a jump-diffusion
model with {non-degenerate} diffusion and jump components. Our analysis
can also be applied to purely-discontinuous processes as in L\'eandre
\cite
{Leandre}, but we prefer to {consider} a ``mixture model'' due to its
relevance in financial applications where {empirical evidence supports
models containing} both continuous and jump components (see Section~\ref{FirstOptPriceSect} below for detailed references in this direction).
More concretely, we consider the following stochastic differential
equations (SDE) driven by a Wiener process $\{W_{t}\}_{t\geq{}0}$ and
an independent pure-jump L\'evy process $\{Z_{t}\}_{t\geq{}0}$:
%
%
\begin{eqnarray}
\label{SE0} X_{t}(x)&=&x+{\int_{0}^{t}
b\bigl(X_{u}(x)\bigr)\,\mathrm{d}u}+ \int_{0}^{t}
\sigma \bigl(X_{u}(x) \bigr)\,\mathrm{d}W_{u}
\nonumber
\\[-8pt]
\\[-8pt]
\nonumber
& &{}+\sum_{u\in(0,t]\dvt |\Delta Z_{u}|\geq{}1}\gamma \bigl(X_{u^{-}}(x),
\Delta Z_{u} \bigr)+\sum^{c}_{u\in(0,t]\dvt 0<|\Delta
Z_{u}|\leq{}1}
\gamma \bigl(X_{u^{-}}(x),\Delta Z_{u} \bigr).
\end{eqnarray}
Here, $\Delta Z_{u}:=Z_{u}-Z_{u^{-}}:=Z_{u}-\lim_{s\nearrow t} Z_{s}$
denotes the jump {of} $Z$ at time $u$, {while $\sum^c$ denotes
the compensated Poisson sum of the terms therein. The functions $b,
\sigma\dvtx \mr\to\mr, \gamma\dvtx \mr\times\mr\to\mr$ are some suitable
deterministic functions so that (\ref{SE0}) is well-posed.}

As it will be evident from our work, an important difficulty to deal
with the model (\ref{SE0}) arises from the more complex interplay of
the jump and continuous components. In particular, conditioning on the
first ``big jump'' of $\{X_{s}(x)\}_{s\leq{}t}$ leads us to consider the
short-time expansions of the tail probability of a SDE with random
initial value $\tilde{J}$, which creates important, albeit interesting,
subtleties. More concretely, in the case of a L\'evy process ({i.e.,}
when $b$, $\sigma$, and $\gamma$ above are state-independent),
conditioning on the first big jump naturally leads to analyzing the
small-time expansion of the tail probability $\bbp(X_{t}^{\varepsilon
}(x)+\tilde{J}\geq x+y)$, where $\{X_{s}^{\varepsilon}(x)\}$ stands for
the ``small jump'' component of $\{X_{s}(x)\}$ (see the end of Section~\ref{SectIntro} for the terminology). This task is relatively simple
to handle since the smooth density of $\tilde{J}$ ``regularizes'' the
problem. By contrast, in the general local jump-diffusion model,
conditioning on the first big jump leads to consider $\bbp
(X_{t}^{\varepsilon}(x+\tilde{J})\geq x+y)$, a problem that does not
allow a direct application of Dynkin's formula. Instead,\vadjust{\goodbreak} to obtain the
second-order expansion of the latter tail probability, we need to rely
on smooth approximations of the tail probability building on the
theoretical machinery of the flow of diffeomorphisms for SDEs and
time-reversibility.

Under {certain regularity} conditions on $b, \sigma$ and $\gamma$, as
well as the L\'evy measure $\nu$ of $Z$, we show the following second-order expansion (as $t\to{}0$) for the tail distribution of $\{X_t(x)\}
_{t\geq{}0}$:
%
%
\begin{equation}
\label{tail-intro0} \bbp\bigl(X_{t}(x)\geq{}x+y\bigr)= t
A_{1}(x;y)+ \frac{t^{2}}{2}A_{2}(x;y)+\mathrm{O}
\bigl(t^{3}\bigr)\qquad \mathrm{for}\ x\in \mr,  y>0.
\end{equation}
{The assumptions required for (\ref{tail-intro0}) include boundedness
and sufficient smoothness of the SDE's coefficients as well as
non-degeneracy conditions on $|\partial_{\zeta} \gamma(x,\zeta)|$ {and
$|1+\partial_{x} \gamma(x,\zeta)|$}.}
As in L\'eandre \cite{Leandre}, the key assumption on the L\'evy
measure $\nu$ of
$Z$ is that this admits a density $h\dvtx \bbr\setminus\{0\}\to\bbr_{+}$
that is bounded and sufficiently smooth outside any neighborhood of the
origin. In that case, the leading term $A_{1}(x;y)$ depends only on the
jump component of the process as follows
\[
A_{1}(x;y)=\nu \bigl( \bigl\{\zeta\dvt \gamma(x,\zeta)\geq{}y \bigr\}
\bigr)=\int_{\{\zeta:\gamma(x,\zeta)\geq{}y\}}h(\zeta)\,\mathrm{d}\zeta.
\]
The second-order term $A_2(x;y)$ admits a {more complex} (but explicit)
representation, {which enables us, for instance, to precisely
characterize} the effects of the drift $b$ and the diffusion $\sigma$
of the process in the likelihood of a ``large'' positive move (say, a
move of size more than $y$) during a short time period $t$ (see Remark
\ref{RemkSnsbty} below for further details).

Once the asymptotic expansion for tail distribution is obtained, we
{proceed to obtain} a second-order expansion for the transition density
function $p_{t}(y;x)$.
{As expected from taking formal differentiation of the tail expansion
(\ref{tail-intro0}) with respect to $y$}, the leading term of
{$p_{t}(x+y;x)$} is of the form {$t g(x;y)$ for $y>0$}, where
{$g(x;y)$} is the so-called L\'evy density of the process $\{X_{t}(x)\}
_{t\geq{}0}$ defined by
%
%
\begin{equation}
\label{LvyDsty0} g(x;y):= -\frac{\partial}{\partial y} \nu\bigl(\bigl\{\zeta\dvt \gamma(x,
\zeta )\geq {}y\bigr\}\bigr) \qquad (y>0),
\end{equation}
{while the second-order term takes the form $-\partial_{y}A_{2}(x;y)
t^{2}/2$.} One of the main {subtleties here arises from attempting to
control the density of $X_{t}(x)$ given that there is no ``large''
jump}. {To this end,} we generalize the result in L\'eandre \cite
{Leandre} to the
case where there is a {non-degenerate} diffusion component. Again,
Malliavin calculus is proved to be the key tool for this task.

{Let us briefly make some remarks about the practical relevance of our
results. Short-time asymptotics for the transition densities and
distributions of Markov processes are important tools in many
applications such as non-parametric estimation methods of the model
under high-frequency sampling data and numerical approximations of
functionals of the form $\Phi_{t}(x):=\bbe (\phi(X_{T}(x) ))$.
In many of these applications, a certain discretization of the
continuous-time object under study is needed and, in that case,
short-time asymptotics are important not only in developing such
discrete-time approximations but also to determine the rate of
convergence of the discrete-time proxies to their continuous-time
counterparts.\vadjust{\goodbreak}

As an instance of the applications referred to in the previous
paragraph, a problem that has received a great deal of attention in the
last few years is the study of small-time asymptotics for option prices
and implied volatilities (see, e.g., Gatheral et~al. \cite
{Gatheral2009}, Feng, Forde and
Fouque~\cite
{FengFordeFouque}, Forde and Jacquier \cite{FordeJacquier}, Berestyki,
Busca and
Florent \cite{Busca22004}, Figueroa-L\'{o}pez and Forde~\cite
{FLF11}, Roper~\cite{Rop10}, Tankov \cite{Tnkv10}, Gao and Lee \cite
{GL11}, Muhle-Karbe and Nutz \cite{Muhle}, Figueroa-L\'opez, Gong and
Houdr\'e \cite
{FigGongHou2011}). As a byproduct of the asymptotics for the tail
distributions (\ref{tail-intro0})}, we derive {here} {the leading term
of the small-time expansion for the arbitrage-free prices of
{out-of-the-money}} European call options. Specifically, let $\{S_t\}
_{t\geq0}$ be the stock price process and denote $X_t=\log S_t$ for
each $t\geq0$. {We assume that $\bbp$ is the option pricing measure
and that under this measure} the process {$\{X_t\}_{t\geq0}$} is of the
form in (\ref{SE0}). Then, we prove that
%
%
\begin{equation}
\label{EOTMOPr} \lim_{t\rightarrow0}\frac{1}{t}
\me(S_t-K)_+=\int_{-\infty}^\infty
\bigl(S_0\mathrm{e}^{\gamma(x,\zeta)}-K \bigr)_+h(\zeta)\,\mathrm{d}\zeta,
\end{equation}
{which extends the analog result for exponential L\'evy model (cf.
Roper \cite
{Rop10} and Tankov \cite{Tnkv10})}. {A related paper is Levendorskii
\cite{Levendorskii},
where (\ref{EOTMOPr}) was obtained for a wide class of multi-factor L\'
{e}vy Markov models under certain technical conditions (see Theorem 2.1
therein), including the requirement that $\lim_{t\to{}0} \mathbb{E}
(S_{t}-K)_{+}/t$ exists in the ``out-of-the-money region'' and some
stringent integrability conditions on the L\'evy density $h$.}

The paper is organized as follows. In Section~\ref{SectIntro}, we
introduced the model and the assumptions needed for our results. The
probabilistic tools, such as the iterated Dynkin's formula as well as
tail estimates for semimartingales with bounded jumps, are {presented}
in Section~\ref{PrbTlsSec}. The main results of the paper are then
{stated} in Sections~\ref{DistrTailSect} and \ref{ExpDstSect}, where
the second-order expansion for the tail distributions and the
transition densities are {obtained}, respectively. The application of
the expansion for the tail distribution to option pricing in local
jump-diffusion financial models is {presented} in Section~\ref{FirstOptPriceSect}. The proofs of our main results as well as some
preliminaries of Malliavin calculus on Wiener--Poisson spaces are
{given} in several appendices.

\section{{Setup, assumptions and notation}}\label{SectIntro}
{Throughout, $C^{\geq{}1}_{b}$ (resp., $C^{\infty}_{b}$) represents the
class of continuous (resp., bounded) functions with bounded and
continuous partial derivatives of arbitrary order $n\geq{}1$. We} let
{$Z:=\{Z_{t}\}_{t\geq{}0}$} be a {pure-jump L\'evy process with L\'evy
measure $\nu$} and $\{W_{t}\}_{t\geq{}0}$ be a Wiener process
independent of $Z$, both of which are defined on a complete probability
space $(\Omega, \msf, \mp)$, equipped with the natural filtration
$({\msf_{t}})_{t\geq{}0}$ generated by $W$ and $Z$ and augmented by all
the null sets in
${\msf}$ so that it satisfies the \textit{usual conditions} (see, e.g.,
Chapter I in Protter \cite{Protter}). {The jump measure of the process
$Z$ is
denoted by $M(\mathrm{d}u,\mathrm{d}\zeta):=\#\{u>0\dvt (u,\Delta Z_{u})\in du\times d\zeta
\}
$, where $\Delta Z_{u}:=Z_{u}-Z_{u^{-}}:=Z_{u}-\lim_{s\nearrow t}
Z_{s}$ denotes the jump $Z$ at time $u$. This is necessarily a Poisson
random measure on $\bbr_{+}\times\bbr\setminus\{0\}$ with mean measure
$\bbe M(\mathrm{d}u,\mathrm{d}\zeta)=\mathrm{d}u\nu(\mathrm{d}\zeta)$. The corresponding compensated random
measure is denoted $\bar{M}(\mathrm{d}u,\mathrm{d}\zeta):=M(\mathrm{d}u,\mathrm{d}\zeta)-\mathrm{d}u\,\nu(\mathrm{d}\zeta)$.}

As stated in the \hyperref[sec1]{Introduction}, in this paper, we consider the following
{local jump-diffusion model:
%
%
\begin{eqnarray}\label{SE1General}
\nonumber
X_{t}(x)&=&x+\int_{0}^{t} b
\bigl(X_{u}(x)\bigr)\,\mathrm{d}u+ \int_{0}^{t}
\sigma \bigl(X_{u}(x) \bigr)\,\mathrm{d}W_{u}
\\
&&{} +\int_{0}^{t}\int_{|\zeta|>1}
\gamma \bigl(X_{u^{-}}(x),\zeta \bigr)M(\mathrm{d}u,\mathrm{d}\zeta)
\\
&&{} +\int_{0}^{t}\int_{|\zeta|\leq{}1}
\gamma \bigl(X_{u^{-}}(x),\zeta \bigr) \bar{M}(\mathrm{d}u,\mathrm{d}\zeta),
\nonumber
\end{eqnarray}
where} $b, \sigma\dvtx \bbr\to\bbr$ and $\gamma\dvtx \mr\times\mr\to\mr
$ are
deterministic functions {satisfying suitable conditions under which
(\ref{SE1General}) admits a unique solution. Typical sufficient
conditions for (\ref{SE1General}) to be well-posed include linear
growth and Lipschitz continuity of the coefficients $b$, $\sigma$, and
$\gamma$ (see, e.g., Applebaum \cite{Applebaum}, Theorem 6.2.3,
Oksendal and Sulem \cite{OS}, Theorem 1.19)}.

{Below, we will make use of the following assumptions about $Z$:}
\begin{longlist}
\item[(C1)] The L\'evy measure $\nu$ of $Z$ has a $C^{\infty}(\bbr
\setminus\{0\})$ {strictly positive} density $h$ such that, for every
$\varepsilon>0$ and $n\geq{}0$,
%
%
\begin{equation}
\label{Cnd1} \sup_{|\zeta|>\varepsilon}\bigl|h^{(n)}(\zeta)\bigr|<\infty.
\end{equation}
\end{longlist}

\begin{remk}
{Condition (\ref{Cnd1})} is {actually} {needed} for the tail
probabilities of $\{X_{t}(x)\}_{t\geq{}0}$ to admit an expansion in
integer powers of time. Indeed, even in the simplest pure L\'evy case
($X_{t}(x)=Z_{t}+x$), it is possible to build examples where $\bbp
(Z_{t}\geq{}y)$ converges to $0$ at a fractional power of~$t$ in the
absence of (\ref{Cnd1})(ii) (see Marchal \cite{Marchal}).
\end{remk}

Throughout the paper, the {jump coefficient} $\gamma$ is assumed to
satisfy the following conditions:
\begin{longlist}[(C2)(a)]
\item[(C2)(a)] {$\gamma(\cdot,\cdot)\in C^{\geq{}1}_{b}(\bbr\times
\bbr)$
and $\gamma(x,0)=0$ for all $x\in\bbr$; }

\item[(C2)(b)] {There exists a constant $\delta>0$ such that $\llvert \partial_{\zeta}\gamma(x,\zeta)\rrvert \geq\delta$,
for all $x,\zeta\in\mr$.}
\end{longlist}
Both of the previous conditions were also imposed in L\'eandre \cite{Leandre}.
Note that (C2)(a) implies that, for any $\varepsilon>0$, there
exists $C_{\varepsilon}<\infty$ such that
%
%
\begin{equation}
\label{Cnd2} \sup_{x} \biggl\llvert \frac{\partial^{i} \gamma(x,\zeta)}{\partial
x^{i}}\biggr
\rrvert \leq C_{\varepsilon} |\zeta|
\end{equation}
for all $|\zeta|\leq{}\varepsilon$ and $i\geq{}0$.
Condition (C2)(b) is imposed so that, for each $x\in\bbr$, the
mapping $\zeta\to\gamma(x,\zeta)$ admits {an inverse function
$\gamma
^{-1}(x,\zeta)$ with bounded derivatives}. Note that (C2)(b)
together with the continuity of $\partial\gamma(x,\zeta)/\partial
\zeta
$ implies that the mapping $\zeta\to\gamma(x,\zeta)$ is either strictly
increasing or decreasing for all $x$.

We {will also require} the following {boundedness and non-degeneracy}
conditions:
\begin{longlist}[(C3)]
\item[(C3)] The functions {$b(x)$ and} $v(x):=\sigma^{2}(x)/2$ {belong
to} $C^{\infty}_{b}(\bbr)$.

\item[(C4)] There exists a constant $\delta>0$ such that, for all
$x,\zeta
\in\mr$,
%
%
\begin{equation}
\label{NndegncyCnd} \mathrm{(i)} \quad\biggl\llvert 1+\frac{\partial\gamma(x,\zeta)}{\partial
x}\biggr\rrvert \geq
\delta,\qquad \mathrm{(ii)}\quad \sigma(x)\geq{}\delta.
\end{equation}
\end{longlist}
%
\begin{remk}\label{JustNonDegCnd}
{Boundedness conditions of the type (C3) above are not
restrictive in practice. Indeed, on one hand, extremely large values of
$b$ and $\sigma$ will not typically make sense in a particular financial
or physical phenomenon in mind (e.g., a large volatility value $\sigma$
could hardly be justified financially). On the other hand, a stochastic
model with arbitrary (but sufficiently regular) functions $b$ and $v$
could be closely approximated by a model with $C^{\infty}_{b}$
functions $b$ and $v$. The condition} {(\ref{NndegncyCnd})(i), which was
also imposed in L\'eandre \cite{Leandre}, guarantees the a.s.
existence of a flow
$\Phi_{s,t}(x)\dvtx \bbr\to\bbr,x\to X_{s,t}(x)$ of diffeomorphisms for all
$0\leq s\leq{}t$ (cf. L\'eandre \cite{Leandre}), where here $\{
X_{s,t}(x)\}_{t\geq
{}s}$ is defined as in {(\ref{SE1General})} but with initial condition
$X_{s,s}(x)=x$.} {Finally, let us mentioned that condition {(C4)(ii)}
is used only for the density expansion, but not the tail expansion.}\vadjust{\goodbreak}
\end{remk}

As it is usually the case with L\'evy processes, we shall decompose $Z$
into a compound Poisson process and a process with bounded jumps. More
specifically, let $\phi_{\varepsilon}\in C^{\infty}(\bbr)$ be a
truncation function such that $\mathbf{1}_{|\zeta|\geq\varepsilon}\leq
\phi
_{\varepsilon}(\zeta)\leq\mathbf{1}_{|\zeta|\geq\varepsilon/2}$ and let
{$Z(\varepsilon):=\{Z_{t}(\varepsilon)\}_{t\geq{}0}$} and
{$Z'(\varepsilon):=\{Z'_{t}(\varepsilon)\}_{t\geq{}0}$} {be independent
L\'evy processes with} respective L\'evy densities
%
%
\begin{equation}
\label{EqDTrcLDsty} h_{\varepsilon}(\zeta):=\phi_{\varepsilon}(\zeta)h(\zeta)
\quad\mbox{and}\quad \bar{h}_{\varepsilon}(\zeta):=\bigl(1-\phi_{\varepsilon}(\zeta )
\bigr)h(\zeta).
\end{equation}
{Clearly, we have that
%
%
\begin{equation}
\label{FDLP} Z \ed Z'(\varepsilon)+Z(\varepsilon).
\end{equation}
The} process $Z'(\varepsilon)$, that we referred to as the small-jump
component of $Z$, is a pure-jump L\'evy process with jumps bounded by
$\varepsilon$. In contrast, the process $Z(\varepsilon)$, hereafter
referred to as the big-jump component of $Z$, is {taken to be} a
compound Poisson process with intensity of jumps $\lambda_{\varepsilon
}:=\int\phi_{\varepsilon}(\zeta) h(\zeta)\,\mathrm{d}\zeta$ and jumps {$\{
J_{i}^{\varepsilon}\}_{i\geq{}1}$} with probability density function
%
%
\begin{equation}
\label{trncDnsty} {\breve{h}}_{\varepsilon}(\zeta):=\frac{\phi_{\varepsilon}(\zeta)
h(\zeta)}{\lambda_{\varepsilon}}.
\end{equation}
Throughout the paper, {$\{\tau_{i}\}_{i\geq1}$ and $N:=\{
N_{t}^{\varepsilon}\}_{t\geq{}0}$, respectively, denote the jump arrival
times and the jump counting process of the compound Poisson process
$Z(\varepsilon)$,} and $J:=J^{\varepsilon}$ {represents a generic}
random variable with density ${\breve{h}}_{\varepsilon}(\zeta)$.

{The next result will be needed in what follows. The different
properties below follow from standard applications of the implicit
function theorem, and the required smoothness and non-degeneracy
conditions stated above. We refer the reader to Figueroa-L\'opez, Luo
and Ouyang~\cite{FLO11} for a
detailed proof.}
%
\begin{lmma}\label{LmED}
Under the conditions \textup{(C1)}, \textup{(C2)} and \textup{(C4)}, the
following statements hold:
\begin{enumerate}
\item Let $\tilde{\gamma}(z,\zeta):=\gamma(z,\zeta)+z$. Then, for each
$z\in\bbr$, the mapping $\zeta\to\tilde{\gamma}(z,\zeta)$ (equiv.
$\zeta\to\gamma(z,\zeta)$) is invertible and its inverse $\tilde
\gamma
^{-1}(z,\zeta)$ (resp., $\gamma^{-1}(z,\zeta)$) is $C^{\geq
1}_{b}(\bbr
\times\bbr)$.\vadjust{\goodbreak}

\item Both $\tilde\gamma(z,J^{\varepsilon})$ and $\gamma
(z,J^{\varepsilon})$ admit densities in $C^{\infty}_{b}(\bbr\times
\bbr
)$, denoted by $\widetilde\Gamma(\zeta;z):=\widetilde\Gamma
_{\varepsilon
}(\zeta;z)$ and $\Gamma(\zeta;z):=\Gamma_{\varepsilon}(\zeta;z)$,
respectively. Furthermore, they have the representation:
%
%
\begin{eqnarray}
\label{DstyTildeGamma} \widetilde\Gamma_{\varepsilon}(\zeta;z)&=&{
\breve{h}}_{\varepsilon
}\bigl(\tilde{\gamma}^{-1}(z,\zeta)\bigr)\biggl
\llvert \frac{\partial\gamma
}{\partial
\zeta} \bigl(z,\tilde\gamma^{-1}(z,\zeta) \bigr)
\biggr\rrvert ^{-1},
\\
\Gamma_{\varepsilon}(\zeta;z)&=&{\breve{h}}_{\varepsilon}\bigl({\gamma
}^{-1}(z,\zeta)\bigr)\biggl\llvert \frac{\partial\gamma}{\partial\zeta} \bigl(z,
\gamma^{-1}(z,\zeta) \bigr)\biggr\rrvert ^{-1}.
\label{DstyTildeGammab}
\end{eqnarray}
\item The {mappings $(z,\zeta)\to\bbp (\tilde\gamma
(z,J^{\varepsilon
})\geq{}\zeta )$ and $(z,\zeta)\to\bbp (\gamma
(z,J^{\varepsilon
})\geq{}\zeta )$} are $C^{\infty}_{b}(\bbr\times\bbr)$.
\item The mapping $z\to u:=z+\gamma(z,\zeta)$ admits {an} {inverse},
denoted hereafter $\bar{\gamma}(u,\zeta)$, that\vspace*{1pt} belongs to $C^{\geq
{}1}_{b}(\bbr\times\bbr)$.
\end{enumerate}
\end{lmma}

We finish this section with the definition of some important processes.
{Let $\widetilde{M}$ and $M':=M'_{\varepsilon}$ denote the jump measure
of the process $\widetilde{Z}:=Z(\varepsilon)+Z'(\varepsilon)$ and
$Z'(\varepsilon)$, respectively. For each $\varepsilon\in(0,1)$, we
construct a process $ \{\widetilde{X}_{s}(\varepsilon,x) \}
_{s\geq{}0}${, defined} as the solution of the SDE\looseness=-1
\begin{eqnarray*}
\widetilde{X}_{t}(\varepsilon,x)&=&x+\int
_{0}^{t} b\bigl(\widetilde {X}_{u}(
\varepsilon,x)\bigr)\,\mathrm{d}u+ \int_{0}^{t}\sigma \bigl(
\widetilde {X}_{u}(\varepsilon,x) \bigr)\,\mathrm{d}\widetilde{W}_{u}
\\
&&{} +\int_{0}^{t}\int_{|\zeta|>1}
\gamma \bigl(\widetilde {X}_{u^{-}}(\varepsilon,x),\zeta \bigr)
\widetilde{M}(\mathrm{d}u,\mathrm{d}\zeta )\\
&&{}+\int_{0}^{t}\int
_{|\zeta|\leq{}1} \gamma \bigl(\widetilde {X}_{u^{-}}(
\varepsilon,x),\zeta \bigr) \overline{\widetilde {M}}(\mathrm{d}u,\mathrm{d}\zeta),
\end{eqnarray*}\looseness=0
where $\overline{\widetilde{M}}$ is the compensated measure of
$\widetilde{M}$ and $\widetilde{W}$ is a Wiener process, which is
independent of~$\widetilde{Z}$. In terms of the {jumps of the}
processes $Z(\varepsilon)$ and $Z'(\varepsilon)$, we can express
$\widetilde{X}(\varepsilon,x)$ as
\begin{eqnarray}
\label{SE2General}
\nonumber
\widetilde{X}_{t}(\varepsilon,x)&=&x+\int
_{0}^{t} b_{\varepsilon
}\bigl(
\widetilde{X}_{u}(\varepsilon,x)\bigr)\,\mathrm{d}u+ \int_{0}^{t}
\sigma \bigl(\widetilde{X}_{u}(\varepsilon,x) \bigr)\,\mathrm{d}
\widetilde{W}_{u}
\nonumber
\\[-8pt]
\\[-8pt]
\nonumber
&&{} +\sum_{i=1}^{N_{t}^{\varepsilon}}\gamma \bigl(\widetilde
{X}_{\tau
_{i}^{-}}(\varepsilon,x),J_{i}^{\varepsilon} \bigr)+\int
_{0}^{t}\int \gamma \bigl(\widetilde{X}_{u^{-}}(x),
\zeta \bigr) \bar {M}'(\mathrm{d}u,\mathrm{d}\zeta ),
\end{eqnarray}
where $\bar{M}'$ is the compensated random measure {$\bar
{M}'(\mathrm{d}u,\mathrm{d}\zeta
):=M'(\mathrm{d}u,\mathrm{d}\zeta)-\bar{h}_{\varepsilon}(\zeta)\,\mathrm{d}u \,\mathrm{d}\zeta$} and
\[
{b_{\varepsilon}(x):=b(x)-\int_{|\zeta|\leq{}1}\gamma (x,\zeta
)h_{\varepsilon}(\zeta)\,\mathrm{d}\zeta}.
\]
Since $Z$ has the same distribution {law} as $\widetilde
{Z}:=Z(\varepsilon)+Z'(\varepsilon)$, the process $\{\widetilde
{X}_{t}(\varepsilon,x)\}_{t\geq{}0}$ has the same distribution as $\{
X_{t}(x)\}_{t\geq{}0}$. Hence, in order to obtain the short time
asymptotics of $\bbp(X_{t}(x)\geq{}x+y)$,\vadjust{\goodbreak} we {can (and will)} analyze
the behavior of $\bbp(\widetilde{X}_{t}(\varepsilon,x)\geq{}x+y)$. For
simplicity and with certain abuse of notation, we shall write from now
on $X(x)$ instead of $\widetilde{X}(\varepsilon,x)$ and $W$ instead of
$\widetilde{W}$.

Next,} we let $ \{{X}_{s}(\varepsilon,\varnothing,x) \}
_{s\geq
{}0}$ be the solution of the SDE:
%
%
\begin{eqnarray}
\label{SE2}
{X}_{s}(\varepsilon,\varnothing,x) &=&
x+{\int_{0}^{s}{b_{\varepsilon}
\bigl({X}_{u}(\varepsilon ,\varnothing ,x) \bigr)}\,\mathrm{d}u}+\int
_{0}^{s}\sigma \bigl({X}_{u}(
\varepsilon ,\varnothing ,x) \bigr)\,\mathrm{d}{W}_{u}
\nonumber
\\[-9pt]
\\[-9pt]
\nonumber
&&{} + {\int_{0}^{s} \int\gamma
\bigl({X}_{u^{-}}(\varepsilon,\varnothing ,x),\zeta \bigr)
\bar{M'}(\mathrm{d}u,\mathrm{d}\zeta)}.
\end{eqnarray}
As seeing from the representation (\ref{SE2General}), the law of the
process (\ref{SE2}) can be interpreted as the law of {$\{\widetilde
{X}_{s}(\varepsilon,x)\}_{0\leq s\leq{}t}=\{{X}_{s}(x)\}_{0\leq s\leq
{}t}$} conditioning on not having any ``big'' jumps during $[0,t]$. In
other words, denoting the law of a process $Y$ (resp., the conditional
law of $Y$ given an event~$B$) by $\mathcal{L}(Y)$ (resp., $\mathcal
{L}(Y |B)$), we have that, for each fixed $t>0$,
\[
\mathcal{L} \bigl( \bigl\{X_{s}(x) \bigr\}_{0\leq s\leq
{}t}
\vert N_{t}^{\varepsilon}=0 \bigr)=\mathcal{L} \bigl( \bigl\{
X_{s}(\varepsilon ,\varnothing,x) \bigr\}_{0\leq s\leq{}t} \bigr).
\]
Similarly, for a collection of times $0<s_{1}<\cdots<s_{n}$, let $
\{
X_{s}(\varepsilon,\{s_{1},\ldots,s_{n}\},x) \}_{s\geq{}0}$ be the
solution of the SDE:
\begin{eqnarray*}
X_{s}\bigl(\varepsilon,\{s_{1},\ldots,s_{n}\},x
\bigr)&:=& x+\int_{0}^{s} {b_{\varepsilon}
\bigl(X_{u}\bigl(\varepsilon,\{s_{1},\ldots,s_{n}
\},x\bigr)\bigr)}\,\mathrm{d}u
\\[-2pt]
& &{}+\int_{0}^{s}\sigma \bigl(X_{u}
\bigl(\varepsilon,\{s_{1},\ldots ,s_{n}\} ,x\bigr)
\bigr)\,\mathrm{d}W_{u}
\\[-2pt]
& &{}+\sum_{i\dvt s_{i}\leq{}s}\gamma\bigl( X_{s_{i}^{-}}\bigl(
\varepsilon,\{ s_{1},\ldots,s_{n}\},x\bigr),J_{i}^{\varepsilon}
\bigr)
\\[-2pt]
&&{} + {\int_{0}^{s} \int\gamma
\bigl({X}_{u^{-}}\bigl(\varepsilon,\{ s_{1},\ldots
,s_{n}\},x\bigr),\zeta \bigr) \bar{M'}(\mathrm{d}u,\mathrm{d}\zeta).}
\end{eqnarray*}
{From (\ref{SE2General})}, it {then} follows that
\[
\mathcal{L} \bigl( \bigl\{X_{s}(x) \bigr\}_{0\leq s\leq
{}t}
\vert N_{t}^{\varepsilon}=n,\tau_{1}=s_{1},
\ldots,\tau_{n}=s_{n} \bigr)= \mathcal{L} \bigl( \bigl
\{X_{s}\bigl(\varepsilon,\{s_{1},\ldots,s_{n}\} ,x
\bigr) \bigr\}_{0\leq s\leq{}t} \bigr).
\]
The previous two processes will be needed in order to implement L\'
eandre's approach in which the tail distribution $\bbp(X_{t}(x)\geq
{}x+y)$ is expanded in powers of time by conditioning on the number of
jumps of $Z(\varepsilon)$ by time $t$.\vspace*{-2pt}

\section{Probabilistic tools}\label{PrbTlsSec}\vspace*{-2pt}
Throughout, $C_{b}^{n}(I)$ (resp., $C_{b}^{n}$) denotes the class of
functions having continuous and bounded derivatives of order $0\leq
k\leq n$ in an open interval $I\subset\bbr$ (resp., in $\bbr$). Also,
$\|g\|_{\infty}=\sup_{y}|g(y)|$.\vspace*{-2pt}

\subsection{Uniform tail probability estimates}
The following general result will be important in the sequel.\vadjust{\goodbreak}

\begin{prop}\label{martingale-property} Let {$M$} be a Poisson random
measure on $\mr_+\times\mr_0$ {with mean measure $\bbe M(\mathrm{d}u,\mathrm{d}\zeta
)=\nu
(\mathrm{d}\zeta)\,\mathrm{d}t$} and {$\bar{M}$} be its compensated random measure. Let
$Y:=Y^{(x)}$ be the solution of the SDE
\begin{eqnarray*}
Y_t=x+\int_0^t
\bar{b}(Y_{s})\,\mathrm{d}s+\int_0^t \bar{
\sigma}(Y_{s})\,\mathrm{d}W_s+ \int_0^t
\int\bar{\gamma}(Y_{s-},\zeta){\bar{M}}(\mathrm{d}s,\mathrm{d}\zeta).
\end{eqnarray*}
Assume that {$\bar{b}(x)$ and $\bar{\sigma}(x)$ are uniformly bounded}
and {$\bar{\gamma}(x,\zeta)$ is such that, for a constant {$S\in
(0,\infty)$}, $\sup_{y}|\bar\gamma(y,\zeta)|\leq S(|\zeta|\wedge1)$,
for $\nu$-a.e. $\zeta$. In particular}, the jumps of $\{Y_t\}_{t\geq
{}0}$ are bounded by $S$, {and there exists a constant $k$ such that
the quadratic variation for the {martingale part of $Y$} is bounded by
$kt$ for any time $t$}.
Then there exists a constant {$C(S,k)$} depending on $S$ {and $k$},
such that, for {any fixed $p>0$ and} all $0\leq t\leq1$,
\[
\mp \Bigl\{\sup_{0\leq s\leq t}|{Y_s}-x|\geq{2} pS \Bigr\}
\leq{C(S,k)t^p}.
\]
\end{prop}
\begin{pf}
{Let
\[
{V_t}=\int_0^t \bar{
\sigma}(Y_{s})\,\mathrm{d}W_s+\int_0^t
\int\bar{\gamma }(Y_{s-},z){\bar{M}}(\mathrm{d}s,\mathrm{d}z)
\]
be the martingale part of $Y_t$. It is clear that ${V_t}$ is a
martingale with its jumps bounded by $S$. Moreover, in light of the
boundedness of $\bar\sigma$ and $\bar\gamma$, {its} quadratic variation
satisfies ${\langle V, V\rangle_t}\leq kt$, for some constant $k$. By
{equation (9) in Lepeltier and Marchal \cite{LM}}, we have
%
%
\begin{equation}
\mp \Bigl\{\sup_{0\leq s\leq t}|{V_s}|
\geq C \Bigr\} \leq{2\exp \biggl[- \lambda C+\frac{\lambda^2}{2}kt\bigl(1+\exp[
\lambda S]\bigr) \biggr]\qquad \mbox{for all } C, \lambda>0.}
\end{equation}
Now take {$C=2pS$ and $\lambda=|\log t|/2S$}, the claimed result
follows for the martingale part $V_t$ of~$Y_t$. {By {equation (9) in
Lepeltier and Marchal \cite{LM}} and the fact that the drift term is
bounded by $\|\bar{b}\|
_\infty t$, we have for all $C, \lambda>0$
\begin{eqnarray}
\label{martingale-est} \mp \Bigl\{\sup_{0\leq s\leq t}|{Y_s-x}|
\geq C \Bigr\}&\leq&\mp \Bigl\{\sup_{0\leq s\leq t}|{V_s}|
\geq C-t\|\bar{b}\|_\infty \Bigr\}
\nonumber
\\[-8pt]
\\[-8pt]
\nonumber
&\leq&{2\exp \biggl[- \lambda\bigl(C-\|\bar{b}\|_\infty t\bigr)+
\frac{\lambda
^2}{2}kt\bigl(1+\exp[\lambda S]\bigr) \biggr].}
\nonumber
\end{eqnarray}
Now take {$C=2pS$ and $\lambda=|\log t|/2S$}, the claimed result follows.}}
\end{pf}
As a direct corollary of the previous proposition, we have the
following {estimate for the tail probability of the} small-jump
component $\{X_{t}(\varepsilon,\varnothing,x)\}_{t\geq{}}$ of $X$ defined
in (\ref{SE2}). {We also provide a related estimate for the tail
probability of $\exp(\llvert X_{t}(\varepsilon,\varnothing,x)\rrvert )$,
which will be needed for the asymptotic result of option prices
discussed in Section~\ref{FirstOptPriceSect} below.}

\begin{lmma}\label{NLBS}Fix any $\eta>0$ and a positive integer $N$.
Then, under {the conditions \textup{(C2)--(C3)} of Section~\ref{SectIntro}},
there {exist} an {$\varepsilon:=\varepsilon(N,\eta)>0$} and
$C:=C(N,\eta
)<\infty$ such that
\begin{enumerate}
\item[(1)] For all $t<1$,
%
%
\begin{equation}
\label{FndBnd1} \sup_{0<\varepsilon'<\varepsilon,x\in\bbr} \bbp \bigl(\bigl|X_{t}\bigl(
\varepsilon ',\varnothing,x\bigr)-x\bigr|\geq{}\eta\bigr)<C
t^{N}.
\end{equation}
\item[(2)] For all $t<1$,
\[
\sup_{\varepsilon'< \varepsilon,x\in\mr}\int_{{\mathrm{e}^\eta}}^\infty \mp
\bigl(\mathrm{\mathrm{e}}^{|X_t(\varepsilon',\varnothing,x)-x|}\geq s \bigr)\,\mathrm{d}s <Ct^N.
\]
\end{enumerate}
\end{lmma}
\begin{pf}
The first statement is a special case of Proposition \ref
{martingale-property}, which can be applied in light of the boundedness
conditions (C3) as well as the condition (C2)(a).
To prove the second statement, we keep the notation of the proof of
Proposition \ref{martingale-property} and note that,\vadjust{\goodbreak} by (\ref
{martingale-est}), there exists a constant $C>0$ such that
\begin{eqnarray*}
\int_{{\mathrm{e}^\eta}}^\infty\mp\bigl\{{\bigl|X_t(
\varepsilon,\varnothing,x)-x\bigr|}\geq \log s\bigr\}\,\mathrm{d}s&\leq &C\int_{{\mathrm{e}^\eta}}^\infty
\exp{ \biggl[- \lambda\log s+\frac
{\lambda^2}{2}kt\bigl(1+\exp[\lambda\varepsilon]
\bigr) \biggr]}\,\mathrm{d}s
\\
&=&\frac{C{\mathrm{e}^\eta}}{(\lambda-1) {\mathrm{e}^{\lambda\eta}}}\exp{ \biggl[\frac
{\lambda^2}{2}kt\bigl(1+\exp[\lambda
\varepsilon]\bigr) \biggr].}
\end{eqnarray*}
Now it suffices to take $\lambda= |\log t|/2\varepsilon$ and
{$\varepsilon= \eta/2N$}.
\end{pf}

\subsection{Iterated Dynkin's formula}\label{DynkSect}
We now proceed to state a second-order iterated Dynkin's formula for
the ``small-jump component'' of $X$, $\{X_{t}(\varepsilon,\varnothing
,x)\}
_{t\geq{}0}$, defined in (\ref{SE2}). To this end, let us first recall
that the infinitesimal generator of $X(\varepsilon,\varnothing,x)$,
hereafter denoted by $L_{\varepsilon}$, can be written as follows
({cf. Oksendal and Sulem \cite{OS}, Theorem 1.22}):
%
%
\begin{eqnarray}
\label{InfGenSmallJumps} L_{\varepsilon} f(y)&:=& \calD_{\varepsilon}f(y)+
\calI_{\varepsilon
}f(y)\qquad\mbox{with}
\nonumber\\
\calD_{\varepsilon}f(y)&:=&\frac{\sigma
^{2}(y)}{2}f''(y)+b_{\varepsilon
}(y)
f'(y),
\\
\calI_{\varepsilon}f(y)&:=&\int \bigl(f\bigl(y+\gamma(y,\zeta )\bigr)-f(y)-
\gamma (y,\zeta)f'(y) \bigr)\bar{h}_{\varepsilon}(\zeta) \,\mathrm{d}\zeta.
\nonumber
\end{eqnarray}
{The following two alternative representations of $\calI_{\varepsilon
}f$ will be useful in the sequel:
%
%
\begin{eqnarray}
\label{UsRepIf1} \calI_{\varepsilon}f(y)&=&\int\int_{0}^{1}f''
\bigl(y+\gamma(y,\zeta )\beta \bigr) (1-\beta)\,\mathrm{d}\beta \bigl(\gamma(y,\zeta)
\bigr)^{2} \bar {h}_{\varepsilon
}(\zeta) \,\mathrm{d}\zeta
\\
\label{UsRepIf2} &=&\int\int_{0}^{1} \bigl[
f''\bigl(y+\gamma(y,\zeta\beta)\bigr) \bigl(
\partial_{\zeta}\gamma(y,\zeta\beta )\bigr)^{2}+f'
\bigl(y+\gamma(y,\zeta\beta)\bigr)\partial^{2}_{\zeta}\gamma
(y,\zeta\beta )
 \nonumber
 \\[-8pt]
 \\[-8pt]
 \nonumber
 &&\hspace*{28pt}{}- f'(y) \partial^{2}_{\zeta}\gamma(y,
\zeta\beta) \bigr](1-\beta)\,\mathrm{d}\beta \zeta^{2}\bar{h}_{\varepsilon}(
\zeta)\,\mathrm{d}\zeta.
\end{eqnarray}
In particular, from the previous representations, it is evident that
$\calI_{\varepsilon}f$} is well-defined whenever $f\in C^{2}_{b}$, in
view of (\ref{Cnd2}), which follows from our condition (C2)(a).\vadjust{\goodbreak}

The $n$-order iterated Dynkin's formula for the process $X(\varepsilon
,\varnothing,x)$ takes the generic form
%
%
\begin{equation}
\label{DynkinF} \hspace*{-10pt}\bbe f\bigl(X_{t}(\varepsilon,\varnothing,x)\bigr)=
\sum_{k=0}^{n-1}
\frac{t^{k}}{k!} L^{k}_{\varepsilon}f(x)+ \frac{t^{n}}{(n-1)!} \int
_{0}^{1} (1-\alpha)^{n-1}\bbe \bigl\{
L^{n}_{\varepsilon} f\bigl(X_{\alpha t}(\varepsilon,
\varnothing,x)\bigr) \bigr\} \,\mathrm{d}\alpha,
\end{equation}
where as usual $L_{\varepsilon}^{0}f=f$ and $L_{\varepsilon
}^{n}f=L_{\varepsilon}(L_{\varepsilon}^{n-1}f)$, $n\geq{}1$.
(\ref{DynkinF}) can be proved for $n=1$ using It\^o's formula (see
Oksendal and Sulem \cite{OS}, Theorem 1.23) while, {for a general
order $n$, it can be
proved by induction, provided that the iterated generators
$L_{\varepsilon}^{k}f$ satisfy sufficient smoothness and boundedness
conditions for any $k=0,\ldots,n$. The next lemma explicitly states the
second-order formula so that we can refer to it in the sequel. Its
proof is standard and is omitted for the sake of brevity (see
Figueroa-L\'opez, Luo and Ouyang \cite
{FLO11} for the details).}
%
\begin{lmma}\label{RemLoc}
{For a fix $\varepsilon\in(0,1)$, let $K_{\varepsilon,m}$ denote a
finite constant whose value only depends on $\int\zeta^{2} \bar
{h}_{\varepsilon}(\zeta) \,\mathrm{d}\zeta$, $\|f^{(k)}\|_{\infty}$, $\|
b^{(k)}\|
_{\infty}$, and $\|v^{(k)}\|_{\infty}$ with $k=0,\ldots,m$. Then,} under
the conditions \textup{(C1)--(C3)} of Section~\ref{SectIntro}, the
following assertions hold true:
\begin{enumerate}
\item For any function $f$ in $C^{2}_{b}$, {$\sup_{y}L_{\varepsilon
}f(y)\leq K_{\varepsilon,2}$, and the iterated Dynkin's formula (\ref
{DynkinF}) is satisfied with $n=1$.}

\item{If, additionally, $f\in C^{4}_{b}$, then $\sup_{y}L^{2}_{\varepsilon}f(y)\leq K_{\varepsilon,4}$ and, furthermore,
the iterated Dynkin's formula (\ref{DynkinF}) is satisfied with $n=2$.}
\end{enumerate}
\end{lmma}

\section{Second-order expansion for the tail distributions}\label
{DistrTailSect}
{We are ready to state our first main result; namely, we}
{characterize} the small-time behavior of the tail distribution of $\{
X_{t}(x)\}_{t\geq{}0}$:
%
%
\begin{equation}
\label{TDJD} \bar{F}_{t}(x,y):=\bbp\bigl(X_{t}(x)\geq{}x+y
\bigr)\qquad (y>0).
\end{equation}
{As in L\'eandre \cite{Leandre},} the key idea is to {take advantage
of} the
decomposition (\ref{FDLP}), by conditioning on the number of ``large''
jumps occurring before time $t$.
Concretely, {recalling that} $\{N_{t}^{\varepsilon}\}_{t\geq{}0}$ and
$\lambda_{\varepsilon}:=\int\phi_{\varepsilon}(\zeta) h(\zeta
)\,\mathrm{d}\zeta$
{represent} the jump counting process and the jump intensity of the
{large-jump component process $\{Z_{t}(\varepsilon)\}_{t\geq{}0}$ of
$Z$}, we have
%
%
\begin{equation}
\label{FndDcmp} \bbp\bigl(X_{t}(x)\geq{}x+y\bigr)=\mathrm{e}^{-\lambda_{\varepsilon}t}
\sum_{n=0}^{\infty
}\bbp\bigl(
X_{t}(x)\geq{}x+y\vert N_{t}^{\varepsilon}=n\bigr)
\frac{(\lambda_{\varepsilon}t)^{n}}{n!}.
\end{equation}
The first term in (\ref{FndDcmp}) (when $n=0$) can be written as
\[
\bbp\bigl(X_{t}(x)\geq{}x+y\vert N_{t}^{\varepsilon}=0
\bigr)=\bbp \bigl(X_{t}(\varepsilon,\varnothing,x)\geq{}x+y\bigr).
\]
In light of (\ref{FndBnd1}), this term can be made $\mathrm{O}(t^{N})$ for an
arbitrarily large $N\geq{}1$, by taking $\varepsilon$ small enough.
In order to deal with the other terms in (\ref{FndDcmp}), we use the
iterated Dynkin's formula introduced in Section~\ref{DynkSect}. The
following is the main result of this section (see Appendix \ref
{SecMnThPr} for the proof).\vadjust{\goodbreak}
Below, $h_{\varepsilon}$ and $\bar{h}_{\varepsilon}$ denote the L\'evy
densities defined in (\ref{EqDTrcLDsty}), while $g(x;y)$ denotes the
so-called L\'evy density of the process {$\{X_{t}(x)\}_{t\geq{}0}$}
defined by
%
%
\begin{equation}\label{DfnLvyDty0}
    g(x;y):=
    \cases{
    -\displaystyle\frac{\partial}{\partial y} \int_{\{\zeta:\gamma(x,\zeta)\geq y\}}
h(\zeta)\,\mathrm{d}\zeta,&\quad$y>0$,
\cr
\displaystyle\frac{\partial}{\partial y} \int_{\{\zeta:\gamma(x,\zeta)\leq y\}} h(\zeta)
\,\mathrm{d}\zeta,&\quad$y<0$.}
\end{equation}
for $y\neq 0$. In light of Lemma \ref{LmED}, $g$ admits the representation:
\[
g(x;y)= h\bigl({\gamma}^{-1}(x,y)\bigr)\bigl\llvert (
\partial_{\zeta} \gamma) \bigl(x,\gamma^{-1}(x,y) \bigr)\bigr
\rrvert ^{-1},
\]
where $\partial_{\zeta} \gamma$ is the partial derivative of the
function $\gamma(x,\zeta)$ with respect to its second variable.

\begin{thrm}\label{ThTail}
Let $x\in\bbr$ and $y>0$. Then, under {the conditions \textup{(C1)--(C4)} of
Section~\ref{SectIntro}}, we have
%
%
\begin{equation}
\bar{F}_{t}(x,y):=\bbp\bigl(X_{t}(x)\geq{}x+y
\bigr)= t A_{1}(x;y)+ \frac{t^{2}}{2}A_{2}(x;y)+\mathrm{O}
\bigl(t^{3}\bigr)
\end{equation}
as $t\to{}0$, where $A_{1}(x;y)$ and $A_{2}(x;y)$ admit the following
representations (for $\varepsilon>0$ {small enough}):
\begin{eqnarray*}
A_{1}(x;y)&:=&\int_{y}^{\infty}g(x;\zeta)\,\mathrm{d}
\zeta=\int_{\{\gamma
(x,\zeta
)\geq{}y\}}h(\zeta)\,\mathrm{d}\zeta,
\\
A_{2}(x;y)&:=& \calD(x;y)+\calJ_{1}(x;y)+
\calJ_{2}(x;y),
\end{eqnarray*}
with
\begin{eqnarray}\label{MPTN}
\nonumber
\calD(x;y)&=&{b_{\varepsilon}(x)} \biggl(\frac{\partial}{\partial
x}\int
_{y}^{\infty}g(x;\zeta)\,\mathrm{d}\zeta+g(x;y)
\biggr)+{b_{\varepsilon
}(x+y)}g(x;y)
\\
&&{} +\frac{\sigma^{2}(x)}{2} \biggl(\frac{\partial^{2}}{\partial
x^{2}}\int_{y}^{\infty}g(x;
\zeta)\,\mathrm{d}\zeta+2\frac{\partial}{\partial
x}g(x;y)-\frac{\partial}{\partial y}g(x;y) \biggr)
\nonumber
\\
&&{} -\frac{\sigma(x+y)}{2} \biggl(\sigma(x+y)\frac{\partial
}{\partial
y}g(x;y)+2
\sigma'(x+y)g(x;y) \biggr),
\nonumber
\\[-8pt]
\\[-8pt]
\nonumber
\calJ_{1}(x;y)&=& \int \biggl(\int_{y-\gamma(x,\bar{\zeta
})}^{\infty
}g
\bigl({x+\gamma(x,\bar\zeta)};\zeta\bigr)\,\mathrm{d}\zeta +\int_{\bar{\gamma}(x+y,\bar{\zeta})-x}^{\infty}g(x;
\zeta)\,\mathrm{d}\zeta -2\int_{y}^{\infty}g(x;\zeta)\,\mathrm{d}\zeta
\nonumber
\\
&&\hspace*{16pt}{} -{\gamma(x,\bar\zeta)\partial_{x}\int_{y}^{\infty}g(x;
\zeta )\,\mathrm{d}\zeta -\gamma(x,\bar\zeta)g(x;y)-\gamma(x+y,\bar\zeta)g(x;y)} \biggr)
\bar{h}_{\varepsilon}(\bar\zeta) \,\mathrm{d}\bar\zeta,
\nonumber
\\
\calJ_{2}(x;y)&=&{\int\int_{y-\gamma(x,\bar\zeta)}^{\infty
}g
\bigl(x+\gamma (x,\bar\zeta);\zeta\bigr)\,\mathrm{d}\zeta h_{\varepsilon}(\bar\zeta)\,\mathrm{d}\bar
\zeta }-2\int_{y}^{\infty}g(x;\zeta)\,\mathrm{d}\zeta\int
h_{\varepsilon}(\zeta)\,\mathrm{d}\zeta.
\nonumber
\end{eqnarray}
\end{thrm}

%
\begin{remk}
Note that if ${\rm supp}(\nu)\cap\{\zeta\dvt  \gamma(x,\zeta)\geq{}y\}
=\varnothing$ (so that it is not possible to reach the level $y$ from $x$
with only one jump),\vadjust{\goodbreak} then $A_{1}(x;y)=0$ and $\bbp(X_{t}(x)\geq
{}x+y)=\mathrm{O}(t^{2})$ as $t\to{}0$. If, in addition, it is possible to reach
the level $y$ from $x$ with two jumps, then {$\calJ_{2}(x;y)\neq{}0$},
implying that $\bbp(X_{t}(x)\geq{}x+y)$ decreases at the order of
$t^{2}$. These observations are consistent with the results in Ishikawa
\cite
{Ishikawa} and Picard \cite{Picardb}.
\end{remk}

\begin{remk}\label{RemkSnsbty}
In the case that the coefficient $\gamma(x,\zeta)$ does not depend on
$x$, we get the following expansion for $\bbp(X_{t}(x)\geq{}x+y)$:
\begin{eqnarray*}
\bbp\bigl(X_{t}(x)\geq{}x+y\bigr)&=&t\int_{y}^{\infty}
g(\zeta)\,\mathrm{d}\zeta+\frac
{{b_{\varepsilon}(x)+b_{\varepsilon}(x+y)}}{2} g(y) t^{2}
\\
&&{} - \biggl(\frac{\sigma^{2}(x)+\sigma^{2}(x+y)}{2}g'(y)+2\sigma (x+y)
\sigma'(x+y)g(y) \biggr) \frac{t^{2}}{2}
\nonumber
\\
&&{} +{\int \biggl(\int_{y-\gamma(\bar\zeta)}^{\infty}g(\zeta)\,\mathrm{d}\zeta -
\int_{y}^{\infty}g(\zeta)\,\mathrm{d}\zeta -{2g(y)\gamma(\bar
\zeta)} \biggr)\bar{h}_{\varepsilon}({\bar{\zeta}}) \,\mathrm{d}{\bar{\zeta}}
t^{2}}
\\
&&{} + { \biggl(\int\int_{y-\gamma({\bar\zeta})}^{\infty}g(\zeta )\,\mathrm{d}\zeta
h_{\varepsilon}({\bar\zeta})\,\mathrm{d}{\bar\zeta}-2\int_{y}^{\infty
}g(
\zeta)\,\mathrm{d}\zeta\int h_{\varepsilon}(\zeta)\,\mathrm{d}\zeta \biggr)\frac
{t^{2}}{2}+\mathrm{O}
\bigl(t^{3}\bigr)}.
\end{eqnarray*}
The leading term in the above expression is {determined} by the jump
component of the process and it has a natural interpretation: if within
a very short time interval there is a ``large'' positive move (say, a
move by more than $y$), this move must be due to a ``large'' jump. It is
until the second term, when the diffusion and drift terms of the
process {$X(x)$} appear. {If, for instance, $b$ and $\sigma$ are constants,
the effect of a positive {``drift''} ${b_{\varepsilon}}>0$ is to
increase the probability of a ``large'' positive move {of more than
$y$} by ${b_{\varepsilon}}g(y) t^{2}(1+\mathrm{o}(1))$. Similarly, since
typically $g'(y)<0$ when $y>0$, the effect of a non-zero spot volatility
$\sigma$ is to increase the probability of a ``large'' positive move by
$\frac{\sigma^{2}}{2}|g'(y)| t^{2}(1+\mathrm{o}(1))$}.
\end{remk}

\section{Expansion for the transition densities}\label{ExpDstSect}
Our goal here is to obtain a second-order small-time approximation for
the transition densities $\{p_{t}(\cdot;x)\}_{t\geq{}0}$ of $\{
X_{t}(x)\}_{t\geq{}0}$. As it was done in the previous section, the
idea is {to} work with the expansion (\ref{FndDcmp}) by first showing
that each term there is differentiable {with respect to $y$}, and then
determining their rates of convergence to $0$ as $t\to{}0$. One of the
main difficulties of this approach comes from controlling the term
corresponding to no ``large'' jumps. As in the case of purely diffusion
processes, Malliavin calculus is proved to be the key tool for this
task. This analysis is presented in the following subsection before our
main result is presented in Section~\ref{secExpDnsty}.

\subsection{Density estimates for SDE with bounded jumps}
In this part, we analyze the term corresponding to $N_{t}^{\varepsilon}=0$:
\[
\bbp\bigl(X_{t}(x)\geq{}x+y\vert N_{t}^{\varepsilon}=0
\bigr)=\bbp \bigl(X_{t}(\varepsilon,\varnothing,x)\geq{}x+y\bigr).\vadjust{\goodbreak}
\]
We will prove that, for any {fixed} positive integer $N$ and $\eta>0$,
there {exist an $\varepsilon_{0}>0$ and a constant $C<\infty$ (both
only depending on $N$ and $\eta$)} such that the density $p_{t}(\cdot
;{\varepsilon},\varnothing,x)$ of $X_{t}(\varepsilon,\varnothing,x)$ satisfies
%
%
\begin{equation}
\label{NETDSJ0} \sup_{|y-x|>{\eta}, {\varepsilon}<{\varepsilon}_{0}} p_{t}(y;{\varepsilon},
\varnothing,x)< C t^{N}
\end{equation}
for all $0<t\leq{}1$.

{{To} simplify notation, in this subsection, we write $X_t^x$ for
$X_t(\varepsilon, \varnothing, x)$. Recall that $X_t^x$ satisfies an
equation of the following general form}
%
%
\begin{eqnarray}
\label{SDE} X_t^x=x+\int_0^t{b_{\varepsilon}
\bigl(X_{s-}^x\bigr)}\,\mathrm{d}s+\int_0^t
\sigma \bigl(X_{s-}^x\bigr)\,\mathrm{d}W_s+\int
_0^t\int\gamma\bigl(X_{s-}^x,
\zeta\bigr){\bar {M}'}(\mathrm{d}s,\mathrm{d}\zeta),
\end{eqnarray}
{where, ${M'(\mathrm{d}s,\mathrm{d}\zeta)}$ is a Poisson random measure on $\bbr
_{+}\times\mr\setminus\{0\}$ with {mean measure} ${\mu'(\mathrm{d}s,\mathrm{d}\zeta
)}={\nu'(\mathrm{d}\zeta)\,\mathrm{d}s}={\bar{h}_{\varepsilon}(\zeta)}\,\mathrm{d}\zeta \,\mathrm{d}s$ and
{$\bar
{M'}=M'-\mu'$} is its compensated measure.
Since there are no ``big jumps'' for $X_t^x$, {$\bar{h}_{\varepsilon}$}
is supported in a ball $B(0,\varepsilon)$}.

Malliavin calculus is the main tool to analyze the existence and
smoothness of density for~$X^x_t$. {Throughout this subsection, we
follow closely the presentation of Bichteler, Gravereaux and Jacod
\cite{BBG}, Chapter IV} {(see also
Appendix A in Figueroa-L\'opez, Luo and Ouyang \cite{FLO11} for an
introduction to this theory)}. As
described therein, there are different ways to define a Malliavin
operator for {Wiener--Poisson spaces}. For our purposes, it suffices to
consider the Malliavin operator corresponding to $\rho=0$ {(see
Bichteler, Gravereaux and Jacod
\cite{BBG},
Section 9a--9c, for the details)}. The intuitive explanation of
$\rho=0$ is that when making perturbation of the sample path on the
{Wiener--Poisson} space, we only perturb the Brownian path.

{Let us start by noting that our assumption on the coefficients of
(\ref
{SDE})} ensures that ${x\to X^x_t}$ is a $C^2$-diffeomorphism {with} a
continuous density (see Bichteler, Gravereaux and Jacod \cite{BBG} for
more details). {Define}
%
%
\begin{equation}
\label{Gamma} U_t:= \Gamma\bigl({X_t^{x},X_t^{x}}
\bigr)= { \biggl\{\int_0^t\sigma ^2
\bigl(X^x_{s}\bigr){\mathbf{J}_s(x)}^{-2}\,\mathrm{d}s
\biggr\}{\mathbf{J}_t(x)}^{2}.}
\end{equation}
{In the above, we use the standard notation:}
%
%
\begin{equation}
\label{JcbDfn1} {{\mathbf{J}_t(x)}=\frac{\mathrm{d}X^x_t}{\mathrm{d}x}.}
\end{equation}

\begin{remk}\label{divX-Lpbound}
Under {the condition (C4) of Section~\ref{SectIntro}},
{{$\mathbf{J}_t(x)$} admits an inverse $Y_t:={\mathbf{J}_t(x)}^{-1}$,
almost surely}. Indeed, one can show that ({cf. Bichteler, Gravereaux
and Jacod~\cite{BBG}})
\begin{eqnarray*}
d{\mathbf{J}_t(x)} &=&1+{\partial_x} {b_{\varepsilon}
\bigl(X_{t-}^x\bigr)} {\mathbf {J}_{t-}}(x)\,\mathrm{d}t+{
\partial_x}\sigma\bigl(X^x_{t-}\bigr){\mathbf
{J}_{t-}}(x)\,\mathrm{d}{W_t}
\\
&&{}+{\partial_x}\gamma\bigl(X^x_{t-},\zeta
\bigr){\mathbf{J}_{t-}}(x){\bar {M}'}(\mathrm{d}t,\mathrm{d}\zeta),
\end{eqnarray*}
{while $Y_t={\mathbf{J}_t(x)}^{-1}$} satisfies {an equation of the form:}
\[
dY_t =1+Y_{t-}D_t\,\mathrm{d}t+Y_{t-}E_t\,d{W_t}+Y_{t-}F_t{
\bar{M}'}(\mathrm{d}t,\mathrm{d}\zeta).
\]
Here $D_t, E_t$ and $F_t$ are determined by ${b_{\varepsilon}(x)},
\sigma(x), \gamma(x,\zeta)$ and $X_t^x$. As a consequence, together
with our assumption on $b, \sigma$ and
$\gamma$, one {has}
\[
\me\sup_{0\leq t\leq1}{{\mathbf{J}_t(x)}^p}\quad  \mathrm{and} \quad\me \sup_{0\leq t\leq1}{\mathbf{J}_t(x)}^{-p}<
\infty
\]
for all $p>1$.
\end{remk}

The main result of this section is Theorem \ref{densityest-main}
{below}. For this purpose, we state {some preliminary known results.
Let us start with the} following integration by parts formula ({the
main ingredient for {existence} and smoothness of {the} density of
$X^x_t$}), which is a special case of Lemma~4--14 {in Bichteler,
Gravereaux and Jacod \cite{BBG}}
together with the discussion of Chapter IV {therein}.

\begin{prop}[{(Integration by parts)}]\label{IBP}
For any $f\in C^\infty_c(\mr)$, there exists a random variable
${G_t}\in L^p$ for all $p\in\mn$, such that
\[
\bbe {\partial_{x}} f\bigl(X^x_t\bigr)=\me
{G_t} U^{-2}_t f\bigl(X_t^x
\bigr).
\]
\end{prop}

The following existence and regularity result for the density of a
finite measure is well known (see, e.g., Theorem 5.3 in Shigekawa \cite
{Shige}).
%
\begin{prop}\label{reg-density}
Let $m$ be a finite measure supported in an open set $\mathrm{O}\subset{\mr}$.
Take any $p>{1}$. Suppose that there exists ${g}\in\mathbb{L}^p(m)$
such that
\[
\int_{\mr}{\partial_xf} \,\mathrm{d}m=\int
_{\mr}f {g} \,\mathrm{d}m,\qquad f\in C^\infty_c(\mathrm{O}).
\]
Then $m$ has a bounded density function ${q}\in C_b(\mathrm{O})$ satisfying
\[
\|{q}\|_\infty\leq C{\|g\|_{\mathbb{L}^p(m)}m(\mathrm{O})^{1-1/p}}.
\]
Here the constant $C$ {depends on $p$}.
\end{prop}

{The following lemma is the main ingredient in proving Theorem \ref
{densityest-main}.}
%
\begin{lmma}\label{Malliavin-singularity} Recall $U_t=\Gamma(X_t,X_t)$.
Under {the condition \textup{(C4)} of Section~\ref{SectIntro}}, we have
\[
\me U_t^{-p}\leq Ct^{-p},
\]
{for all $p>1$}.
\end{lmma}
\begin{pf}
The proof is a direct consequence of assumption (C4) and Remark
\ref{divX-Lpbound}. More precisely,
{
\begin{eqnarray*}
\me U_t^{-p}&=&\me\frac{{\mathbf{J}_t(x)}^{-2p}}{ (\int_0^t
{\mathbf
{J}_s(x)}^{-2}\sigma(X_s^x)^2\,\mathrm{d}s )^{p}}\leq\frac{1}{t^p}
\me \frac
{{\mathbf{J}_t(x)}^{-2p}}{\delta^{2p} \inf_{0\leq s\leq t} {\mathbf
{J}_s(x)}^{-2p}}
\\
&=&\frac{1}{t^p}\delta^{-2p}\me \Bigl({\mathbf{J}_t(x)}^{-2p}{
\sup_{0\leq s\leq1} {\mathbf{J}_s(x)}^{2p}}
\Bigr).
\end{eqnarray*}
}
The proof is completed.
\end{pf}

\begin{remk}
The above lemma is where condition (C4)(ii) is used. It could be
relaxed to include degenerate diffusion coefficients. But in the
degenerate case, we need to take a non-trivial $\rho$ (as opposed to
$\rho=0$ in the present setting) in the construction of Malliavin
operator on the Wiener--Poisson space. In this case, the process $U_t$ becomes
\begin{eqnarray*}
U_t&:=& {\mathbf{J}_t(x)}^{2}\int
_0^t\sigma ^2\bigl(X^x_{s}
\bigr){\mathbf{J}_s(x)^{-2}}\,\mathrm{d}s
\\
&&{} +{\mathbf{J}_t(x)}^{2}\int_\mr
\int_0^t\mathbf {J}_{s-}(x)^{-2}
\bigl(1+\partial_x\gamma\bigl(X^x_{s-},\zeta
\bigr) \bigr)^2\bigl(\partial_\zeta\gamma
\bigl(X_{s-}^x,\zeta\bigr)\bigr)^2\rho(
\zeta){M'(\mathrm{d}s,\mathrm{d}\zeta)}.
\end{eqnarray*}
Under suitable conditions on $\rho$, {the above is well-defined and it
is also possible to obtain an estimate of the form:}
\[
\me U_t^{-p}\leq Ct^{-N(p)}.
\]
\end{remk}

Finally, we can state and prove our main result of this section.
%
\begin{thrm}\label{densityest-main} Assume {the condition \textup{(C3)}
of Section~\ref{SectIntro}} is satisfied. Let $\{X_t^{x}\}_{t\geq{}0}$
be the solution to equation (\ref{SDE}) {and} denote the density of
$X^x_t$ by {$p_{t}(y;x)$}. Fix {$\eta>0$ and $N>0$. Then,} there exists
$r(\eta,N)>0$ such that, if {$\nu'$} is supported in $B(0,r)$ with
$r\leq r(\eta,N)$, we have, for all $0\leq t\leq1$,
\[
\sup_{{|x-y|\geq\eta}} {p_{t}(y;x)}\leq{C(\eta,N)}
t^{N}.
\]
\end{thrm}
\begin{pf}
For a fix $t\geq0$, define a finite measure $m_t^\eta$ on $\mr$ by
\[
m_t^\eta(A)=\mp \bigl( \bigl\{X_t^x
\in A\cap\bar{B}^c(x,\eta ) \bigr\} \bigr),\qquad A\subset\mr,
\]
where $\bar{B}^c(x,r)$ denotes the complement of {the} closure of
$B(x,r)$. Thus, to prove our result it suffices to prove that $m_t^\eta
$ admits a density that has the desired bound. {To this end}, for any
smooth function $f$ compactly supported in $\bar{B}^c(x,\eta)$, we have:
\begin{eqnarray*}
\int_\mr({\partial_x}f) (y)m_t^\eta(\mathrm{d}y)
&=&\me{\partial_x}f\bigl(X^x_t\bigr)=
\me{G_t}U^{-2}_tf\bigl(X^x_t
\bigr)
\\
&=&\int_\mr\me \bigl[{G_t}U_t^{-2}
| X^x_t=y \bigr] f(y)m_t^\eta(\mathrm{d}y),
\end{eqnarray*}\eject\noindent
where the second equality follows from integration by parts. {Now by an
application of Proposition \ref{martingale-property} to $X_t^x$, one
has, for any $p>0$,
\[
m_t^\eta(\mr)\leq\mp \Bigl(\sup_{0\leq s\leq t}\bigl|X_s^x-x\bigr|
\geq\eta \Bigr)\leq C(\eta,p)t^p.
\]
The rest of the proof follows from {Proposition \ref{reg-density} and
Lemma \ref{Malliavin-singularity}}.}
\end{pf}

\subsection{Expansion for the transition density}\label{secExpDnsty}
We are ready to state the main result of this section, {namely,} the
second-order expansion for the transition densities {$\{p_{t}(\cdot
;x)\}
_{t\geq{}0}$} of the process $\{X_{t}(x)\}_{t\geq{}0}$ in terms of the
L\'evy density $g(x;y)$ defined in (\ref{DfnLvyDty0}).
The proof is presented in Appendix \ref{SecMnThPr2}.
%
\begin{thrm}\label{ThDsty}
Let $x\in\bbr$ and $y>0$. Then, under the {hypothesis} of Theorem
\ref
{ThTail}, we have
%
%
\begin{equation}
\label{SOExp} p_{t}(x+y;x):=-\frac{\partial\bbp(X_{t}(x)\geq{}x+y)}{\partial y}= t
a_{1}(x;y)+ \frac{t^{2}}{2}a_{2}(x;y)+\mathrm{O}
\bigl(t^{3}\bigr)
\end{equation}
as $t\to{}0$, where $a_{1}(x;y)$ and $a_{2}(x;y)$ admit the following
representations (for $\varepsilon>0$ {small enough}):
\[
a_{1}(x;y):=g(x;y),\qquad a_{2}(x;y):=\eth(x;y)+
\Im_{1}(x;y)+\Im_{2}(x;y),
\]
with
%
%
\begin{eqnarray}\label{MPTNDsty}
\eth(x;y)&=&{-\frac{\partial}{\partial y} \calD(x;y),}
\nonumber\\
\Im_{1}(x;y)&=& \int \bigl(g \bigl({x+\gamma(x,\bar\zeta)};y-\gamma
(x,\bar{\zeta}) \bigr) +g \bigl(x;\bar{\gamma}(x+y,\bar\zeta)-x \bigr)
\partial_{u} \bar {\gamma }(x+y,\bar\zeta) \nonumber\\
&&\quad{} -2g(x;y)
- \gamma(x,\bar\zeta)\partial_{x}g(x;y) +\gamma(x,\bar\zeta)
\partial_{y}g(x;y)\\
&&\quad{}+\partial_{y}\bigl(\gamma (x+y,\bar
\zeta)g(x;y)\bigr)\bigr)\bar{h}_{\varepsilon}(\bar\zeta) \,\mathrm{d}\bar\zeta,
\nonumber\\
\Im_{2}(x;y)&=&{\int g \bigl(x+\gamma(x,\zeta);y-\gamma(x,\zeta )
\bigr)h_{\varepsilon}(\zeta)\,\mathrm{d}\zeta-2g(x;y)\int h_{\varepsilon}(\zeta )\,\mathrm{d}\zeta,}
\nonumber
\end{eqnarray}
{and $\calD(x,y)$ be given as in (\ref{MPTN}).}
\end{thrm}

\section{The first-order term of the option price expansion}\label
{FirstOptPriceSect}
In this section, we use our previous results to derive {the leading
term of the small-time expansion for option prices of {out-of-the-money
(OTM)}} European call options. This can be achieved by either the
asymptotics of the tail distributions or the transition density. {Given
that the former requires less stringent conditions on the coefficients
of the SDE}, we choose {the former approach.}

It is well known by practitioners that the market implied volatility
skewness is more pronounced as the expiration time approaches. Such a
phenomenon indicates that a jump risk should be included into classical
purely-continuous financial models (e.g., local volatility models and
stochastic volatility models) to reproduce more accurately the implied
volatility skews observed in short-term option prices. Moreover,
further studies have shown that accurate modeling of the option market
and asset prices requires a mixture of a continuous diffusive component
and a jump component (see A\"{\i}t-Sahalia and Jacod \cite
{AitJacod06}, A\"{\i}t-Sahalia and Jacod \cite{AitJacod10},
Barndorff-Nielsen and Shephard \cite{BNSh06}, Podolskij \cite
{Podolskij06}, Carr and Wu \cite{CW03}, and Medvedev and Scailllet
\cite{MedSca07}).
The study of small-time asymptotics of option prices and implied
volatilities has grown significantly during the last decade, as it
provides a convenient tool for testing various pricing models and
calibrating parameters in each model (see, e.g., Gatheral et~al. \cite
{Gatheral2009},
Feng, Forde and
Fouque~\cite{FengFordeFouque}, Forde and Jacquier \cite
{FordeJacquier}, Berestyki, Busca and
Florent \cite{Busca22004}, Figueroa-L\'{o}pez and Forde \cite
{FLF11}, Roper~\cite{Rop10}, Tankov \cite{Tnkv10}, Gao and Lee \cite
{GL11}, Muhle-Karbe and Nutz \cite{Muhle}, Figueroa-L\'opez, Gong and
Houdr\'e~\cite
{FigGongHou2011}).
In spite of the ample literature on the asymptotic behavior of the
transition densities and option prices for either purely-continuous or
purely-jump models, results on local jump-diffusion models are scarce.
Our result in this section is thus a first attempt in this direction.

{Throughout this section, let $\{S_t\}_{t\geq0}$} be the stock price
process and let $X_t=\log S_t$ for each $t\geq0$. We assume that $\bbp
$ is the option pricing measure and that under this measure the process
{$\{X_t\}_{t\geq0}$} is of the form in {(\ref{SE1General})}. {As usual,
without loss of generality we assume that the risk-free interest rate
$r$ is $0$.} In particular, in order for $S_{t}=\exp X_{t}$ to be a
$\bbq$-(local) martingale, we fix
\[
b(x):=-\frac{1}{2}\sigma^{2}(x)-\int \bigl(\mathrm{e}^{\gamma
(x,z)}-1-{
\gamma (x,\zeta)\mathbf{1}_{\{|\zeta|\leq{}1\}}}
\bigr)h(z)\,\mathrm{d}z.
\]
We assume that $\sigma$ and $\gamma$ are such that the conditions
(C1)--(C4) of Section~\ref{SectIntro} are satisfied. We also
impose an extra condition for
$h(z)$ and $\gamma(x,z)$ in order to
derive option price expansion, as we are working
with the exponential of a jump-diffusion now:
\begin{longlist}
\item[(C5)] {$h(z)$ and $\gamma(x,z)$ are such that $ \sup_{x\in\mr
}\int_{|z|\geq{}1} \mathrm{e}^{3|\gamma(x,z)|}h(z)\,\mathrm{d}z<\infty.$}
\end{longlist}
Note that this condition ensures immediately that $b(x)$ above is well defined.

By the Markov property of the system, {it will suffice to compute a
small-time expansion for}
\[
v_t=\me(S_t-K)_+=\me \bigl(\mathrm{e}^{X_t}-K
\bigr)_{+}.
\]
In particular, using the well-known formula
\[
\me U\mathbf{1}_{\{U>K\}}
=K\mp\{U> K\}+\int_K^\infty
\mp\{U>s\}\,\mathrm{d}s,
\]
we can write
\[
\me\bigl(\mathrm{e}^{X_t}-K\bigr)_+=\int_{K}^\infty
\mp\{S_t>s\}\,\mathrm{d}s=S_0\int_ {{K}/{S_0}}^\infty\mp\{X_t-x>\log s\}\,\mathrm{d}s,
\]
where $x=X_0=\log S_0$. Recall that
%
%
\begin{equation}
\label{FndDcmp-recall} \bbp(X_{t}-x\geq{}y)=\mathrm{e}^{-\lambda_{\varepsilon}t} \sum
_{n=0}^{\infty
}\bbp\bigl(X_{t}-x
\geq{}y\vert N_{t}^{\varepsilon}=n\bigr) \frac{(\lambda_{\varepsilon}t)^{n}}{n!},
\end{equation}
where $\lambda_{\varepsilon}:=\int\phi_{\varepsilon}(\zeta)
h(\zeta
)\,\mathrm{d}\zeta$ is the jump intensity of $\{N_{t}^{\varepsilon}\}_{t\geq{}0}$.
We proceed as in Section~\ref{DistrTailSect}. First, note that
%
%
\begin{eqnarray}
\label{KDcFExp} v_t=S_0\int_{{K}/{S_0}}^\infty
\mp\{X_t-x>\log s\} \,\mathrm{d}s=S_0\mathrm{e}^{-\lambda
_\varepsilon t}(I_1+I_2+I_3),
\end{eqnarray}
where
\begin{eqnarray*}
I_1&=&\int_{{K}/{S_0}}^\infty\mp\bigl
\{X_t-x\geq\log s|N_t^\varepsilon=0\bigr\} \,\mathrm{d}s=\int
_{{K}/{S_0}}^\infty\mp\bigl\{X_t(\varepsilon,
\varnothing ,x)-x\geq \log s\bigr\}\,\mathrm{d}s,
\\
I_2&=&\lambda_\varepsilon t\int_{{K}/{S_0}}^\infty
\mp\bigl\{ X_t-x\geq\log s|N_t^\varepsilon=1\bigr\}\,\mathrm{d}s,
\\
I_3&=&\lambda_\varepsilon^2t^2\sum
_{n=2}^\infty\frac{(\lambda
_\varepsilon t)^{n-2}}{n!}\int
_{{K}/{S_0}}^\infty\mp\bigl\{X_t-x\geq \log
s|N_t^\varepsilon=n\bigr\}\,\mathrm{d}s.
\end{eqnarray*}
It is clear that {$I_1/t\to0$ as $t\to{}0$} by Lemma \ref{NLBS}. We
show that the same is true for $I_3$, which is the content of the
following lemma. {Its proof is given in Appendix \ref{SecTcnLm}.}
%
\begin{lmma}\label{l1}
With the {above} notation, we have
\[
\sup_{n\in\mn, t\in[0,1]}\frac{1}{n!}\int_0^\infty
\mp \bigl(|X_t-x|\geq\log y|N_t^\varepsilon=n\bigr)\,\mathrm{d}y<
\infty.
\]
As a consequence, $I_3/t\rightarrow0$ as $t\rightarrow0$.
\end{lmma}
Note that the above lemma actually implies that $\me \mathrm{e}^{|X_t-x|}<\infty
$ for all $t\in[0,1)$. We are ready to state the main result of this section.
%
\begin{thrm}\label{MnResOptPri}
Let $v_t=\me(S_t-K)_+$ be the price of a European call option with
strike {$K>S_0$}. Under the conditions \textup{(C1)--(C5)}, we have
%
%
\begin{equation}
\label{LOTMOP} \lim_{t\rightarrow0}\frac{1}{t}v_t=\int
_{-\infty}^\infty \bigl(S_0\mathrm{e}^{\gamma(x,\zeta)}-K
\bigr)_+h(\zeta)\,\mathrm{d}\zeta.
\end{equation}
\end{thrm}
\begin{pf}
We use the notation introduced in (\ref{KDcFExp}).
Following a similar argument as in the proof of Lemma \ref{l1}, one can
show that
%
%
\begin{equation}
\label{I2} \int_{{K}/{S_0}}^\infty\sup_{t\in[0,1]}
\mp\bigl\{X_t-x\geq\log s|N_t^\varepsilon=1\bigr\}\,\mathrm{d}s<
\infty.
\end{equation}
Also, it is clear that $I_1/t$ converges to {$0$} when $t$ approaches
to {$0$} by Lemma \ref{NLBS}.
Using the latter fact, equation (\ref{I2}), Lemma \ref{l1}, equation (\ref{KDcFExp}), and
dominated convergence theorem, we have
\[
\lim_{t\rightarrow0} \frac{v_t}{t}= \lim_{t\rightarrow0}
\frac{S_0I_2}{t} 
=\lambda_\varepsilon
S_0\int_{{K}/{S_0}}^\infty\lim
_{t\rightarrow
0}\mp\bigl\{X_t-x\geq\log s|N_t^\varepsilon=1
\bigr\}\,\mathrm{d}s.
\]
Next, using Theorem \ref{ThTail}, it follows that
\[
\lim_{t\rightarrow0} \frac{v_t}{t}=S_0\int
_{{K}/{S_0}}^\infty A_1(x,\log s)\,\mathrm{d}s
=S_0\int_{{K}/{S_0}}^\infty\int
_{\{\gamma(x,\zeta)\geq\log
s\}
}h(\zeta)\,\mathrm{d}\zeta \,\mathrm{d}s. 
\]
Finally, (\ref{LOTMOP}) follows from applying Fubini's theorem to the
right-hand side of the above equality.
\end{pf}

\begin{remk}
As a special case of our result, let $\gamma(x,\zeta)=\zeta$. The model
reduces to an exponential L\'{e}vy model. The above first-order
asymptotics becomes to
\[
\lim_{t\rightarrow0}\frac{1}{t}v_t=\int
_{-\infty}^\infty \bigl(S_0\mathrm{e}^\zeta-K
\bigr)_+h(\zeta)\,\mathrm{d}\zeta.
\]
This recovers the well-known first-order asymptotic behavior for
exponential L\'evy model (see, e.g., Roper \cite{Rop10} {and} Tankov
\cite{Tnkv10}).\vadjust{\goodbreak}
\end{remk}

\begin{appendix}\label{app}

\section{Proof of the tail distribution expansion} \label{SecMnThPr}
The proof of Theorem \ref{ThTail} is decomposed into {three steps
described in the following three} subsections. {For future use in
obtaining the expansion for the transition densities}, we {will write
explicitly} the {remainder} terms when applying Dynkin's formula {(\ref
{DynkinF}) or in any other type of approximation.}
\subsection{Key lemma to control the tail of the process with one
large jump}
The following result will allow us to obtain the second-order expansion
for the process with one large jump. {Below, we recall that
$J:=J^{\varepsilon}$ represents the jump size of the big-jump component~$Z(\varepsilon)$; that is, a random variable with density ${\breve
{h}}_{\varepsilon}(\zeta):=h_{\varepsilon}(\zeta)/\lambda
_{\varepsilon
}:=\phi_{\varepsilon}(\zeta) h(\zeta)/\lambda_{\varepsilon}$.}
%
\begin{lmma}\label{Lmm2MnLm}
Under the setting and conditions \textup{(C1)--(C4)} of Section~\ref{SectIntro},
%
%
\begin{equation}
\label{2ndExpOLJ} \mp \bigl(X_t\bigl(\varepsilon,\varnothing,z+
\gamma(z,J)\bigr)\geq\vartheta \bigr)=H_{0}(z;\vartheta)
+tH_{1}(z;\vartheta)+t^{2}\breve{\calR}_{t}(z;
\vartheta)
\end{equation}
for any $z,\vartheta\in\bbr$, where
%
%
\begin{eqnarray}
\label{EqDfnH0H1} H_{0}(z;\vartheta)&:=&\mp \bigl(\gamma(z,J)+z\geq
\vartheta \bigr),\qquad H_{1}(z;\vartheta):=D(z;\vartheta)+I(z;\vartheta),
\nonumber\\
{D}(z;\vartheta)&:=&\widetilde\Gamma(\vartheta;z){b_{\varepsilon
}(\vartheta)}-
\partial_{\vartheta}\widetilde\Gamma(\vartheta ;z)v(\vartheta)-\widetilde
\Gamma(\vartheta;z)v'(\vartheta),
\\
{I}(z;\vartheta)&:=&{\int \bigl[\mp \bigl(z+\gamma(z,J)\geq{}\bar {\gamma }(
\vartheta,\zeta) \bigr)-\mp \bigl(z+\gamma(z,J)\geq\vartheta \bigr)- \widetilde
\Gamma(\vartheta;z)\gamma(\vartheta,\zeta) \bigr]\bar {h}_\varepsilon(\zeta)\,\mathrm{d}
\zeta},
\nonumber
\end{eqnarray}
and, for $\varepsilon>0$ small enough,
\[
\limsup_{t\to{}0} \sup_{z\in\bbr}\bigl\llvert
\breve{\calR }^{1}_{t} (z;\vartheta )\bigr\rrvert <\infty,\qquad
\sup_{z,\vartheta}\bigl|H_{1}(z;\vartheta)\bigr|<\infty.
\]
\end{lmma}
The idea to obtain (\ref{2ndExpOLJ}) consists of approximating the
function $\mathbf{1}_{\{X_t(\varepsilon,\varnothing,z+\gamma(z,J))\geq
\vartheta\}}$ by a smooth sequence of functions $f_{\delta
}(X_t(\varepsilon,\varnothing,z+\gamma(z,J)))$, $\delta>0$.
Concretely, we let
\[
f_{\delta}(w):={k_{\vartheta}*\varphi_{\delta}(w)= \int
_{-\infty
}^{w-\vartheta}\varphi_{\delta}(u)\,\mathrm{d}u},
\]
where {$*$ denotes the convolution operation, $k_{\vartheta}(w):=\mathbf{
1}_{w\geq{}\vartheta}$, and} $\varphi_{\delta}(w):=\delta^{-1}
\varphi
(\delta^{-1}w)$ for a density function $\varphi\in C^{\infty}$ with
{${\rm supp}(\varphi)=[-1,1]$}. In particular, as $\delta\to{}0$,
%
%
\begin{equation}
\label{BAIdnt} {f_{\delta}(w)\to k_{\vartheta}(w)={
\mathbf1}_{\{w\geq{}\vartheta\}
}}\quad \mbox{and}\quad \int g(w) f'_{\delta}(w)
\,\mathrm{d}w=\int g(w)\varphi _{\delta}(w-\vartheta) \,\mathrm{d}w\to g(\vartheta),
\end{equation}
whenever {$w\neq{}\vartheta$ and} $g$ is bounded and continuous at
$\vartheta$. It is then natural to apply Dynkin's formula to
$f_{\delta
}(X_t(\varepsilon,\varnothing,z+\gamma(z,J)))$ and show that each of the
resulting terms is convergent when $\delta\to{}0$. The following
result, whose proof is presented in Appendix \ref{SecTcnLm}, is needed
to formalize the last step.
%
\begin{lmma} \label{UEDSJ}
Let $\widetilde{\Gamma}(\cdot;z)$ be the density of the random variable
$z+\gamma(z,J)$ and let $p_{t} (\cdot;\varepsilon,\varnothing
,\zeta
)$ be the density of $X_{t}(\varepsilon,\varnothing,\zeta)$. Then,
under the conditions \textup{(C1)--(C4)} of Section~\ref{SectIntro}, there
exists an $\varepsilon>0$ small enough such that for any compact set
$K\subset\bbr$,
%
%
\begin{equation}
\label{WNADSJ2} \limsup_{t\to{}0}\sup_{z\in\bbr}\sup
_{\eta\in K}\biggl\llvert \frac
{\partial
^{k}}{\partial\eta^{k}}\int\widetilde{\Gamma}(
\zeta;z)p_{t} (\eta ;\varepsilon,\varnothing,\zeta )\,\mathrm{d}\zeta\biggr\rrvert
<\infty,\qquad k\geq0.
\end{equation}
{Furthermore, \textup{(\ref{WNADSJ2})} holds also true with $\partial_{\eta}
p_t(\eta;\varepsilon,\varnothing,\zeta)$ in place of $p_t(\eta
;\varepsilon
,\varnothing,\zeta)$ inside the integral.}
\end{lmma}
We are now in position to show (\ref{2ndExpOLJ}).

\begin{pf*}{Proof of Lemma \ref{Lmm2MnLm}}
Throughout, $\partial_{y}\gamma$ and $\partial_{\zeta}\gamma$ will
denote the partial derivatives of $\gamma(y,\zeta)$ with respect to its
first and second arguments, respectively. By dominated convergence
theorem, we have
%
%
\begin{equation}
\label{LTT1} \mp \bigl(X_t\bigl(\varepsilon,\varnothing,z+\gamma(z,J)
\bigr)\geq\vartheta \bigr)=\lim_{\delta\downarrow0}\me f_\delta
\bigl(X_t\bigl(\varepsilon,\varnothing ,z+\gamma(z,J)\bigr)\bigr).
\end{equation}
Note that
%
%
\begin{equation}
\label{NEF} \me f_\delta\bigl(X_t\bigl(\varepsilon,
\varnothing,z+\gamma(z,J)\bigr)\bigr)=\int \widetilde \Gamma(\zeta;z) \me
f_\delta\bigl(X_t(\varepsilon,\varnothing,\zeta)\bigr) \,\mathrm{d}
\zeta,
\end{equation}
and, thus, an application of {the} Dynkin's formula {(\ref{DynkinF})}
{with $n=2$} to the expectation in the above integral {yields}
%
%
\begin{eqnarray}
\label{NEFl1}&& \me f_\delta\bigl(X_t\bigl(\varepsilon,
\varnothing,z+\gamma(z,J)\bigr)\bigr)
\nonumber
\\[-8pt]
\\[-8pt]
\nonumber
&&\quad=\int \widetilde \Gamma(
\zeta;z)f_\delta(\zeta) \,\mathrm{d}\zeta+t\int\widetilde\Gamma (\zeta;z)
L_{\varepsilon}f_\delta(\zeta) \,\mathrm{d}\zeta
\\
&&\qquad{} +t^{2}\int\widetilde\Gamma(\zeta;z) \int_0^1(1-
\alpha)\me (L_\varepsilon)^2f_\delta\bigl(X_{\alpha t}(
\varepsilon,\varnothing,\zeta )\bigr)\,\mathrm{d}\alpha \,\mathrm{d}\zeta.
\end{eqnarray}
We analyze the limit of each of the three terms on the right-hand side
of the previous equation.
By dominated convergence theorem, the leading term of (\ref{LTT1}) is
given by
\begin{eqnarray*}
H_{0}(z; \vartheta):=\lim_{\delta\downarrow0}\int\widetilde {
\Gamma }(\zeta;z)f_\delta(\zeta)\,\mathrm{d}\zeta=\int\widetilde{\Gamma}(\zeta
;z)I_{[\vartheta,\infty)}(\zeta)\,\mathrm{d}\zeta={\mp \bigl(\gamma (z,J)+z\geq \vartheta
\bigr)}.
\end{eqnarray*}
To compute the limit of the {second term, recall that} $L_\varepsilon
f_\delta=\calD_\varepsilon f_\delta+\calI_\varepsilon f_\delta$ with
$\calD_{\varepsilon}$ and $\calI_{\varepsilon}$ defined as in (\ref
{InfGenSmallJumps}). Then, the term of order $t$ has {the following two
contributions}:
\[
A_\delta:=\int\widetilde\Gamma(\zeta;z) \calD_{\varepsilon}
f_\delta (\zeta) \,\mathrm{d}\zeta,\qquad B_\delta:=\int\widetilde\Gamma(
\zeta;z) \calI_{\varepsilon} f_\delta (\zeta) \,\mathrm{d}\zeta.
\]
Using that $f_{\delta}'(\zeta)=\varphi_{\delta}(\zeta-\vartheta)$ and
by integration by parts, it follows that
\begin{eqnarray*}
A_{\delta}&=\int \bigl(\widetilde\Gamma(\zeta;z){b_{\varepsilon
}(\zeta
)}-\partial_{\zeta}\widetilde\Gamma(\zeta;z)v(\zeta)-\widetilde \Gamma (
\zeta;z)v'(\zeta) \bigr)\varphi_{\delta}(\zeta-\vartheta)\,\mathrm{d}
\zeta,
\end{eqnarray*}
{where we recall that $v(x):=\sigma^{2}(x)/2$ and $b_{\varepsilon
}(x):=b(x)-\int_{|\zeta|\leq{}1}\gamma (x,\zeta
)h_{\varepsilon
}(\zeta)\,\mathrm{d}\zeta$.} Applying~(\ref{BAIdnt}) and {Lemma \ref{LmED}(2)},
\[
\lim_{\delta\downarrow0}A_{\delta}=\widetilde\Gamma(\vartheta
;z){b_{\varepsilon}(\vartheta)}-\partial_{\vartheta}\widetilde \Gamma (
\vartheta;z)v(\vartheta)-\widetilde\Gamma(\vartheta ;z)v'(
\vartheta).
\]
We now analyze the limit of the second term $B_{\delta}$.
{Since $f_{\delta}'(\cdot)=\varphi_{\delta}(\cdot-\vartheta)$ {has
compact support}, we can apply (\ref{DlInOp2}) below to} write
$B_{\delta}$ as
\begin{eqnarray}
\label{FEstNHr} B_{\delta}&=&{\int\varphi_\delta(w-\vartheta)
\widetilde{\calH }_{\varepsilon}\widetilde{\Gamma}(w;z)\,\mathrm{d}w}
\nonumber
\\[-8pt]
\\[-8pt]
\nonumber
&=&{\int\varphi_\delta(w-\vartheta)\int \biggl( \int
_{\bar{\gamma
}(w,\zeta)}^{w} \widetilde\Gamma(\eta;z)\,\mathrm{d}\eta -
\widetilde\Gamma(w;z)\gamma(w,\zeta) \biggr) \bar{h}_\varepsilon (\zeta )\,\mathrm{d}\zeta
\,\mathrm{d}w}.
\nonumber
\end{eqnarray}
Since
\begin{eqnarray*}
\partial^{2}_{\zeta} \biggl( \int_{\bar{\gamma}(w,\zeta)}^{w}
\widetilde \Gamma(\eta;z)\,\mathrm{d}\eta - \widetilde\Gamma(w;z)\gamma(w,\zeta) \biggr)
&=&
-\partial_{\zeta}\widetilde\Gamma\bigl(\bar{\gamma}(w,\zeta);z\bigr)
\bigl(\partial_{\zeta}\bar{\gamma}(w,\zeta) \bigr)^{2}
\\
&&{} -\widetilde\Gamma\bigl(\bar{\gamma}(w,\zeta);z\bigr)\partial^{2}_{\zeta
}
\bar {\gamma}(w,\zeta) -\widetilde\Gamma(w;z)\partial^{2}_{\zeta}
\gamma(w,\zeta),
\end{eqnarray*}
the factor multiplying $ \varphi_\delta(w-\vartheta)$ in (\ref
{FEstNHr}) can be written as
\begin{eqnarray*}
\widetilde{\calH}_{\varepsilon}\widetilde{\Gamma}(w;z)&=&-\int\int
_{0}^{1} \bigl[\partial_{\zeta}\widetilde
\Gamma\bigl(\bar{\gamma }(w,\zeta\beta);z\bigr) \bigl(\partial_{\zeta}
\bar{\gamma}(w,\zeta\beta) \bigr)^{2}+\widetilde \Gamma\bigl(\bar{
\gamma}(w,\zeta\beta);z\bigr)\partial^{2}_{\zeta}\bar {\gamma
}(w,\zeta\beta)
\\
&&\hspace*{38pt}{} +\widetilde\Gamma(w;z)\partial^{2}_{\zeta}\gamma(w,\zeta
\beta ) \bigr](1-\beta)\,\mathrm{d}\beta \zeta^{2}\bar{h}_{\varepsilon}(\zeta)\,\mathrm{d}
\zeta,
\end{eqnarray*}
which shows that $\widetilde{\calH}_{\varepsilon}\widetilde{\Gamma
}(w;z)$ is bounded and continuous in $w$ in light of conditions (C2) and~(C4). Thus, using (\ref{BAIdnt}),
\begin{eqnarray*}
\lim_{\delta\downarrow0}B_{\delta}=\int \biggl( \int
_{\bar{\gamma
}(\vartheta,\zeta)}^{\vartheta} \widetilde\Gamma(\eta;z)\,\mathrm{d}\eta-
\widetilde\Gamma(\vartheta;z)\gamma(\vartheta,\zeta) \biggr) \bar
{h}_\varepsilon(\zeta)\,\mathrm{d}\zeta=:B_{0}(z;\vartheta).
\end{eqnarray*}
Recalling that $\widetilde{\Gamma}(\zeta;z)$ is the density of
$\tilde
{J}:=z+\gamma(z,J)$, $B_{0}(z;\vartheta)$ can also be written as
\[
B_{0}(z;\vartheta) = \int \bigl( \bbp \bigl(z+\gamma(z,J)\geq{}\bar{
\gamma }(\vartheta,\zeta ) \bigr)-\bbp \bigl(z+\gamma(z,J)\geq\vartheta \bigr)-
\widetilde \Gamma (\vartheta;z)\gamma(\vartheta,\zeta) \bigr)
\bar{h}_\varepsilon (\zeta )\,\mathrm{d}\zeta.
\]
Putting together the previous two limits, we obtain the term of order $t$:
\[
H_1(z;\vartheta):=\lim_{\delta\downarrow0}\int\widetilde\Gamma
(\zeta ;z)L_\varepsilon f_\delta(\zeta)\,\mathrm{d}\zeta= {D}(z;\vartheta
)+{I}(z;\vartheta),
\]
with ${D}(z;\vartheta)$ and ${I}(z;\vartheta)$ given as in the
statement of {the lemma}.

Finally, we estimate the remainder term
%
%
\begin{equation}
\label{DfnRmdCT} \breve{\calR}_{t}(z;\vartheta):=\lim
_{\delta\downarrow0}\int \widetilde \Gamma(\zeta;z)\int_0^1(1-
\alpha)\me(L_\varepsilon)^2f_\delta \bigl(X_{\alpha
t}(
\varepsilon,\varnothing,\zeta)\bigr)\,\mathrm{d}\alpha \,\mathrm{d}\zeta
\end{equation}
and show that
this is uniformly bounded for $t$ small enough. Let $\breve{\calR
}_t(z;\vartheta;\delta,\varepsilon)$ be the expression following
$\lim_{\delta\downarrow0}$ and note that
%
%
\begin{eqnarray}\label{DcmCTDfnRmdTr}
\nonumber
\breve{\calR}_t(z;\vartheta,\delta,\varepsilon)&=&\int
\widetilde \Gamma (\zeta;z)\int_0^1(1-\alpha)
\me(\calD_\varepsilon)^2f_\delta \bigl(X_{\alpha
t}(
\varepsilon,\varnothing,\zeta)\bigr)\,\mathrm{d}\alpha \,\mathrm{d}\zeta
\\
&&{} +\int\widetilde\Gamma(\zeta;z)\int_0^1(1-
\alpha)\me(\calI _\varepsilon)^2f_\delta
\bigl(X_{\alpha t}(\varepsilon,\varnothing,\zeta )\bigr)\,\mathrm{d}\alpha \,\mathrm{d}\zeta
\nonumber
\\[-8pt]
\\[-8pt]
\nonumber
&&{} +\int\widetilde\Gamma(\zeta;z)\int_0^1(1-
\alpha)\me\calI _\varepsilon\calD_\varepsilon f_\delta
\bigl(X_{\alpha t}(\varepsilon ,\varnothing,\zeta)\bigr)\,\mathrm{d}\alpha \,\mathrm{d}\zeta
\\
&&{} +\int\widetilde\Gamma(\zeta;z)\,\mathrm{d}\zeta\int_0^1(1-
\alpha)\me \calD _\varepsilon\calI_\varepsilon f_\delta
\bigl(X_{\alpha t}(\varepsilon ,\varnothing,\zeta)\bigr)\,\mathrm{d}\alpha \,\mathrm{d}\zeta.\nonumber
\end{eqnarray}
{The idea is to use} Lemmas \ref{UEDSJ} and \ref{KPDOp} to deal with
the four terms on the right-hand side of the previous equation. For
simplicity, we only give the details for second term, that we denote
hereafter $\bar{I}^{(2)}_{t}(\vartheta;\delta,\varepsilon,z)$. The
other terms can similarly be handled. {First, let us {show} that $\calI
_{\varepsilon}f_{\delta}(\cdot)$ has compact support in light of {our
condition} (\ref{NndegncyCnd}) and the fact that $f'_{\delta}$ has
compact support. Indeed, writing $ \calI_{\varepsilon}f_{\delta}$ as
\begin{eqnarray*}
\calI_{\varepsilon}f_{\delta}(y)&=&\int\int_{0}^{1}
\bigl( f_{\delta}''\bigl(y+\gamma(y,\zeta
\beta)\bigr) \bigl(\partial_{\zeta}\gamma (y,\zeta \beta)
\bigr)^{2}+f_{\delta}'\bigl(y+\gamma(y,\zeta\beta)
\bigr)\partial ^{2}_{\zeta
}\gamma(y,\zeta\beta)
\\
&&\hspace*{28pt}{} - f_{\delta}'(y) \partial^{2}_{\zeta}
\gamma(y,\zeta\beta) \bigr) (1-\beta)\,\mathrm{d}\beta \zeta^{2}
\bar{h}_{\varepsilon}(\zeta)\,\mathrm{d}\zeta,
\end{eqnarray*}
{it} is clear that $\calI_{\varepsilon}f_{\delta}(y)=0$ if $y\notin
{\rm supp}f'_{\delta}$ and $y+\gamma(y,\zeta\beta)\notin\calS
:=({\rm
supp}f'_{\delta})\cap({\rm supp}f''_{\delta})$ for any $\zeta,\beta$.
Since $|1+\partial_{y}\gamma(y,\zeta)|\geq{}\delta$, it follows that,
for $y$ large enough, $y+\gamma(y,\zeta\beta)\notin\calS$ {regardless
of $\zeta$ and $\beta$. Next,} since $\calI_{\varepsilon}f_{\delta
}(\cdot)$ has compact support, we can apply (\ref{DlInOp1}) to get
\begin{eqnarray*}
\bar{I}^{(2)}_{t}(z;\vartheta,\delta,\varepsilon)=\int
\widetilde \Gamma (\zeta;z)\int_0^1(1-\alpha)
\int\calI_{\varepsilon}f_{\delta
}(w)\widetilde{\calI}_{\varepsilon}p_{\alpha t}(w;
\varepsilon ,\varnothing ,\zeta)\,\mathrm{d}w \,\mathrm{d}\alpha \,\mathrm{d}\zeta.
\end{eqnarray*}
Next, let $\tilde{p}_{t}(\eta;\zeta):=\widetilde{\calI
}_{\varepsilon
}p_{t}(\eta;\varepsilon,\varnothing,\zeta)$. {An application of the
identity (\ref{DlInOp2}) followed by Fubini leads to}
\begin{eqnarray*}
&& {\bar{I}^{(2)}_{t}(z;\vartheta,\delta,
\varepsilon)}\nonumber \\
&&\quad= \int f_{\delta}'(w)\int_0^1(1-
\alpha)\int \biggl(\int_{\bar
\gamma
(w,\tilde\zeta)}^{w}\int\widetilde
\Gamma(\zeta;z)\tilde {p}_{\alpha
t}(\eta;\zeta)\,\mathrm{d}\zeta \,\mathrm{d}\eta
 \\
 &&\hspace*{99pt}\qquad{}-\int\widetilde\Gamma(\zeta;z)\tilde{p}_{\alpha t}(w;\zeta)\,\mathrm{d}\zeta \gamma
(w,\tilde\zeta) \biggr)\bar{h}_{\varepsilon}(\tilde\zeta)\,\mathrm{d}\tilde \zeta \,\mathrm{d}\alpha
\,\mathrm{d}w.\nonumber
\end{eqnarray*}
{Now,} fix $\breve{p}_{t}(\eta;z,\varepsilon):=\int\widetilde
{\Gamma
}(\zeta;z)p_{t} (\eta;\varepsilon,\varnothing,\zeta
)\,\mathrm{d}\zeta$
{and note that
%
%
\begin{equation}
\label{EqSRelPar} {\breve{p}_{t}'(\eta;z,\varepsilon)=
\partial_{\eta} \int \widetilde \Gamma(\zeta;z)p_{t}(\eta;
\varepsilon,\varnothing,\zeta)\,\mathrm{d}\zeta= \int\widetilde\Gamma(\zeta;z)p_{t}'(
\eta;\varepsilon,\varnothing ,\zeta )\,\mathrm{d}\zeta},
\end{equation}
in light of the last statement of Lemma \ref{UEDSJ}, which will
allow us to pass the derivative into the integration sign. Using (\ref
{EqSRelPar}) and Fubini's theorem, it follows that}
%
%
\begin{eqnarray}
\int\widetilde\Gamma(\zeta;z)\tilde{p}_{\alpha t}(\eta;\zeta )\,\mathrm{d}\zeta =\int
\widetilde\Gamma(\zeta;z)\widetilde{\calI}_{\varepsilon
}p_{\alpha
t}(\eta;
\varepsilon,\varnothing,\zeta)\,\mathrm{d}\zeta=\widetilde{\calI }_{\varepsilon}
\breve{p}_{\alpha t}(\eta;z,\varepsilon).\label {EqRepInttildeI}
\end{eqnarray}
Therefore,
%
%
\begin{eqnarray}
{\bar{I}^{(2)}_{t}(z;\vartheta,\delta,
\varepsilon)} =\sum_{j=1}^{2} \int
f_{\delta}'(w)\bar {I}^{(2,j)}_{t}(w;z,
\varepsilon ) \,\mathrm{d}w, \label{EqRepbarI2}
\end{eqnarray}
where
\begin{eqnarray*}
&&\bar{I}^{(2,1)}_{t}(w;z,\varepsilon)\\
&&\quad=-\int
_0^1(1-\alpha)\int\int_{0}^{1}(
\widetilde{\calI}_{\varepsilon}\breve{p}_{\alpha t})'\bigl(
\bar \gamma(w,\tilde\zeta\tilde\beta);z,\varepsilon\bigr) (\partial_{\zeta
}
\bar \gamma) (w,\tilde\zeta\tilde\beta) (1-\tilde{\beta})\,\mathrm{d}\tilde\beta \tilde
\zeta^{2} \bar{h}_{\varepsilon}(\tilde\zeta)\,\mathrm{d}\tilde\zeta \,\mathrm{d}\alpha,
\\
&&\bar{I}^{(2,2)}_{t}(w;z,\varepsilon)=-\int
_0^1(1-\alpha)\widetilde {\calI}_{\varepsilon}
\breve{p}_{\alpha t}(w;z,\varepsilon)\int\int_{0}^{1}
\bigl(\partial_{\zeta}^{2}\gamma\bigr) (w,\tilde\zeta\tilde
\beta) (1-\tilde \beta )\,\mathrm{d}\tilde\beta \tilde\zeta^{2}
\bar{h}_{\varepsilon}(\tilde\zeta)\,\mathrm{d}\tilde \zeta \,\mathrm{d}\alpha.
\end{eqnarray*}
Now, let us define the operator
\[
\hat{\calI} g(y;\zeta):=g\bigl(\bar{\gamma}(y,\zeta)\bigr)\partial_{y}
\bar {\gamma }(y,\zeta) -\bigl(1+\partial_{y}\gamma(y,\zeta)
\bigr)g(y)-g'(y)\gamma(y,\zeta).
\]
By writing $\widetilde{\calI}_{\varepsilon}g(y)$ as
\[
\widetilde{\calI}_{\varepsilon}g(y)=\int\int_{0}^{1}
\bigl(\partial _{\zeta
}^{2}\hat{\calI} g\bigr) (y;\bar\zeta\bar
\beta) (1-\bar\beta)\,\mathrm{d}\bar \beta \bar\zeta^{2}\bar{h}_{\varepsilon}(
\bar\zeta)\,\mathrm{d}\bar\zeta,
\]
it is not hard to see that $\widetilde{\calI}_{\varepsilon}\breve
{p}_{\alpha t}(w;z,\varepsilon)$ can be expressed as follows
%
%
\begin{eqnarray}
\label{EqRepTildeI}
\widetilde{\calI}_{\varepsilon}
\breve{p}_{\alpha t}(w;z,\varepsilon) &=&\sum_{k=0}^{2}
\int\int_{0}^{1}\breve{p}_{\alpha t}^{(k)}
\bigl(\bar {\gamma }(w,\bar\zeta\bar\beta);z,\varepsilon\bigr)
\calD^{(1)}_{k}(w;\bar\zeta\bar\beta) (1-\bar\beta)\,\mathrm{d}\bar{
\beta} \bar \zeta^{2}\bar{h}_{\varepsilon}(\bar\zeta)\,\mathrm{d}\bar\zeta
\nonumber
\\[-8pt]
\\[-8pt]
\nonumber
&&{}+\sum_{k=0}^{1} \breve{p}_{\alpha t}^{(k)}(w;z,
\varepsilon)\int \int_{0}^{1}
\calD^{(2)}_{k}(w;\bar\zeta\bar\beta) (1-\bar\beta)\,\mathrm{d}\bar{
\beta} \bar \zeta^{2}\bar{h}_{\varepsilon}(\bar\zeta)\,\mathrm{d}\bar\zeta,
\end{eqnarray}
where $\calD_{j}^{1}(w;\zeta)$ is a finite sum of terms, which consists
of the product of partial derivatives of $\bar{\gamma}(w;\zeta)$.
Similarly, $\calD_{j}^{2}(w;\zeta)$ is a finite sum of terms, which
consists of the product of partial derivatives of ${\gamma}(w;\zeta)$.
In particular, both $\calD_{j}^{1}(w;\zeta)$ and $\calD
_{j}^{2}(w;\zeta
)$ are uniformly bounded and continuous and, also, in light of Lemma
\ref{UEDSJ}, $(\widetilde{\calI}_{\varepsilon}\breve{p}_{\alpha
t})'(w;z,\varepsilon)$ will also be of the same form as~(\ref{EqRepTildeI}).

Upon the substitutions of (\ref{EqRepTildeI}) (and the analog
representation for $(\widetilde{\calI}_{\varepsilon}\breve
{p}_{\alpha
t})'(w;z,\varepsilon)$) into~(\ref{EqRepbarI2}),
we can represent $\bar{I}^{(2)}_{t}(z;\vartheta,\delta,\varepsilon
)$ as
the sum of terms of the form
\[
{\int f_{\delta}'(w)\int_0^1(1-
\alpha)\tilde{I}_{\alpha
t}(w;z,\varepsilon)\,\mathrm{d}\alpha \,\mathrm{d}w},
\]
where $\tilde{I}_{\alpha t}(w;z,\varepsilon)$ will take one of the
following four generic {forms with some function $\widetilde
{D}(w,\zeta
)$ in $C^{\geq{}1}_{b}(\bbr\times\bbr)$}:
%
%
\begin{eqnarray}
\nonumber
\tilde{I}^{(1)}_{\alpha t}(w;z,\varepsilon)&=&
\int\int_{0}^{1} \int\int_{0}^{1}
\breve{p}_{\alpha t}^{(k)}\bigl(\bar{\gamma}\bigl(\bar \gamma (w,
\tilde\zeta\tilde\beta),\bar\zeta\bar\beta\bigr);z,\varepsilon \bigr)\widetilde
\calD\bigl(\bar\gamma(w,\tilde\zeta\tilde\beta);\bar\zeta\bar\beta \bigr)
\\
&&\hspace*{52pt}{} \times(1-\bar\beta)\,\mathrm{d}\bar{\beta} \bar\zeta^{2} \bar
{h}_{\varepsilon
}(\bar\zeta)\,\mathrm{d}\bar\zeta (\partial_{\zeta}\bar\gamma) (w,
\tilde\zeta\tilde\beta) (1-\tilde {\beta })\,\mathrm{d}\tilde\beta \tilde\zeta^{2}
\bar{h}_{\varepsilon}(\tilde \zeta )\,\mathrm{d}\tilde\zeta,
\nonumber
\\
\tilde{I}^{(2)}_{\alpha t}(w;z,\varepsilon)&=&\int\int
_{0}^{1} \int\int_{0}^{1}
\widetilde\calD\bigl(\bar\gamma(w,\tilde\zeta\tilde \beta );\bar\zeta\bar\beta
\bigr) (1-\bar\beta)\,\mathrm{d}\bar{\beta} \bar\zeta ^{2} \bar
{h}_{\varepsilon}(\bar\zeta)\,\mathrm{d}\bar\zeta
\nonumber
\\[-8pt]
\\[-8pt]
\nonumber
&&\hspace*{52pt}{} \times\breve{p}_{\alpha t}^{(k)}\bigl(\bar\gamma(w,\tilde\zeta
\tilde \beta );z,\varepsilon\bigr) (\partial_{\zeta}\bar\gamma) (w,\tilde
\zeta\tilde\beta) (1-\tilde {\beta })\,\mathrm{d}\tilde\beta \tilde\zeta^{2}
\bar{h}_{\varepsilon}(\tilde \zeta )\,\mathrm{d}\tilde\zeta,\qquad
\\
\tilde{I}^{(3)}_{\alpha t}(w;z,\varepsilon)&=& \int\int
_{0}^{1}\breve {p}_{\alpha t}^{(k)}
\bigl(\bar{\gamma}(w,\bar\zeta\bar\beta );z,\varepsilon\bigr) \widetilde{\calD}(w;
\bar\zeta\bar\beta) (1-\bar\beta)\,\mathrm{d}\bar {\beta} \bar \zeta^{2}
\bar{h}_{\varepsilon}(\bar\zeta)\,\mathrm{d}\bar\zeta\label {CmplGenTrm0}
\nonumber\\
&&{} \times\int\int_{0}^{1} \bigl(
\partial_{\zeta}^{2}\gamma\bigr) (w,\tilde\zeta\tilde\beta) (1-
\tilde \beta )\,\mathrm{d}\tilde\beta \tilde\zeta^{2} \bar{h}_{\varepsilon}(
\tilde\zeta)\,\mathrm{d}\tilde \zeta,
\nonumber
\\
\tilde{I}^{(4)}_{\alpha t}(w;z,\varepsilon)&=&\breve{p}_{\alpha
t}^{(k)}(w;z,
\varepsilon)\int\int_{0}^{1} \widetilde\calD(w;
\bar\zeta\bar\beta) (1-\bar\beta)\,\mathrm{d}\bar{\beta } \bar \zeta^{2}
\bar{h}_{\varepsilon}(\bar\zeta)\,\mathrm{d}\bar\zeta
\nonumber
\\
&&{} \times\int\int_{0}^{1} \bigl(
\partial_{\zeta}^{2}\gamma\bigr) (w,\tilde\zeta\tilde\beta) (1-
\tilde \beta )\,\mathrm{d}\tilde\beta \tilde\zeta^{2} \bar{h}_{\varepsilon}(
\tilde\zeta)\,\mathrm{d}\tilde \zeta .
\nonumber
\end{eqnarray}
Using Lemma \ref{UEDSJ}, it is now clear that each $\tilde{I}_{\alpha
t}^{(i)}(w;z,\varepsilon)$ is uniformly bounded in $w$ and $z$ for $t$
small enough. Concretely, using (\ref{WNADSJ2}), it follows that, for
$\varepsilon,t>0$ small enough,
%
%
\begin{equation}
\label{BndI1} \sup_{z\in\bbr,w\in{}{\rm supp} f_{1}}\bigl|\tilde {I}_{t}^{(i)}(w;z,
\varepsilon)\bigr|<\infty.
\end{equation}
Due to the continuity $\tilde{I}_{t}^{(i)}(w;z,\varepsilon)$ and
uniformly boundedness condition (\ref{BndI1}), it turns out that
%
%
\begin{equation}
\label{EqGenlimterm} \lim_{\delta\to{}0}\int f_{\delta}'(w)
\int_0^1(1-\alpha)\tilde {I}_{\alpha t}(w;z,
\varepsilon)\,\mathrm{d}\alpha \,\mathrm{d}w=\int_0^1(1-\alpha)\tilde
{I}_{\alpha t}(\vartheta;z,\varepsilon)\,\mathrm{d}\alpha,
\end{equation}
which is uniformly bounded in $z$ for any fixed $\vartheta$ and
$0<t<t_0$ with $t_{0}>0$ small enough.}
\end{pf*}

\subsection{The leading term}
In order to determine the leading term of (\ref{TDJD}), we analyze the
second term in (\ref{FndDcmp}) corresponding to $n=1$ (only one
``large'' jump). {Again, we emphasize that in order to obtain the
expansion for the transition densities below, we will need to write
explicitly the remainder terms when applying Dynkin's formula {(\ref
{DynkinF})}.}

By conditioning on the time of the jump (necessarily uniformly
distributed on $[0,t]$), 
%
%
\begin{eqnarray}
\label{Cnd1Jmp} \bbp\bigl(X_{t}(x)\geq{}x+y\vert
N_{t}^{\varepsilon}=1\bigr)=\frac
{1}{t}\int
_{0}^{t}\bbp\bigl(X_{t}\bigl(
\varepsilon,\{s\},x\bigr)\geq{}x+y\bigr)\,\mathrm{d}s.
\end{eqnarray}
Conditioning on {$\msf_{s^{-}}$},
%
%
\begin{equation}
\bbp\bigl(X_{t}\bigl(\varepsilon,\{s\},x\bigr)\geq{}x+y\bigr) =\bbe
\bigl( G_{t-s} \bigl(X_{s^{-}}(\varepsilon,\varnothing,x) \bigr)
\bigr)=\bbe \bigl( G_{t-s} \bigl(X_{s}(\varepsilon,
\varnothing ,x) \bigr) \bigr), \label{Simpl1}
\end{equation}
where
%
%
\begin{eqnarray}
\label{FTE1} G_{t}(z)&:=G_{t}(z;x,y):={\bbp
\bigl[X_{t} \bigl(\varepsilon ,\varnothing ,z+\gamma(z,J) \bigr)
\geq{}x+y \bigr]}.
\end{eqnarray}
{Using Lemma \ref{Lmm2MnLm},}
%
%
\begin{eqnarray}
\label{Dcmp3T1}
&&
\bbp\bigl(X_{t}\bigl(\varepsilon,\{s\},x
\bigr)\geq{}x+y\bigr)\nonumber\\
&&\quad=\bbe H_{0}\bigl(X_{s}(\varepsilon ,
\varnothing,x);{x+y}\bigr)+(t-s)\bbe H_{1}\bigl(X_{s}(
\varepsilon,\varnothing ,x);{x+y}\bigr)\qquad
\\
&&\qquad{} + {(t-s)^{2}\bbe\calR^{1}_{t-s}
\bigl(X_{s}(\varepsilon,\varnothing ,x);x,y\bigr)},\nonumber
\end{eqnarray}
where $\calR^{1}_{t}(w;x,y):=\breve{\calR}_{t}(w;x+y)$.
Next, we apply {the Dynkin's formula (\ref{DynkinF})} {with $n=2$} to
$\bbe H_{0}(X_{s}(\varepsilon,\varnothing,x);{x+y})$, {which is valid since
$H_{0}(z;x+y)=\bbp (\gamma(z,J)+z\geq{}x+y )$ is $C_{b}^{4}$
in light of Lemma \ref{LmED}(3).}
By {(\ref{DynkinF})},
%
%
\begin{equation}
\label{Dcmp2} \bbe H_{0}\bigl(X_{s}(\varepsilon,
\varnothing,x);x+y\bigr)= H_{0,0}(x;y)+ s{H}_{0,1}(x;y)+{s^{2}
\calR_{s}^{2}(x;y)},
\end{equation}
where
%
%
\begin{eqnarray}\label{NETBP0}
H_{0,0}(x;y)&:=&H_{0}(x;x+y)= \bbp \bigl[ \gamma(x,J)\geq{}y
\bigr],
\nonumber
\\
{H}_{0,1}(x;y)&:=&(L_{\varepsilon}H_{0}) (x;x+y) =
{b_{\varepsilon}(x)}\frac{\partial H_{0}(z;x+y)}{\partial z}\bigg |_{z=x}\nonumber\\
&&\hspace*{82pt}{}+ \frac{\sigma^{2}(x)}{2}
\frac{\partial^{2}H_{0}(z;x+y)}{\partial
z^{2}} \bigg|_{z=x}
\nonumber
\\[-8pt]
\\[-8pt]
\nonumber
&&\hspace*{82pt}{} +\int \biggl(H_{0}\bigl(x+\gamma(x,\zeta );x+y\bigr)-H_{0}(x;x+y)
\\
 &&\hspace*{108pt}{}-{\gamma(x,\zeta)\frac{\partial_{z} H_{0}(z;x+y)}{\partial z} \bigg|_{z=x}} \biggr)
\bar{h}_{\varepsilon}(\zeta) \,\mathrm{d}\zeta,
\nonumber\\
\calR_{s}^{2}(x;y)&:=&{\int_{0}^{1}
(1-\alpha) \bbe \bigl(L^{2}_{\varepsilon
} H_{0}\bigr)
\bigl(X_{\alpha s}(\varepsilon,\varnothing,x);x+y\bigr)\,\mathrm{d}\alpha}.
\nonumber
\end{eqnarray}
Note that $\sup_{s<1,x,y}|\calR_{s}^{2}(x;y)|<\infty$ in light of
{Lemma \ref{RemLoc}} and, also, by writing $\bbp [ \tilde\gamma
(z,J)\geq{}x+y ]=\bbp [ \gamma(z,J)\geq{}x+y-z ]$ {as
$G(z,x+y-z)$} with $G(x,y)=\bbp (\gamma(x,J)\geq{}y )$,
we {have}
\begin{eqnarray*}
{\frac{\partial H_{0}(z;x+y)}{\partial z} \bigg|_{z=x}}&=& \frac
{\partial
\bbp [ \tilde\gamma(z,J)\geq{}x+y ]}{\partial z} \bigg|_{z=x}=
\frac{\partial\bbp [ \gamma(x,J)\geq{}y ]}{\partial
x}+\Gamma(y;x),
\\
{\frac{\partial^{2} H_{0}(z;x+y)}{\partial z^{2}} \bigg|_{z=x}}&=&\frac
{\partial^{2}\bbp [ \tilde\gamma(z,J)\geq{}x+y
]}{\partial
z^{2}}\bigg |_{z=x}\\
&=&
\frac{\partial^{2}\bbp [\gamma(x,J)\geq
{}y
]}{\partial x^{2}}+2\frac{\partial\Gamma(y;x)}{\partial x}-\frac
{\partial\Gamma(y;x)}{\partial y}.
\end{eqnarray*}
Substituting the previous identities in (\ref{NETBP0}), we can write
$H_{0,1}(x;y)$ as
\begin{eqnarray}
\label{NETBP}
\nonumber
H_{0,1}(x;y)&=&{b_{\varepsilon}(x)} \biggl(
\frac{\partial\bbp [
\gamma
(x,J)\geq{}y ]}{\partial x}+\Gamma(y;x) \biggr)
\nonumber
\\[-8pt]
\\[-8pt]
\nonumber
&&{} +\frac{\sigma^{2}(x)}{2} \biggl( \frac{\partial^{2}\bbp
[\gamma
(x,J)\geq{}y ]}{\partial x^{2}}+2\frac{\partial\Gamma
(y;x)}{\partial x}-
\frac{\partial\Gamma(y;x)}{\partial y} \biggr) +\hat{H}_{0,1}(x;y),
\end{eqnarray}
with $\hat{H}_{0,1}(x;y)$ given by
\begin{eqnarray}
\label{MPTN1} \hat{H}_{0,1}(x;y)&=& \int \biggl(\bbp \bigl[ \gamma
\bigl(x+\gamma (x,\zeta ),J\bigr)\geq{}y-\gamma(x,\zeta) \bigr] -\bbp \bigl[
\gamma(x,J)\geq {}y \bigr]
\nonumber
\\[-8pt]
\\[-8pt]
\nonumber
& &\hspace*{14pt}{}-{\gamma(x,\zeta) \biggl( \frac{\partial\bbp [
\gamma
(x,J)\geq{}y ]}{\partial x}+\Gamma(y;x) \biggr)} \biggr)\bar
{h}_{\varepsilon}(\zeta) \,\mathrm{d}\zeta.
\nonumber
\end{eqnarray}

Plugging (\ref{Dcmp2}) in (\ref{Dcmp3T1}) {and recalling from Lemma
\ref
{Lmm2MnLm} that the second and third terms on the right-hand side of
(\ref{Dcmp3T1}) are bounded for $t$ small enough}, we get that
\[
\bbp\bigl(X_{t}\bigl(\varepsilon,\{s\},x\bigr)\geq{}x+y\bigr)=\bbp
\bigl[ \gamma (x,J)\geq {}y \bigr]+\mathrm{O}(t).
\]
The latter can then be plugged in (\ref{Cnd1Jmp}) to get
\[
\bbp\bigl(X_{t}(x)\geq{}x+y\vert N_{t}^{\varepsilon}=1
\bigr)=\bbp \bigl[ \gamma(x,J)\geq{}y \bigr]+\mathrm{O}(t).
\]
Finally, (\ref{FndDcmp}) can be written as
%
%
\begin{eqnarray}
\label{EqLedOrTrm}
\bbp\bigl(X_{t}(x)\geq{}x+y
\bigr)&=&\mathrm{e}^{-\lambda_{\varepsilon}t} t\lambda _{\varepsilon} \bbp \bigl[ \gamma(x,J)\geq{}y
\bigr]+\mathrm{O}\bigl(t^{2}\bigr)
\nonumber
\\[-8pt]
\\[-8pt]
\nonumber
&=& t\int\mathbf{1}_{\{\gamma(x,\zeta)\geq{}y\}}h(\zeta)\,\mathrm{d}\zeta +\mathrm{O}\bigl(t^{2}\bigr),
\end{eqnarray}
{where, in the first equality, we used (\ref{FndBnd1}) to justify that
$\bbp(X_{t}(x)\geq{}x+y\vert  N_{t}^{\varepsilon}=0)=\bbp
(X_{t}(\varepsilon,\varnothing,x)\geq{}x+y)=\mathrm{O}(t^{2})$ while, in the
second equality above, we take} $\varepsilon>0$ small enough. Equation~{(\ref
{EqLedOrTrm}) gives first-order asymptotic expansion of the tail
probability $\bbp(X_{t}(x)\geq{}x+y)$. We now proceed to obtain the
second-order term.}

\subsection{Second-order term}
In addition to (\ref{Dcmp2}), we also consider the leading terms
in the term $\bbe H_{1}(X_{s}(\varepsilon,\varnothing,x);{x+y})$ of~(\ref
{Dcmp3T1}) and
the term $\bbp(X_{t}(x)\geq{}x+y\vert  N_{t}^{\varepsilon}=2)$
of (\ref{FndDcmp}).
Let us first show that $z\to H_{1}(z;x+y)$ is $C^{2}_{b}$. {To this
end,} let
\begin{eqnarray*}
\calK(\zeta;x,y,z)&:=&\bbp \bigl[ z+\gamma(z,J)\geq{}\bar\gamma (x+y,\zeta )
\bigr] -\bbp \bigl[ z+\gamma(z,J)\geq{}x+y \bigr]\\
&&{}-{\widetilde{\Gamma }(x+y;z)
\gamma(x+y,\zeta)},
\end{eqnarray*}
and recall that
\begin{eqnarray*}
H_{1}(z;x+y)&=&\widetilde\Gamma(x+y;z){b_{\varepsilon
}(x+y)}-(
\partial _{\zeta}\widetilde\Gamma) (x+y;z)v(x+y)
\\
&&{} -\widetilde\Gamma(x+y;z)v'(x+y)+\int\calK(\zeta;x,y,z)\bar
{h}_{\varepsilon}(\zeta) \,\mathrm{d}\zeta,
\end{eqnarray*}
{where $\partial_{\zeta}\widetilde\Gamma$ and $\partial
_{z}\widetilde
\Gamma$ denote the partial derivatives of the density $\widetilde
\Gamma
(\zeta;z)$.}
Obviously, the first three terms on the right-hand side of the previous
expression are $C^{2}_{b}$ in light of {Lemma \ref{LmED}(2).}
Hence, for the derivative $\partial_{z} H_{1}(z;x+y)$ to exist, it
suffices to show that $\partial_{z} \calK(\zeta;x,y,z)$ exists and that
%
%
\begin{equation}
\label{IntStpDffH1} \sup_{z,x,y}\biggl\llvert \frac{\partial\calK(\zeta;x,y,z)}{\partial
z}
\biggr\rrvert <C {|\zeta|^{2}}
\end{equation}
for any $|\zeta|<\varepsilon$ and some constant $C<\infty$.
Recalling that
\begin{eqnarray*}
\calK(\zeta;x,y,z)&=&\int_{\bar\gamma(x+y,\zeta)}^{x+y}\widetilde {
\Gamma }(\eta;z)\,\mathrm{d}\eta-{\widetilde{\Gamma}(x+y;z)\gamma(x+y,\zeta)}
\\
&=&\int_{0}^{1} \bigl[(\partial_{\zeta}
\widetilde{\Gamma}) \bigl(\bar \gamma (x+y,\zeta\beta);z\bigr) (
\partial_{\zeta}\bar\gamma) (x+y,\zeta\beta)\\
&&\hspace*{17pt}{} - \widetilde{\Gamma}(x+y;z) \bigl(\partial^{2}_{\zeta
}
\gamma \bigr) (x+y,\zeta\beta) \bigr](1-\beta)\,\mathrm{d}\beta\zeta^{2}
\end{eqnarray*}
and using that $\widetilde{\Gamma}(\eta;z)\in C_{b}^{\infty}$, {we can
write $\partial_{z} \calK(\zeta;x,y,z)$ as}
\begin{eqnarray*}
&&\int_{0}^{1} \bigl(\bigl(\partial^{2}_{z,\zeta}
\widetilde{\Gamma }\bigr) \bigl(\bar\gamma (x+y,\zeta\beta);z\bigr) (
\partial_{\zeta}\bar\gamma) (x+y,\zeta\beta)\\
&&\hspace*{18pt}{}- (\partial _{z}
\widetilde {\Gamma}) (x+y;z) \bigl(\partial^{2}_{\zeta}\gamma
\bigr) (x+y,\zeta\beta ) \bigr) (1-\beta)\,\mathrm{d}\beta\zeta^{2}.
\end{eqnarray*}
Therefore, in light of Lemma \ref{LmED} {and the fact that $\gamma
\in
C^{\geq{}1}_{b}$}, there exists a constant $C$ such that~(\ref
{IntStpDffH1}) holds.
We can similarly prove that $\partial^{2}_{z} H_{1}(z;x,y)$ exists and
is bounded.

{Using {Dynkin's formula} (\ref{DynkinF}) with $n=1$} and that
$\widetilde{\Gamma}(\zeta;z)=\Gamma(\zeta-z;z)$, we get
%
%
\begin{equation}
\label{Dcmp3} \bbe H_{1}\bigl(X_{s}(\varepsilon,
\varnothing,x);x,y\bigr)= {H}_{1,0}(x,y)+s \calR ^{3}_{s}(x;y),
\end{equation}
where
%
%
\begin{eqnarray}\label{EqDfnH10a}
\nonumber
{H}_{1,0}(x;y)&:=&H_{1}(x;x+y)=
\calD_{1,0}(x;y)+\hat {H}_{1,0}(x;y)\qquad \mbox{with}
\\
\calD_{1,0}(x;y)&:=&\Gamma(y;x){b_{\varepsilon}(x+y)}-(\partial
_{\zeta
}\Gamma) (y;x)v(x+y)-\Gamma(y;x)v'(x+y),
\nonumber
\\
\hat{H}_{1,0}(x;y) &:=&\int \bigl(\bbp \bigl[ x+\gamma(x,J)\geq {}\bar
\gamma(x+y,\zeta) \bigr] -\bbp \bigl[ {\gamma(x,J)\geq{}y} \bigr]\\
&&\quad{} -{{\Gamma}(y;x)\gamma(x+y,\zeta)} \bigr)\bar{h}_{\varepsilon}(\zeta) \,\mathrm{d}\zeta,
\nonumber\\
\calR^{3}_{s}(x;y)&:=&\int_{0}^{1}
\bbe L_{\varepsilon
}H_{1}\bigl(X_{\alpha
s}(\varepsilon,
\varnothing,x);x+y\bigr)\,\mathrm{d}\alpha=\mathrm{O}(1)\qquad \mbox{as }s\to {}0.
\nonumber
\end{eqnarray}

In order to handle $\bbp(X_{t}(x)\geq{}x+y\vert
N_{t}^{\varepsilon}=2)$, we again condition on the times of the jumps,
which are necessarily distributed as the order statistics of two
independent uniform $[0,t]$ random variables. Concretely,
%
%
\begin{eqnarray}
\label{ScndOrdEq2} \bbp\bigl(X_{t}(x)\geq{}x+y\vert
N_{t}^{\varepsilon}=2\bigr)=\frac
{2}{t^{2}}\int
_{0}^{t} \int_{s_{1}}^{t}
\bbp\bigl(X_{t}\bigl(\varepsilon,\{s_{1},s_{2}
\},x\bigr)\geq {}x+y\bigr)\,\mathrm{d}s_{2}\,\mathrm{d}s_{1}.
\end{eqnarray}
Next, we determine the leading term of $\bbp(X_{t}(\varepsilon,\{
s_{1},s_{2}\},x)\geq{}x+y)$. By conditioning on $\msf_{s^{-}_{2}}$,
\[
\bbp\bigl(X_{t}\bigl(\varepsilon,\{s_{1},s_{2}
\},x\bigr)\geq{}x+y\bigr)=\bbe \bigl( G_{t-s_{2}} \bigl(X_{s_{2}}
\bigl(\varepsilon,\{s_{1}\},x\bigr) \bigr) \bigr),
\]
where, by Lemma \ref{Lmm2MnLm},
%
%
\begin{eqnarray}
\label{FTE12} G_{t}(z)&=&\bbp \bigl[X_{t} \bigl(\varepsilon,
\varnothing,z+\gamma (z,J) \bigr)\geq{}x+y \bigr]
\nonumber
\\[-8pt]
\\[-8pt]
\nonumber
&=&H_{0}(z;x+y)
+tH_{1}(z;x+y)+t^{2}\breve{\calR}_{t}(z;x+y).
\end{eqnarray}
Then, for $\varepsilon>0$ and $t$ small enough,
%
%
\begin{eqnarray}
\label{ScndOrdEq3} &&\bbp\bigl(X_{t}\bigl(\varepsilon,
\{s_{1},s_{2}\},x\bigr)\geq{}x+y\bigr)
\nonumber
\\[-8pt]
\\[-8pt]
\nonumber
&&\quad=\bbe \bigl(
H_{0} \bigl(X_{s_{2}}\bigl(\varepsilon,\{s_{1}\},x
\bigr);x+y \bigr) \bigr)
 + (t-s_{2})\bbe\calR^{4}_{t-s_{2}}
\bigl(X_{s_{2}}\bigl(\varepsilon,\{ s_{1}\} ,x\bigr);x,y
\bigr),
\end{eqnarray}
with
\[
 \calR^{4}_{t}(z;x,y):=H_{1}(z;x+y)+t
\breve{\calR }_{t}(z;x+y).
\]
Again, conditioning on $\msf_{s^{-}_{1}}${,}
\[
\bbe \bigl( H_{0} \bigl(X_{s_{2}}\bigl(\varepsilon,
\{s_{1}\},x\bigr);x+y \bigr) \bigr)=\bbe \bigl(\widehat{G}_{s_{2}-s_{1}}
\bigl(X_{s_{1}}(\varepsilon ,\varnothing,x);x+y \bigr) \bigr),
\]
where
\[
\widehat{G}_{t}(z;x+y):=\bbe H_{0} \bigl(X_{t}
\bigl(\varepsilon,\varnothing ,z+\gamma(z,J)\bigr);x+y \bigr).
\]
Since $z\to H_{0}(z;x+y)=\bbp(z+\gamma(z,J)\geq{}x+y)$ is
$C_{b}^{\infty
}$ by Lemma \ref{LmED}(3), {we can apply Dynkin's formula} {(\ref
{DynkinF}) with $n=1$} to deduce
\begin{eqnarray*}
\widehat{G}_{t}(z;x+y)&=& \int\widetilde\Gamma(\zeta;z) \me
H_{0} \bigl(X_t(\varepsilon ,\varnothing ,\zeta);x+y
\bigr) \,\mathrm{d}\zeta\\
&=&\int\widetilde\Gamma(\zeta;z) H_{0} (\zeta;x+y ) \,\mathrm{d}\zeta+t
\calR_{t}^{6}(z;x,y)
\\
&=:& H_{2}(z;x+y)+t \calR_{t}^{6}(z;x,y),
\end{eqnarray*}
where, denoting two independent copies of $J$ by $J_{1},J_{2}$,
\begin{eqnarray*}
H_{2}(z;x+y)&:=&\bbp \bigl(z+\gamma(z,J_{1})+\gamma\bigl(z+
\gamma (z,J_{1}),J_{2}\bigr)\geq x+y \bigr),
\\
\calR_{t}^{6}(z;x,y)&:=&\int\widetilde\Gamma(\zeta;z) \int
_{0}^{1}\bbe L_{\varepsilon}H_{0}
\bigl(X_{\alpha t}(\varepsilon ,\varnothing,\zeta);x+y \bigr) \,\mathrm{d}\alpha \,\mathrm{d}\zeta.
\end{eqnarray*}
Therefore,
\begin{eqnarray*}
&&\bbp\bigl(X_{t}\bigl(\varepsilon,\{s_{1},s_{2}
\},x\bigr)\geq{}x+y\bigr)\\
&&\quad=\bbe \bigl( H_{2} \bigl(X_{s_{1}}(
\varepsilon,\varnothing,x);x+y \bigr) \bigr)+(s_{2}-s_{1})
\bbe\calR_{s_{2}-s_{1}}^{6}\bigl(X_{s_{1}}(\varepsilon ,
\varnothing,x);x,y\bigr)
\\
&&\qquad{}+(t-s_{2})\bbe\calR^{4}_{t-s_{2}}
\bigl(X_{s_{2}}\bigl(\varepsilon,\{ s_{1}\} ,x\bigr);x,y
\bigr).
\end{eqnarray*}
Applying again {Dynkin's formula} {(\ref{DynkinF}) with $n=1$} to the
first term on the right-hand side of the previous equation, {we can write}
%
%
\begin{eqnarray}
\label{ExNdLt}
&&\bbp\bigl(X_{t}\bigl(\varepsilon,\{s_{1},s_{2}
\},x\bigr)\geq{}x+y\bigr)\nonumber\\
&&\quad= H_{2,0}(x;y)+s_{1}
\calR^{5}_{s_{1}}(x;y)
\nonumber
\\[-8pt]
\\[-8pt]
\nonumber
&&\qquad{} +(s_{2}-s_{1})\bbe\calR _{s_{2}-s_{1}}^{6}
\bigl(X_{s_{1}}(\varepsilon ,\varnothing,x);x,y\bigr)
\\
&&\qquad{} +(t-s_{2})\bbe\calR^{4}_{t-s_{2}}
\bigl(X_{s_{2}}\bigl(\varepsilon ,\{ s_{1}\},x\bigr);x,y
\bigr),
\nonumber
\end{eqnarray}
where
\begin{eqnarray*}
H_{2,0}(x;y)&:=&H_{2} (x;x+y )=\bbp \bigl(\gamma
(x,J_{1})+\gamma \bigl(x+\gamma(x,J_{1}),J_{2}
\bigr)\geq{}y \bigr),
\\
\calR^{5}_{s_{1}}(x;y)&:=&\int_{0}^{1}
\bbe L_{\varepsilon
}H_{2}\bigl(X_{\alpha s_{1}}(\varepsilon,
\varnothing,x);x+y\bigr)\,\mathrm{d}\alpha.
\end{eqnarray*}
Therefore, we conclude that
%
%
\begin{equation}
\label{EST} \bbp \bigl(X_{t}(x)\geq{}x+y|N_{t}^{\varepsilon}=2
\bigr)=H_{2,0}(x;y)+\mathrm{O}(t).
\end{equation}
In light of (\ref{Cnd1Jmp}), {(\ref{Dcmp3T1})--(\ref{NETBP})}, (\ref
{Dcmp3}), and (\ref{EST}), we have the following second-order
decomposition of the tail distribution $\bbp(X_{t}(x)\geq{}x+y)$:
\begin{eqnarray*}
&&\bbp\bigl(X_{t}(x)\geq{}x+y\bigr)\\
&&\quad=\mathrm{e}^{-\lambda_{\varepsilon} t}
\lambda_{\varepsilon} t H_{0,0}(x;y)+ \mathrm{e}^{-\lambda_{\varepsilon} t}
\frac{\lambda_{\varepsilon} t^{2}}{2} \bigl({H}_{0,1}(x;y)+{H}_{1,0}(x;y) \bigr)
\\
&&\qquad{} +\mathrm{e}^{-\lambda_{\varepsilon}t} \frac{(\lambda_{\varepsilon}t)^{2}}{2} H_{2,0}(x;y) +\mathrm{O}
\bigl(t^{3}\bigr)
\\
&&\quad=\lambda_{\varepsilon} t H_{0,0}(x;y)+ \frac{t^{2}}{2} \bigl\{
\lambda_{\varepsilon} \bigl[{H}_{0,1}(x;y)+{H}_{1,0}(x;y)
\bigr]
 + \lambda _{\varepsilon}^{2} \bigl[H_{2,0}(x;y) - 2
H_{0,0}(x;y) \bigr] \bigr\}\\
&&\qquad{}+\mathrm{O}\bigl(t^{3}\bigr),
\end{eqnarray*}
where, in the first equality {above,} we had again used (\ref{FndBnd1})
to justify that
\[
\bbp\bigl(X_{t}(x)\geq{}x+y\vert N_{t}^{\varepsilon}=0
\bigr)=\bbp \bigl(X_{t}(\varepsilon,\varnothing,x)\geq{}x+y\bigr)=\mathrm{O}
\bigl(t^{3}\bigr)
\]
for $\varepsilon$ small enough.
The expressions in (\ref{MPTN}) follows from the fact that,
\begin{eqnarray}\label{EsyRelUs}
    \lambda_{\varepsilon} \bbp\bigl [
\gamma(x,J)\geq y \bigr]&=&\int_{y}^{\infty}\lambda_{\varepsilon}
\Gamma_{\varepsilon}(\zeta;x)\,\mathrm{d}\zeta=
    \int_{\{\zeta:\gamma(x,\zeta)\geq y\}}
h(\zeta)\phi_{\varepsilon}(\zeta)\,\mathrm{d}\zeta
\nonumber
\\[-8pt]
\\[-8pt]
&=:&\int_{y}^{\infty} g_{\varepsilon}(x;\zeta)\,\mathrm{d}\zeta\nonumber
\end{eqnarray}
for some function $g_{\varepsilon}(x;\zeta)$. Thus, for fixed $x\in\bbr$ and $y>0$,
\begin{equation}\label{EsyRelUs2}
    {\lambda_{\varepsilon} \Gamma_{\varepsilon}(y;x)=g_{\varepsilon}(x;y)}.
\end{equation}
Furthermore, by differentiation of the last equality in (\ref{EsyRelUs}) and using that $\gamma(x,0)=0$,
it follows that, for $\varepsilon>0$ small enough, $g_{\varepsilon}(x;y)$ admits the representation on
the right-hand side of (\ref{DfnLvyDty0}). Using  (\ref{EsyRelUs})--(\ref{EsyRelUs2}), it then follows
that
\begin{eqnarray*}
\lambda_{\varepsilon} H_{0,0}(x;y)&=&\int_{y}^{\infty}g(x;
\zeta )\,\mathrm{d}\zeta,
\\
\lambda_{\varepsilon} \bigl[\hat{H}_{0,1}(x;y)+\hat
{H}_{1,0}(x;y) \bigr] &=&\calJ_{1}(x;y),
\\
\lambda_{\varepsilon} \bigl[{H}_{0,1}(x;y)+{H}_{1,0}(x;y)
\bigr]& =&\calD(x;y)+\calJ_{1}(x;y), \\
\lambda_{\varepsilon}^{2}
\bigl[H_{2,0}(x;y) - 2 H_{0,0}(x;y) \bigr]&=&
\calJ_{2}(x;y),
\end{eqnarray*}
{with $\calD(x;y)$, $\calJ_{1}(x;y)$, and $\calJ_{2}(x;y)$ given as in
the statement of the theorem.}
This concludes the result of Theorem \ref{ThTail}.
\section{Proof of the expansion for the transition densities}\label{SecMnThPr2}

The following result will allow us to control the higher-order terms
{of the expansion (\ref{FndDcmp})} (see Appendix \ref{SecTcnLm} for
its proof):
%
\begin{lmma}\label{TLUOT}
Let\vspace*{-1pt}
%
%
\begin{equation}
\label{UpRmd} \bar{\calR}_{t}(x,y):=\mathrm{e}^{-\lambda_{\varepsilon}t} \sum
_{n=3}^{\infty
}\bbp\bigl(X_{t}(x)
\geq{}x+y\vert N_{t}^{\varepsilon}=n\bigr) \frac{(\lambda_{\varepsilon}t)^{n}}{n!}.
\end{equation}
Then, under the conditions of Theorem \ref{ThDsty}, there exists
$\varepsilon>0$ small enough as well as $t_{0}:=t_{0}(\varepsilon)>0$
and $B=B(\varepsilon)<\infty$ such that, for any $0<t<{}t_{0}$,\vspace*{-1pt}
\[
\bigl\llvert \partial_{y}\bar{\calR}_{t}(x,y)\bigr\rrvert
\leq{} B t^{3}.
\]
\end{lmma}

\begin{pf*}{Proof of Theorem \ref{ThDsty}}
Let us consider the terms corresponding to one and two ``large'' jumps
in (\ref{FndDcmp}). From (\ref{Cnd1Jmp}), (\ref{Dcmp3T1}), (\ref
{Dcmp2}), and (\ref{Dcmp3}), it follows that\vspace*{-1pt}
%
%
\begin{eqnarray}\label{Eq1Rmd}
&&\bbp\bigl(X_{t}(x)\geq{}x+y|N_{t}^{\varepsilon}=1\bigr)\nonumber\\
&&\quad=
H_{0,0}(x;y)+\frac
{t}{2} \bigl[H_{0,1}(x;y)+H_{1,0}(x;y)
\bigr]
\\
&&\qquad{} +\frac{1}{t}\int_{0}^{t} \bigl
\{{s^{2}} \calR _{s}^{2}(x;y)+(t-s)s
\calR_{s}^{3}(x;y)+{(t-s)^{2} \bbe\calR
^{1}_{t-s} \bigl(X_{s}(\varepsilon,
\varnothing,x);x,y \bigr)} \bigr\} \,\mathrm{d}s.\nonumber
\end{eqnarray}
Similarly, from (\ref{ScndOrdEq2}), (\ref{ScndOrdEq3}), and (\ref
{ExNdLt}), we have\vspace*{-1pt}
%
%
\begin{eqnarray}\label{Eq2Rmd}
&&\bbp\bigl(X_{t}(x)\geq{}x+y|N_{t}^{\varepsilon}=2\bigr)
\nonumber\\
&&\quad=
H_{2,0}(x;y) +\frac{2}{t^{2}}\int_{0}^{t}\int
_{s_{1}}^{t} \bigl\{s_{1}
\calR_{s_{1}}^{5}(x;y)+(s_{2}-s_{1})\bbe
\calR ^{6}_{s_{2}-s_{1}} \bigl(X_{s_{1}}(\varepsilon,
\varnothing,x);x,y \bigr)\qquad
\\[-1pt]
&&\hspace*{105pt}\quad{} +(t-s_{2})\bbe \calR^{4}_{t-s_{2}}
\bigl(X_{s_{2}}\bigl(\varepsilon,\{s_{1}\},x\bigr);x,y \bigr)
\bigr\} \,\mathrm{d}s_{2}\,\mathrm{d}s_{1}.
\nonumber
\end{eqnarray}
Equations (\ref{Eq1Rmd})--(\ref{Eq2Rmd}) show that in order for the derivatives\vspace*{-1pt}
\[
\hat{a}_{1}(x;y):=\frac{\partial}{\partial y}\bbp\bigl(X_{t}(x)\geq
{}x+y|N_{t}^{\varepsilon}=1\bigr), \qquad\hat{a}_{2}(x;y):=
\frac{\partial}{\partial y}\bbp\bigl(X_{t}(x)\geq {}x+y|N_{t}^{\varepsilon}=2
\bigr)
\]
to exist, it suffices that the partial {derivatives with respect to
$y$} of the functions $H_{i,j}(x;y)$ exist and also that the partial
{derivatives with respect to $y$} of the two types of functions,
{$\calR
_{t}^{i}(x;y)$ with $i=2,3,5$ and $\calR_{t}^{j}(w;x,y)$ with
$j=1,4,6$}, exist and are uniformly bounded {on $w\in\bbr$ and on a
neighborhood of $y$}. Furthermore, under the later boundedness
property, we will then {be able to conclude} that\vspace*{-1pt}
%
%
\begin{eqnarray}
\label{DrvA1} \hat{a}_{1}(x;y)&=& \frac{\partial H_{0,0}(x;y)}{\partial y}+\frac{t}{2}
\biggl[\frac{\partial H_{0,1}(x;y)}{\partial y}+\frac{\partial
H_{1,0}(x;y)}{\partial y} \biggr]+\mathrm{O}\bigl(t^{2}
\bigr)\qquad (t\to{}0),
\\[-1pt]
\hat{a}_{2}(x;y)&=& \frac{\partial H_{2,0}(x;y)}{\partial y}+\mathrm{O}(t)\qquad (t\to{}0).\label{DrvA2}
\end{eqnarray}
Note that {(\ref{DrvA1})--(\ref{DrvA2})} suffices to {obtain the
conclusion of the theorem, namely equation (\ref{SOExp}),} in light of
(\ref
{FndDcmp}), Theorem \ref{densityest-main}, and Lemma \ref{TLUOT}. {We
now proceed to verify the differentiability of the functions
$H_{i,j}(x,y)$ and the remainder terms.}

(1) \emph{Differentiability of $H_{i,j}(x;y)$}:
The {desired} differentiability essentially follows from Lemma \ref
{LmED}. Indeed, Lemma \ref{LmED}(2) implies that $\partial_{y}
H_{0,0}(x;y)=\partial_{y}\bbp [ \gamma(x,J)\geq{}y
]=-\Gamma
(y;x)$ and also, recalling the formula of $H_{0,1}(x,y)$ given in
equations~(\ref{NETBP})--(\ref{MPTN1}),
\begin{eqnarray*}
\partial_{y} H_{0,1}(x;y)&:=&\frac{\sigma^{2}(x)}{2} \biggl(-
\frac{\partial^{2}\Gamma(y; x)}{\partial x^{2}} +2\frac{\partial^{2} \Gamma(y;x)}{\partial y\, \partial x} -\frac{\partial^{2}\Gamma(y; x)}{\partial y^{2}} \biggr)
\\
&&{} +{b_{\varepsilon}(x)} \biggl(- \frac{\partial\Gamma(y; x)}{\partial x} +\frac{\partial\Gamma(y;x)}{\partial y} \biggr)
\\
& &{}+\int (\Gamma(y;x)-\Gamma\biggl(y-\gamma(x,\zeta);x+\gamma (x,\zeta )
\\
&&\hspace*{80pt}{} -\gamma(x,\zeta) \biggl(\frac{\partial\Gamma(y;x)}{\partial
y}-\frac
{\partial\Gamma(y;x)}{\partial x} \biggr)
\biggr)\bar {h}_{\varepsilon
}(\zeta) \,\mathrm{d}\zeta.
\end{eqnarray*}
Similarly, recalling the definition of $H_{1,0}(x;y)$ given in (\ref
{EqDfnH10a}),
\begin{eqnarray*}
\partial_{y} H_{1,0}(x;y)&:=&{\partial_{y} \bigl(
\Gamma (y;x){b_{\varepsilon}(x+y)}-(\partial_{\zeta}\Gamma ) (y;x)v(x+y)-
\Gamma (y;x)v'(x+y) \bigr)}
\\
&&{} + \int \bigl(\Gamma (y;x )-\Gamma \bigl(\bar\gamma (x+y,\zeta)-x; x \bigr)
\partial_{y}\bar{\gamma}(x+y,\zeta)
\\
&&\hspace*{21pt}{} -{\partial_{y} \bigl(\Gamma(y;x)\gamma(x+y,\zeta ) \bigr)} \bigr)
\bar{h}_{\varepsilon}(\zeta) \,\mathrm{d}\zeta.
\end{eqnarray*}
To compute $\partial_{y} H_{2,0}(x;y)$, note that
\begin{eqnarray*}
\frac{\partial}{\partial y} H_{2,0}(x;y)&=&\frac{\partial}{\partial y} \int\bbp \bigl(
\gamma\bigl(x+\gamma(x,\zeta_{1}),J_{2}\bigr)\geq{}y-\gamma
(x,\zeta_{1}) \bigr){h}_{\varepsilon}(\zeta_{1})\,\mathrm{d}
\zeta_{1}
\\
&=& \int\frac{\partial}{\partial y}\int_{y-\gamma(x,\zeta
_{1})}^{\infty} \Gamma
\bigl(\zeta_{2};x+\gamma(x,\zeta_{1}) \bigr)\,\mathrm{d}\zeta
_{2}{h}_{\varepsilon}(\zeta_{1})\,\mathrm{d}\zeta_{1}
\\
&=& - \int\Gamma \bigl(y-\gamma(x,\zeta_{1});x+\gamma(x,\zeta
_{1}) \bigr){h}_{\varepsilon}(\zeta_{1})\,\mathrm{d}
\zeta_{1},
\end{eqnarray*}
where the second equality above {again follows from Lemma \ref
{LmED}(2)}. Finally, the representations in (\ref{MPTNDsty}) can be
deduced for $\varepsilon$ small enough from the relationships (\ref
{EsyRelUs})--(\ref{EsyRelUs2}).

(2) \emph{Boundedness of $\partial_{y} \mathcal{R}^{i}(w;x,y)$}:
Analyzing the remainder terms $\calR^{2}(x;y)$, $\calR^{3}_{t}(x;y)$,
$\calR^{5}_{t}(x;y)$, and $\calR^{6}_{t}(w;x,y)$, it transpires that it
suffices to show that $\partial_{y} L^{2}_{\varepsilon}H_{0}(w;x+y)$,
$\partial_{y} L_{\varepsilon}H_{0}(w;x+y)$, $\partial_{y}
L_{\varepsilon
}H_{1}(w;x+y)$, and $\partial_{y} L_{\varepsilon}H_{2}(w;x+y)$ exist
and are uniformly bounded in $w$ and $y$. From the definition of
$L_{\varepsilon}$ in (\ref{InfGenSmallJumps}), one can see that, for
any function $H(w;y)\dvtx \bbr^{2}\to\bbr$ in $C^{\infty}_{b}(\bbr^{2})$,
$\partial_{y}(L_{\varepsilon} H(w;y))$ exists and
\begin{eqnarray*}
\partial_{y}\bigl(L_{\varepsilon} H(w;y)\bigr)= L_{\varepsilon} (
\partial _{y}H) (w;y),\qquad \sup_{w,y}\bigl|
\partial_{y}L_{\varepsilon} H(w;y)\bigr|<\infty.
\end{eqnarray*}
{From Lemma \ref{LmED}(4)} and the relationship (\ref{IntStpDffH1}),
one can verify that $H_{0}(w;x+y),H_{1}(w;x+y),H_{2}(w;x+y)$ are
$C_{b}^{\infty}$ functions.

In order to show that $\partial_{y} \calR^{1}_{t}(w;x,y)$ and
$\partial
_{y}\calR^{4}_{t}(w;x,y)$ exist and are bounded, it suffices that the
remainder term $\breve{\calR}_{t}(z;\vartheta)$ of (\ref
{2ndExpOLJ}) is
differentiable with respect to $\vartheta$ and $\partial_{\vartheta
}\breve{\calR}_{t}(z;\vartheta)$ is bounded. The remainder term is
defined as in (\ref{DfnRmdCT}), which in turn is defined as the limit
as $\delta\to{}0$ of each of the four terms in (\ref{DcmCTDfnRmdTr}).
We will show that the limit {as $\delta\to{}0$} of the second term,
{which was therein denoted by $\bar{I}_{t}^{(2)}(z;\vartheta,\delta
,\varepsilon)$,} is indeed differentiable with respect to $\vartheta$
and its derivative is bounded. The other three terms can be dealt with
similarly. As {shown in the proof of Lemma \ref{Lmm2MnLm} (see (\ref
{EqGenlimterm}) and arguments before)}, the limit of {the} second term
in (\ref{DcmCTDfnRmdTr}) {can be expressed as the sum of terms of the
form $\int_0^1(1-\alpha)\tilde{I}_{\alpha t}(\vartheta
;z,\varepsilon
)\,\mathrm{d}\alpha$, where $\tilde{I}_{\alpha t}(\vartheta;z,\varepsilon)$ takes
one of the four generic {terms} listed in (\ref{CmplGenTrm0}). So, we
only need to show that each of these terms is} differentiable with
respect to $w$ and that their respective derivatives are bounded. The
{latter facts} will follow from Lemma \ref{UEDSJ} together with the
same arguments leading {to (\ref{BndI1})}.
\end{pf*}

\section{Proofs of other lemmas and additional needed results}\label
{SecTcnLm}

The following result is needed in order to prove Lemma \ref{UEDSJ}.\vspace*{-1pt}
%
\begin{lmma}\label{AuxBndDffrm}
Assume {that} the conditions \textup{(C1)--(C4)} of Section~\ref{SectIntro} are enforced. Let $\Phi_{t}\dvtx x\to X_{t}(\varepsilon
,\varnothing,x)$ be the diffeomorphism associated with the solution of
the SDE (\ref{SE2}). Then, for any $k\geq{}1$, $T<\infty${,} and
compact $K\subset\bbr$,\vspace*{-1pt}
%
%
\begin{eqnarray}
\label{CndNdJcb2b} \sup_{t\in(0,T]}\sup_{\eta\in K}\bbe
\biggl(\biggl\llvert \frac{\mathrm{d}^{i}
\Phi
_{t}^{-1}}{\mathrm{d}\eta^{i}}(\eta)\biggr\rrvert ^{k}
\biggr)<\infty,\qquad i=1,2.
\end{eqnarray}
\end{lmma}
\begin{pf}
To simplify the notation, {we write {$\breve{X}(x)=\{\breve{X}_t(x)\}
_{t\in(0,T]}$} for $\{X_{t}(\varepsilon,\varnothing,x)\}_{t\geq0}$} and
{fix {$Y_t(x):=\breve{X}_{(T-t)-}(x)$ for $0\leq t<T$} and
$Y_{T}(x):=\breve{X}_{0}(x)=x$}. We follow a similar approach to that
in the proof of Lemma 3.1 in Ishikawa \cite{Ishikawa2001} based on
time-reversibility (see Section VI.4 in Protter \cite{Protter} for further
information). Recall that the time-reversal process of a c\'ad\'ag
process {$V=\{V_t\}_{0\leq t\leq T}$} is given by the c\'adl\'ag
process\vspace*{-1pt}
%
%
\begin{equation}
\label{reversal} {\overline{V}^T_t}={(V_{(T-t)-}-V_{T-}){
\mathbf1}_{0<t<T} +(V_0-V_{T-})\mathbf{1}_{t=T}}.
\end{equation}
Our main tool is Theorem VI.4.22 in Protter \cite{Protter}. The following
notation and definitions are useful for verifying the assumptions in
the theorem.

Throughout, {$\Phi_{t,T}(\cdot;\omega)\dvtx \bbr\to\bbr$} denotes the
diffeomorphisms defined by {$\Phi_{t,T}(x;\omega):=
X_{t,T}^{\varepsilon
}(x;\omega)$} where {$X_{t,T}^{\varepsilon}(x;\omega)$} is the unique
solution of the SDE\vspace*{-1pt}
%
%
\begin{eqnarray}
\label{flow} X_{t,T}^{\varepsilon}(x)&=&{x}+\int_t^T
\sigma\bigl(X_{t,u}^{\varepsilon
}(x)\bigr)\,\mathrm{d} W_u+\int
_t^T {b_{\varepsilon}\bigl(X_{t,u}^{\varepsilon
}(x)
\bigr)}\,\mathrm{d}u
\nonumber
\\[-8pt]
\\[-8pt]
\nonumber
&&{}+{\sum^{c}_{t< u\leq T}}\gamma
\bigl(X_{t,u^{-}}^{\varepsilon}(x),\Delta Z'_u
\bigr) ,
\end{eqnarray}
{where $\sum^c$ denotes the compensated sum.} The a.s. existence
of this diffeomorphisms is guaranteed from (\ref{NndegncyCnd}) as
stated in Remark \ref{JustNonDegCnd}.
As usual, {$\mathcal{F}_t=\mathcal{F}^0_t\vee\mathcal{N}$} and
{$\mathbb
{F}= (\mathcal{F}_t )_{0\leq t\leq T}$}, where {$\mathcal
{F}^0_t=\sigma\{W_u,Z'_u; u\leq t\}$ ($0\leq t\leq T$)} and $\mathcal
{N}$ are the $\bbp$-null sets of {$\mathcal{F}^0_T$}. {We also define
the backward filtration {$\tilde{\mathbb{H}}= (\mathcal
{H}^t
)_{0<t\leq T}$} by {$\mathcal{H}^t=\bigcap_{t<u\leq T}\bar{\mathcal
{F}}_{u}\vee\sigma\{\breve{X}_T\}$}, where {$(\bar{\mathcal
{F}}_{t})_{0\leq t\leq T}$} is defined analogously to {$ (\mathcal
{F}_t )_{0\leq t\leq T}$} {by} $W$ and $Z'$ replaced with their
reversal processes {$\bar{W}^{T}$} and {$\bar{Z'}^{T}$}.}

We are ready to show the assertions of the lemma. First, note that, by
the uniqueness of the solution of (\ref{flow}), {$\breve{X}_T(x)=\Phi
_{t,T}(\breve{X}_t(x))$}. Thus, {$\breve{X}_t(x)=\Phi
^{-1}_{t,T}(\breve
{X}_T(x))\in\mathcal{H}^{T-t}$} and, of course, {$\breve{X}_t(x)\in
\mathcal{F}_t$}, so that {$\sigma({\breve{X}_{t}(x)})\in\mathcal
{F}_t\wedge\mathcal{H}^{T-t}$}. Also, by It\^o's formula, the quadratic
covariation of {$W=\{W_t\}_{0\leq t\leq T}$} with {$\sigma(\breve
{X}):=\{\sigma(\breve{X}_t(x))\}_{0\leq t\leq T}$} is given by 
%
%
\begin{equation}
\label{quacov} \bigl[\sigma(\breve{X}),W \bigr]_t=\int
_0^t\sigma'\bigl(\breve
{X}_u(x)\bigr)\sigma\bigl(\breve{X}_u(x)\bigr)\,\mathrm{d}u=\int
_0^t\sigma'\bigl(Y_{T-u}(x)
\bigr)\sigma \bigl(Y_{T-u}(x)\bigr)\,\mathrm{d}u.
\end{equation}
Finally, recalling that {$W=\{W_t\}_{0\leq t\leq T}$} is {an} $(\mathbb
{F},\tilde{\mathbb{H}})$-reversible semimartingale (cf. Theorem~VI.4.20
in Protter \cite{Protter}), the assumptions of Theorem~VI.4.22 in
Protter \cite
{Protter} are satisfied with $\sigma(\breve{X})$ and $W$ in place of
$H$ and $Y$, respectively. By the theorem, we have
\begin{eqnarray*}
\overline{\int_0^{\cdot}\sigma\bigl(
\breve{X}_{u}(x)\bigr)\,\mathrm{d}W_u}^{
T}_{ t}+
\overline{ \bigl[\sigma(\breve{X}),W \bigr]}_{ t}^{
T}=\int
_0^t\sigma\bigl(\breve{X}_{T-u}(x)
\bigr)\,\mathrm{d}\bar{W}_u^T,
\end{eqnarray*}
or equivalently, by (\ref{quacov}) and the change of variable {$v=T-u$},
%
%
\begin{eqnarray}
\label{thm22} {\overline{\int_0^{\cdot}
\sigma\bigl(\breve {X}_{u-}(x)\bigr)\,\mathrm{d}W_u}
^{ T}_{ t}-\int_0^t
\sigma'\bigl(Y_v(x)\bigr)\sigma\bigl(Y_v(x)
\bigr)\,\mathrm{d}v=\int_0^t\sigma \bigl(Y_u(x)
\bigr)\,\mathrm{d}\overline{W}^T_u}.
\end{eqnarray}
Omitting for simplicity the dependence of the processes on $x$, the
first term on the left- hand side of (\ref{thm22}) can be written as
\begin{eqnarray*}
&&\overline{\breve{X}_{\cdot}-x-\int_0^{\cdot
}{b_{\varepsilon
}(
\breve{X}_{u-})}\,\mathrm{d}u-{\sum^c_{0<u\leq\cdot}}
\gamma\bigl(\breve {X}_{u-},\Delta Z'_{u}
\bigr)}^{ T}_{ t}
\\
&&\quad =\breve{X}_{(T-t)-}-\breve{X}_{T-}+\int_{T-t}^{T}b(
\breve {X}_u)\,\mathrm{d}u+{{\sum^c_{T-t\leq u< T}}
\gamma\bigl(\breve{X}_{u^{-}},\Delta Z'_{u}
\bigr)}
\\
&&\quad =Y_t-Y_0+\int_0^t
{b_{\varepsilon}({Y_{v}})}\,\mathrm{d}v+{{\sum^c_{0<
v\leq t}}
\gamma\bigl(\breve{X}_{(T-v)^{-}},\Delta Z'_{T-v}
\bigr)},
\end{eqnarray*}
where the last equality above is from the change of variable {$v=T-u$}.
Then, (\ref{thm22}) implies that
\begin{eqnarray*}
Y_t(x)&=&Y_0(x)-\int_0^t
{b_{\varepsilon}\bigl(Y_v(x)\bigr)}\,\mathrm{d}v+ \int_0^t
\sigma '\bigl(Y_v(x)\bigr)\sigma\bigl(Y_v(x)
\bigr)\,\mathrm{d}v+ \int_0^t\sigma\bigl(Y_v(x)
\bigr)\,\mathrm{d}\overline{W}^T_v
\\
&&{} -{\sum^c_{{0< v\leq t}}}\gamma\bigl(
\breve{X}_{(T-v)^{-}}(x),\Delta Z'_{T-v}\bigr),\qquad
Y_{0}(x)=\breve{X}_{T^{-}}(x).
\end{eqnarray*}
{Let us write the jump component of $Y$ in a more convenient way. To
this end, note {that, since} {$\breve{X}_{(T-v)^{-}}(x)+\gamma(\breve
{X}_{(T-v)^{-}}(x),\Delta Z'_{T-v})=\breve{X}_{T-v}(x)$}, one can
express {$\breve{X}_{(T-v)^{-}}(x)$} in terms of the inverse $\bar
{\gamma}(u,\zeta)$ of the mapping $z\to u:=z+\gamma(z,\zeta)$ as follows
\[
Y_{v}(x)=\breve{X}_{(T-v)^{-}}(x)=\bar{\gamma}\bigl(\breve
{X}_{T-v}(x),\Delta Z'_{T-v}\bigr)=\bar{\gamma}
\bigl(Y_{v^{-}}(x),\Delta Z'_{T-v}\bigr).
\]
Then,
\[
\Delta Y_{v}(x)=\bar{\gamma}\bigl(Y_{v^{-}}(x),\Delta
Z'_{T-v}\bigr)-Y_{v^{-}}(x)=\bar{\gamma}
\bigl(Y_{v^{-}}(x),-\Delta\bar {Z}'_{v}
\bigr)-Y_{v^{-}}(x)=
\gamma_{0}\bigl(Y_{v^{-}}(x),
\Delta\bar{Z}'_{v}\bigr),
\]
where $\gamma_0(u,\zeta):=\bar{\gamma}(u,-\zeta)-u$ and {$\bar
{Z}'_{v}:=\overline{Z'}^{ T}_{v}$} is the time-reversal process of
{$\{
Z'_{v}\}_{0\leq{}v\leq{}T}$}}. We conclude that
\begin{eqnarray*}
Y_t(x)&=&\breve{X}_{T^{-}}(x)-{\int_0^t
{b_{\varepsilon}\bigl(Y_v(x)\bigr)}\,\mathrm{d}v+ \int_0^t
\sigma'\bigl(Y_v(x)\bigr)\sigma\bigl(Y_v(x)
\bigr)\,\mathrm{d}v+ \int_0^t\sigma \bigl(Y_v(x)
\bigr)\,\mathrm{d}\overline{W}^T_v}
\\
& &{}+{\sum^c_{0< v\leq t}}\gamma_{0}
\bigl(Y_{v^{-}}(x),\Delta\bar{Z}'_{v}\bigr).
\end{eqnarray*}
Now, define the diffeomorphism {$\Psi_t\dvtx \bbr\to\bbr$} as {$\Psi
_{t}(\eta
):= \breve{Y}_{t}(\eta)$}, where {$\{\breve{Y}_{t}(\eta)\}_{0\leq
{}t\leq
{}T}$} is the solution of the SDE
\begin{eqnarray*}
\breve{Y}_{t}(\eta)&=& \eta-\int_0^t
{b_{\varepsilon}\bigl(\breve {Y}_v(\eta )\bigr)}\,\mathrm{d}v+\int
_0^t\sigma'\bigl(
\breve{Y}_v(\eta)\bigr)\sigma\bigl(\breve{Y}_v(\eta )
\bigr)\,\mathrm{d}v+\int_0^t\sigma\bigl(
\breve{Y}_v(\eta)\bigr)\,\mathrm{d}\overline{W}^T_v
\\
&&{} +{\sum^c_{0<v\leq t}}\gamma_0
\bigl(\breve{Y}_{v-}(\eta),\Delta \bar{Z}'_{v}
\bigr).
\end{eqnarray*}
Since, $\bbp$-a.s.,
\[
\Psi_{T}\bigl(\Phi_{T}(x)\bigr)=\Psi_{T}\bigl(
\breve{X}_{T}(x)\bigr)=\Psi_{T}\bigl(\breve
{X}_{T^{-}}(x)\bigr)=Y_{T}(x)=x\qquad \mbox{for all }x\in\mathbb
{R},T<\infty,
\]
it follows that, $\bbp$-a.s., {$\Psi_t(\eta)=\Phi^{-1}_{t}(\eta)$} for
all {$\eta\in\R$}.
Furthermore, $\{{\breve{Y}_t(\eta)}\}_{t\geq0}$ solves an SDE of the
form (6-2) in Bichteler, Gravereaux and Jacod \cite{BBG} with their
coefficients satisfying the
assumptions of Lemma 10-29 therein. Finally, by Lemma 10-29-c in
Bichteler, Gravereaux and Jacod \cite
{BBG}, with $n=2$ and $q=1$,
\begin{eqnarray*}
{\sup_{0< t\leq T}\sup_{\eta\in K}\bbe \biggl[\biggl
\llvert \frac
{d^{i}\Phi
_t^{-1}(\eta)}{d\eta^{i}}\biggr\rrvert ^k \biggr]=\sup
_{0< t\leq T}\sup_{\eta
\in K}\bbe \biggl[\biggl\llvert
\frac{\mathrm{d}^{i}\Psi_t(\eta)}{\mathrm{d}\eta^{i}}\biggr\rrvert ^k \biggr]=\sup_{0< t\leq T}
\sup_{\eta\in K}\bbe \biggl[ \biggl|\frac
{d^{i}\breve
{Y}_{t}(\eta)}{d\eta^{i}} \biggr|^k
\biggr]<\infty}
\end{eqnarray*}
for $i=1,2$.
\end{pf}
\begin{pf*}{Proof of Lemma \ref{UEDSJ}}
For simplicity, we write $\widetilde{\Gamma}(\zeta)=\widetilde
{\Gamma
}(\zeta;z)$ and only show the case $k=1$ (the other cases can similarly
be proved).
Using the same ideas as in the proof of Proposition I.2 in L\'eandre
\cite
{Leandre}, one can show that
\[
\int\widetilde{\Gamma}(\zeta)p_{t} (\eta;\varepsilon ,\varnothing ,
\zeta )\,\mathrm{d}\zeta =\bbe \bigl(H_{t}(\eta) \bigr),
\]
where
\[
{H_{t}(\eta):=\widetilde{\Gamma}\bigl(\Phi_{t}^{-1}(
\eta)\bigr)\frac{\mathrm{d} \Phi
_{t}^{-1}}{\mathrm{d}\eta}(\eta).} 
\]
Denoting {$\bar{J}_{t}(\eta):=d \Phi_{t}^{-1}(\eta)/d\eta$}, note that
\[
{H_{t}'(\eta)=\widetilde{\Gamma}'\bigl(
\Phi_{t}^{-1}(\eta)\bigr)\bar {J}_{t}(\eta
)^{2}+\widetilde{\Gamma}\bigl(\Phi_{t}^{-1}(\eta)
\bigr)\bar{J}_{t}'(\eta)},
\]
and, using (\ref{CndNdJcb2b}) and that $\widetilde{\Gamma}\in
C^{\infty
}_{b}$, it follows that
$\sup_{\eta\in K}\bbe\llvert H'_{t}(\eta)\rrvert ^{2}<\infty$. In
particular,
%
%
\begin{equation}
\label{CID0} \lim_{h\to{}0}\bbe \biggl( \frac{H_{t}(\eta+h)-H_{t}(\eta
)}{h} \biggr)
=\bbe \biggl(\lim_{h\to{}0} \frac{H_{t}(\eta+h)-H_{t}(\eta
)}{h} \biggr)=\bbe
H_{t}'(\eta),
\end{equation}
since the set of random variables $\{[H_{t}(\eta+h)-H_{t}(\eta)]/h\dvt |h|< 1\}$ is uniformly integrable. Indeed,
\[
\sup_{|h|\leq{}1}\bbe \biggl(\frac{H_{t}(\eta+h)-H_{t}(\eta
)}{h}
\biggr)^{2}=\sup_{|h|\leq{}1}\bbe \biggl(\int
_{0}^{1} H'_{t}(\eta+h
\beta )\,\mathrm{d}\beta \biggr)^{2}\leq \mathop{\sup_{|h|\leq{}1}}_{
\beta\in[0,1]}
\bbe \bigl( H'_{t}(\eta+h\beta) \bigr)^{2},
\]
which is finite in light of (\ref{CndNdJcb2b}).
Then, (\ref{CID0}) can be written as
\begin{eqnarray*}
{\frac{\mathrm{d}}{\mathrm{d}\eta}\int\widetilde{\Gamma}(\zeta)p_{t} (\eta ;
\varepsilon,\varnothing,\zeta )\,\mathrm{d}\zeta =\bbe \bigl(\widetilde{\Gamma}'
\bigl(\Phi_{t}^{-1}(\eta)\bigr) \bigl(\bar
{J}_{t}(\eta) \bigr)^{2} \bigr)+\bbe \bigl(\widetilde{
\Gamma }\bigl(\Phi _{t}^{-1}(\eta)\bigr)\bar{J}_{t}'(
\eta) \bigr)}.
\end{eqnarray*}
It is now clear that (\ref{WNADSJ2}) will hold true in light of (\ref
{CndNdJcb2b}).

{We now show the last assertion of the lemma. First note that,} {from
the non-negativity of $\tilde{\Gamma}$ and $p_t$, (\ref{WNADSJ2})
implies that there exist a constant $t_0>0$ small enough such that for
any $t<t_0$,
\begin{eqnarray*}
\sup_{z\in\mr}\sup_{\eta\in K}\int\bigl\llvert
\widetilde{\Gamma }(\zeta )p_t(\eta;\varepsilon,\varnothing,\zeta)
\bigr\rrvert \,\mathrm{d}\zeta<\infty,
\end{eqnarray*}
{and, thus,} $\widetilde{\Gamma}(\zeta)p_t(\eta;\varepsilon
,\varnothing
,\zeta)$ is uniformly integrable with respect to $\zeta$. {The latter
fact together with (\ref{WNADSJ2}) implies that}
\begin{eqnarray*}
\biggl\llvert \frac{\partial^k}{\partial\eta^k}\int\widetilde{\Gamma }(\zeta )
\frac{\partial p_t}{\partial\eta}(\eta;\varepsilon,\varnothing ,\zeta )\,\mathrm{d}\zeta\biggr\rrvert =
\biggl\llvert \frac{\partial^{k+1}}{\partial\eta
^{k+1}}\int \widetilde{\Gamma}(\zeta)p_t(
\eta,\varepsilon,\varnothing,\zeta )\,\mathrm{d}\zeta \biggr\rrvert <C
\end{eqnarray*}
for some $C>0$ and any $t<t_0$, $z\in\mr$ and $\eta\in K$. Then,
(\ref
{WNADSJ2}) is also true with $\partial p_t/\partial\eta$ in place of
$p_t$ {inside the integral of (\ref{WNADSJ2})}.}
\end{pf*}

\begin{lmma}\label{KPDOp}
Assume the conditions \textup{(C1)--(C4)} of Section~\ref{SectIntro} are
satisfied and let $\calD_\varepsilon$ and $\calI_{\varepsilon}$ be the
operators defined in (\ref{InfGenSmallJumps}). Define the {following
operators:}
\begin{eqnarray*}
\widetilde{\calD}_\varepsilon g(y)&:=&v(y)g''(y)+
\bigl(2v'(y)-b(y) \bigr)g'(y)+ \bigl(v''(y)-b'(y)
\bigr)g(y),
\\
\widetilde{\calI}_{\varepsilon}g(y)&:=&{\int \bigl(g\bigl(\bar{\gamma }(y,\zeta
)\bigr)\partial_{y}\bar{\gamma}(y,\zeta) -\bigl(1+
\partial_{y}\gamma(y,\zeta)\bigr)g(y)-g'(y)\gamma(y,
\zeta) \bigr)\bar {h}_{\varepsilon}(\zeta)\,\mathrm{d}\zeta},
\\
\widetilde{\calH}_{\varepsilon}g(y)&:=&{\int \biggl(\int_{\bar
\gamma
(y,\zeta)}^{y}g(
\eta)\,\mathrm{d}\eta-g(y) \gamma(y,\zeta) \biggr)\bar {h}_{\varepsilon}(\zeta)\,\mathrm{d}\zeta},
\end{eqnarray*}
where hereafter $\bar{\gamma}(u,\zeta)$ denotes the inverse of the
mapping $y\to u:=y+\gamma(y,\zeta)$ for a fixed $\zeta$ and whose
existence is guaranteed from condition \textup{(C4)}. Then, the following
assertions hold:
\begin{enumerate}
\item$\widetilde{\calD}_{\varepsilon}g$ is well defined and uniformly
bounded for any $g\in C^{2}_{b}$ and, furthermore, for any $f\in
C^{2}_{b}$ with compact support,
%
%
\begin{equation}
\label{DlDfOp} \int g(y)\calD_{\varepsilon} f(y)\,\mathrm{d}y=\int f(y)\widetilde{\calD
}_{\varepsilon}g(y)\,\mathrm{d}y.
\end{equation}
\item$\widetilde{\calI}_{\varepsilon}g$ is well defined and uniformly
bounded for any $g\in C^{1}_{b}$ and, additionally, if $g$ is
integrable, then, for any $f\in C^{1}_{b}$ {with compact support},
%
%
\begin{equation}
\int g(y)\calI_{\varepsilon} f(y)\,\mathrm{d}y=\int f(y)\widetilde{\calI
}_{\varepsilon}g(y)\,\mathrm{d}y.\label{DlInOp1}
\end{equation}
\item For any {$g\in C^{1}_{b}$ and $f\in C^{1}_{b}$ such that {$f'$
and $f''$} are integrable},
%
%
\begin{equation}
\int g(y)\calI_{\varepsilon} f(y)\,\mathrm{d}y=\int f'(y) {\widetilde{\calH
}_{\varepsilon}g(y)} \,\mathrm{d}y.\label{DlInOp2}
\end{equation}
\end{enumerate}
\end{lmma}
\begin{pf}
The dual relationships essentially follow from a combination of
integration by parts and change of variables. Let us show (\ref
{DlInOp2}). {First, we show that $\calI_{\varepsilon} f(y)$ is
integrable and, thus, the left-hand side of equation (\ref{DlInOp2}) is well
defined. To this end, we write $\calI_{\varepsilon} f(y)$ as
\begin{eqnarray*}
\calI_{\varepsilon} f(y)&=& \int\int_{0}^{1}
\bigl(f''\bigl(y+\gamma (y,\zeta \beta)\bigr) \bigl(
\partial_{\zeta} \gamma(y,\zeta\beta) \bigr)^{2}+
f'\bigl(y+\gamma(y,\zeta\beta)\bigr)\partial^{2}_{\zeta}
\gamma(y,\zeta \beta )
\\
&&\hspace*{28pt}{} -f'(y)\partial^{2}_{\zeta} \gamma(y,
\zeta \beta ) \bigr) (1-\beta) \,\mathrm{d}\beta\bar{h}_\varepsilon(\zeta)
\zeta^{2} \,\mathrm{d}\zeta.
\end{eqnarray*}
Since $\gamma\in C^{\geq{}1}_{b}$, it is now evident that $\int\llvert \calI_{\varepsilon} f(y)\rrvert  \,\mathrm{d}y<\infty$ provided that
$\int\llvert f^{(k)}(y+ \gamma(y,\zeta\beta))\rrvert \,\mathrm{d}y<\infty$ for
$k=1,2$. To verify the latter fact, note that, by changing variables
from $y$ to $w:=\tilde\gamma(y,\zeta\beta)=y+\gamma(y,\zeta\beta)$,
\begin{eqnarray*}
\int\bigl\llvert f^{(k)}\bigl(y+\gamma(y,\zeta\beta)\bigr)\bigr\rrvert
\,\mathrm{d}y =\int\bigl\llvert f^{(k)}(w)\bigr\rrvert \frac{1}{|1+(\partial_{y} \gamma)(\bar\gamma
(w,\beta
\zeta),\zeta\beta)|}\,\mathrm{d}w <
\infty,
\end{eqnarray*}
due to (\ref{NndegncyCnd}).

Once we have show that $\calI_{\varepsilon} f(y)$ is integrable, we now
prove the equality in equation (\ref{DlInOp2}).}
{Let us first note that
%
%
\begin{eqnarray}
\label{DlRlF1}&& \int g(y)\calI_{\varepsilon} f(y)\,\mathrm{d}y
\nonumber
\\[-8pt]
\\[-8pt]
\nonumber
&&\quad=\lim_{\delta\to{}0}
\int g(y)\int_{|\zeta|\geq{}\delta} \bigl(f\bigl(y+\gamma(y,\zeta )
\bigr)-f(y)-f'(y)\gamma(y,\zeta) \bigr)\bar{h}_{\varepsilon}(\zeta
)\,\mathrm{d}\zeta \,\mathrm{d}y.
\end{eqnarray}
For each $\delta>0$, fix
\begin{eqnarray*}
A_{\delta}=\int g(y)\int_{|\zeta|\geq{}\delta} \bigl(f\bigl(y+\gamma
(y,\zeta )\bigr)-f(y) \bigr)\bar{h}_{\varepsilon}(\zeta)\,\mathrm{d}\zeta \,\mathrm{d}y,
\end{eqnarray*}
and note that
\[
A_{\delta}=\int\int_{|\zeta|\geq{}\delta}\int_{0}^{1}
g(y) f'\bigl(y+\gamma (y,\zeta\beta)\bigr) (\partial_{\zeta}
\gamma) (y,\zeta\beta) \,\mathrm{d}\beta \bar {h}_\varepsilon(\zeta) \zeta \,\mathrm{d}\zeta \,\mathrm{d}y.
\]
Changing variable from $y$ to $w:=\tilde\gamma(y,\zeta\beta
)=y+\gamma
(y,\zeta\beta)$ and applying Fubini, we get
\begin{eqnarray*}
A_{\delta}=\int f'(w) \int_{|\zeta|\geq\delta}\int
_{0}^{1} g\bigl(\bar \gamma(w,\zeta\beta)\bigr)
\frac{(\partial_{\zeta} \gamma)(\bar\gamma
(w,\beta
\zeta),\zeta\beta)}{1+(\partial_{y} \gamma)(\bar\gamma(w,\beta
\zeta
),\zeta\beta)} \,\mathrm{d}\beta\zeta\bar{h}_\varepsilon(\zeta)\,\mathrm{d}\zeta \,\mathrm{d}w.
\end{eqnarray*}
From the identity
\[
\partial_{\zeta}\int_{\bar\gamma(w,\zeta)}^{w} g(\eta)\,\mathrm{d}
\eta =-g\bigl(\bar \gamma(w,\zeta)\bigr)\partial_{\zeta}\bar\gamma(w,
\zeta)=g\bigl(\bar \gamma(w,\zeta )\bigr)\frac{(\partial_{\zeta} \gamma)(\bar\gamma(w,\zeta),\zeta
)}{1+(\partial_{y} \gamma)(\bar\gamma(w,\zeta),\zeta)},
\]
we can then write
\[
A_{\delta}=\int f'(w) \int_{|\zeta|\geq\delta}\int
_{\bar\gamma
(w,\zeta
)}^{w} g(\eta)\,\mathrm{d}\eta\bar{h}_\varepsilon(
\zeta)\,\mathrm{d}\zeta \,\mathrm{d}w.
\]
Plugging the previous formula in (\ref{DlRlF1}), we get
\[
\int g(y)\calI_{\varepsilon} f(y)\,\mathrm{d}y =\lim_{\delta\to{}0} \int
f'(y) \int_{|\zeta|\geq\delta} \biggl(\int
_{\bar\gamma
(y,\zeta
)}^{y} g(\eta)\,\mathrm{d}\eta-\gamma(y,\zeta)g(y) \biggr)
\bar {h}_{\varepsilon
}(\zeta)\,\mathrm{d}\zeta \,\mathrm{d}y.
\]
Let
\[
B_{\delta}(y):=\int_{|\zeta|\geq\delta}C(y,\zeta)\bar
{h}_{\varepsilon
}(\zeta)\,\mathrm{d}\zeta \qquad\mbox{with } C(y,\zeta):=\int
_{\bar\gamma(y,\zeta)}^{y} g(\eta)\,\mathrm{d}\eta-\gamma (y,\zeta)g(y),
\]
and note that, for $g\in C^{1}_{b}$,
%
%
\begin{equation}
\label{Eq2ndDerC} \partial_{\zeta}^{2} C(y,
\zeta)=-g'\bigl(\bar{\gamma}(y,\zeta)\bigr) \bigl(
\partial_{\zeta}\bar{\gamma}(y,\zeta) \bigr)^{2} -g\bigl(\bar{
\gamma}(y,\zeta)\bigr)\partial^{2}_{\zeta}\bar{\gamma }(y,
\zeta) -g(y)\partial^{2}_{\zeta}\gamma(y,\zeta),
\end{equation}
is bounded in light of Lemma \ref{LmED}(4). Then, writing
\[
\int f'(y) B_{\delta}(y)\,\mathrm{d}y=\int f'(y) \int
_{|\zeta|\geq\delta}\int_{0}^{1}
\partial^{2}_{\zeta}C(y,\zeta\beta) (1-\beta)\,\mathrm{d}\beta\zeta
^{2}\bar {h}_{\varepsilon}(\zeta)\,\mathrm{d}\zeta \,\mathrm{d}y,
\]
it is clear that, when $f'$ is integrable,
\begin{eqnarray*}
\lim_{\delta\to{}0}\int f'(y) B_{\delta}(y)\,\mathrm{d}y&=&
\int f'(y) \lim_{\delta
\to{}0} B_{\delta}(y)\,\mathrm{d}y
\\
&=& \int f'(y) \int \biggl(\int_{\bar\gamma(y,\zeta)}^{y}
g(\eta)\,\mathrm{d}\eta -\gamma (y,\zeta)g(y) \biggr)\bar{h}_{\varepsilon}(\zeta)\,\mathrm{d}\zeta \,\mathrm{d}y,
\end{eqnarray*}
which implies (\ref{DlInOp2}).}
\end{pf}

\begin{pf*}{Proof of Lemma \ref{TLUOT}}
By conditioning on the times of the jumps, which are necessarily
distributed as the order statistics of $n$ independent uniform $[0,t]$
random variables, we have
\begin{eqnarray*}
\bbp\bigl(X_{t}(x)\geq{}x+y\vert
N_{t}^{\varepsilon}=n\bigr)=\frac
{n!}{t^{n}}\int
_{\Delta}\bbp\bigl(X_{t}\bigl(\varepsilon,
\{s_{1},\ldots ,s_{n}\} ,x\bigr)\geq{}x+y
\bigr)\,\mathrm{d}s_{n}\cdots \,\mathrm{d}s_{1},
\end{eqnarray*}
where $\Delta:=\{(s_{1},\ldots,s_{n})\dvt 0<{}s_{1}<s_{2}<\cdots
<s_{n}<{}t\}
$. Hence, conditioning on $\msf_{s^{-}_{n}}$,
\begin{eqnarray*}
\bbp\bigl(X_{t}\bigl(\varepsilon,\{s_{1},
\ldots,s_{n}\},x\bigr)\geq{}x+y\bigr)&=&\bbe \bigl[\bbp
\bigl(X_{t}\bigl(\varepsilon,\{s_{1},\ldots,s_{n}
\},x\bigr)\geq{}x+y | {\msf_{s^{-}_{n}}} \bigr) \bigr]
\\
& =&\bbe \bigl[ G_{t-s_{n}} \bigl(X_{s_{n}}\bigl(\varepsilon,\{
s_{1},\ldots ,s_{n-1}\},x\bigr);x,y \bigr) \bigr],
\end{eqnarray*}
where $G_{t}(z;x,y)=\bbp (X_{t} (\varepsilon,\varnothing
,z+\gamma
(z,J) )\geq{}x+y )$.
In terms of the densities $p_{t}(\cdot;\varepsilon,\varnothing,\zeta)$
and $\widetilde{\Gamma}(\cdot;z)$ of $X_{t}(\varepsilon,\varnothing
,\zeta
)$ and $z+\gamma(z,J)$, respectively, we have that
\begin{eqnarray*}
G_{t}(z;x,y)&=&\int\int_{x+y}^{\infty}p_{t}(
\eta;\varepsilon ,\varnothing ,\zeta)\,\mathrm{d}\eta\widetilde\Gamma(\zeta;z) \,\mathrm{d}\zeta\\
&=& \int
_{x+y}^{\infty
}\int p_{t}(\eta;\varepsilon,
\varnothing,\zeta) \widetilde\Gamma(\zeta;z) \,\mathrm{d}\zeta \,\mathrm{d}\eta.
\end{eqnarray*}
From Lemma \ref{UEDSJ}, we know that there exists $\varepsilon$ small
enough such that, for any $\delta>0$, there exists
$B:=B(\varepsilon,\delta)<\infty$ and $t_{0}:=t_{0}(\varepsilon
,\delta
)>0$ {for which}
%
%
\begin{equation}
\label{UnifBnd3} \sup_{z\in\bbr}\sup_{\eta\in[x+y-\delta,x+y+\delta]}\int
p_{t} (\eta;\varepsilon,\varnothing,\zeta ) \widetilde{\Gamma }(
\zeta;z) \,\mathrm{d}\zeta\leq B
\end{equation}
for all $0<t<t_{0}$.
The {uniform bound (\ref{UnifBnd3})} allows us to interchange the
differentiation and the other relevant operations (integration,
expectation, etc.) so that
\[
\mathcal{G}_{t}^{(n)}(x,y):=\partial_{y}\bbp
\bigl(X_{t}(x)\geq {}x+y\vert N_{t}^{\varepsilon}=n
\bigr)
\]
can be written as
\begin{eqnarray*}
\mathcal{G}_{t}^{(n)}(x,y)&=&\frac{n!}{t^{n}}\int
_{\Delta}\partial _{y}\bbp\bigl(X_{t}\bigl(
\varepsilon,\{s_{1},\ldots,s_{n}\},x\bigr)\geq {}x+y
\bigr)\,\mathrm{d}s_{n}\cdots \,\mathrm{d}s_{1}
\\
&=&\frac{n!}{t^{n}}\int_{\Delta}\bbe \bigl[ \partial
_{y}G_{t-s_{n}} \bigl(X_{s_{n}}\bigl(\varepsilon,
\{s_{1},\ldots,s_{n-1}\},x\bigr);x,y \bigr)
\bigr]\,\mathrm{d}s_{n}\cdots \,\mathrm{d}s_{1}
\\
&=&\frac{n!}{t^{n}}\int_{\Delta}\bbe \biggl[ \int
p_{t-s_{n}}(x+y;\varepsilon,\varnothing,\zeta) \widetilde\Gamma \bigl(
\zeta; X_{s_{n}}\bigl(\varepsilon,\{s_{1},\ldots,s_{n-1}
\},x\bigr)\bigr) \,\mathrm{d}\zeta \biggr]\,\mathrm{d}s_{n}\cdots \,\mathrm{d}s_{1}
\end{eqnarray*}
and also, for any $0<t<t_{0}$,
\[
\bigl|\partial_{y}\bbp\bigl(X_{t}(x)\geq{}x+y\vert
N_{t}^{\varepsilon
}=n\bigr)\bigr|\leq B.
\]
Using this {bound,}
\begin{eqnarray*}
\bigl\llvert \partial_{y}\bar{\calR}_{t}(x,y)\bigr\rrvert
&\leq& \mathrm{e}^{-\lambda
_{\varepsilon}t} \sum_{n=3}^{\infty}
\bigl\llvert \partial_{y}\bbp\bigl(X_{t}(x)\geq{}x+y
\vert N_{t}^{\varepsilon}=n\bigr) \bigr\rrvert \frac{(\lambda_{\varepsilon}t)^{n}}{n!}
\\
&\leq& B \mathrm{e}^{-\lambda_{\varepsilon}t} \sum_{n=3}^{\infty}
\frac{(\lambda_{\varepsilon}t)^{n}}{n!}\leq B\lambda_{\varepsilon
}^{3} t^{3}.
\end{eqnarray*}
The proof is then complete.
\end{pf*}

\begin{pf*}{Proof of Lemma \ref{l1}}
By conditioning on the times of the jumps, which are necessarily
distributed as the order statistics of $n$ independent uniform $[0,t]$
random variables, we have
\begin{eqnarray*}
\bbp\bigl(|X_{t}-x|\geq{}\log y\vert N_{t}^{\varepsilon
}=n
\bigr)=\frac
{n!}{t^{n}}\int_{\Delta}\bbp\bigl(\bigl|X_{t}
\bigl(\varepsilon,\{s_{1},\ldots ,s_{n}\} ,x\bigr)-x\bigr|\geq{}
\log y\bigr)\,\mathrm{d}s_{n}\cdots \,\mathrm{d}s_{1},
\end{eqnarray*}
where $\Delta:=\{(s_{1},\ldots,s_{n})\dvt 0<{}s_{1}<s_{2}<\cdots
<s_{n}<{}t\}
$. Hence, we only need to bound
\[
\sup_{n\in\mn, t\in[0,1]}\frac{1}{n!}\int_0^\infty
\mp \bigl(\bigl|X_t\bigl(\varepsilon ,\{s_1,\ldots,s_n
\},x\bigr)-x\bigr|\geq\log y\bigr)\,\mathrm{d}y
\]
uniformly.
By conditioning again,
\begin{eqnarray*}
&&\bbp\bigl(\bigl|X_{t}\bigl(\varepsilon,\{s_{1},
\ldots,s_{n}\},x\bigr)-x\bigr|\geq{}\log y\bigr)
\\
&&\quad=\bbe \bigl[\bbp
\bigl(\bigl|X_{t}\bigl(\varepsilon,\{s_{1},\ldots,s_{n}
\},x\bigr)-x\bigr|\geq {}\log y | {\msf_{s^{-}_{n}}} \bigr) \bigr]
\\
&&\quad \leq\bbe \bigl[\bbp \bigl[\bigl|X_{t-s_{n}}(\varepsilon ,\varnothing,z)-x\bigr|+\bigl|
\gamma(z,J)\bigr|\geq{}\log y \bigr] |_{z=X_{s_{n}}(\varepsilon,\{s_{1},\ldots,s_{n-1}\},x)} \bigr].
\end{eqnarray*}
{Recall the condition (C5)}, we have for some constant $M>0$ and
all $\lambda\leq3$
\[
\sup_x\me \mathrm{e}^{\lambda|\gamma(x,J)|}\sup_x\leq C
\int \mathrm{e}^{|3\gamma
(x,z)|}h(z)\,\mathrm{d}z\leq M<\infty.
\]
Now fix any positive constant $A$ and {$t\leq{}1$}, we have
\begin{eqnarray*}
\me \mathrm{e}^{|X_{t}(\varepsilon,\{s_{1},\ldots,s_{n}\},x)-x|} &=&\int_0^A\mp\bigl
\{\bigl|X_{t}\bigl(\varepsilon,\{s_{1},\ldots,s_{n}\},x
\bigr)-x\bigr|>\log y\bigr\} \,\mathrm{d}y
\\
&&{} +\int_A^\infty\mp\bigl\{\bigl|X_{t}\bigl(
\varepsilon,\{s_{1},\ldots,s_{n}\} ,x\bigr)-x\bigr|>\log y\bigr
\}\,\mathrm{d}y
\\
&\leq& A+{2Me^{({1}/{2})\lambda_1^2 k(1+\exp(\lambda_1\varepsilon)
)}}\frac{1}{A^{\alpha}}\frac{1}{\alpha} \bigl(\me
\mathrm{e}^{\lambda
_1|X_{s_{n}}(\varepsilon,\{s_{1},\ldots,s_{n-1}\},x)-x|} \bigr).
\end{eqnarray*}
{Above}, we used (\ref{martingale-est}) for the last inequality with
$\lambda=\lambda_1=1+\alpha$, where $0<\alpha<2$ is to be chosen later.
Now we iterate the above procedure by taking $\lambda_i=(1+\alpha)^i,
i=1,2,\ldots,n$, at each step, and choose
{$\lambda_n=(1+\alpha)^n=e.$}
We conclude that there exists a large enough constant $C$ independent
of $n$ and $t$ such that
\[
\int_0^\infty\mp\bigl\{\bigl|X_{t}\bigl(
\varepsilon,\{s_{1},\ldots,s_{n}\} ,x\bigr)-x\bigr|>\log y\bigr
\}\,\mathrm{d}y\leq C^n \biggl(\frac{1}{\alpha} \biggr)^n.
\]
In what follows, we only need to show
{$C^n (1/\alpha )^n/n!\to0$ as $n\rightarrow\infty$.}
Recall that {$\alpha=\mathrm{e}^{1/n}-1.$}
We have
\begin{eqnarray*}
\log \biggl[ C^n \biggl(\frac{1}{\alpha} \biggr)^n
\biggr]\sim n \biggl(C+\log \frac{1}{n} \biggr)\qquad \mathrm{as}\ n\rightarrow
\infty.
\end{eqnarray*}
On the other hand, we know {$\log n!\sim n^2/2$ as $n\rightarrow\infty.$}
The proof is then complete.
\end{pf*}
\end{appendix}

\section*{Acknowledgements}
The authors are grateful to an anonymous referee and the editors for
their constructive comments that greatly helped to improve the paper.
J.~E. Figueroa-L\'opez research was partially supported by Grants from
the US National Science Foundation (DMS-09-06919, DMS-11-49692).




\printhistory

\end{document}